\input ./style/arxiv-general.cfg
\documentclass[aop,MSNbibl,nameyear,dvips]{arximspdf}
\makeatletter
   \@ifpackageloaded{graphicx}{}{\usepackage{graphicx}}
\makeatother

\doi{10.1214/14-AOP998}
\volume{44}
\issue{2}
\pubyear{2016}
\firstpage{1134}
\lastpage{1211}
\docsubty{FLA}

\makeatletter
\newcommand{\WWW}{\mathop{W}\limits^{\circ}{}}
\def\mid{|}
\newcommand{\rrvert}{\vert}
\newcommand{\llvert}{\vert}
\def\sphat{^{\wedge}}
\newcommand{\R}{{\mathbb R}}
\newcommand{\C}{{\mathbb C}}
\newcommand{\Z}{{\mathbb Z}}
\newcommand{\T}{{\mathbb T}}
\renewcommand{\d}{\partial}
\renewcommand{\Re}{\operatorname{Re}}
\renewcommand{\Im}{\operatorname{Im}}
\newcommand{\F}{{\mathcal F}}
\newcommand{\Prob}{\mathbb{P}}
\newcommand{\Ai}{{\operatorname{Ai}}}
\newcommand{\dn}{{\operatorname{dn}}}
\newcommand{\al}{\alpha}
\newcommand{\be}{\beta}
\newcommand{\Ga}{\Gamma}
\newcommand{\ep}{\varepsilon}
\newcommand{\De}{\Delta}
\newcommand{\Sg}{\Sigma}
\newcommand{\Om}{\Omega}
\renewcommand{\th}{\vartheta}

\newtheorem{thmm}{Theorem}[section]

\newtheorem{lem}[thmm]{Lemma}
\newtheorem{prop}[thmm]{Proposition}
\newproclaim{rmk}{Remark}[section]

\newcommand{\intZ}{\mathbb{Z}}
\newcommand{\bigO}{\mathcal{O}}
\newcommand{\compC}{\mathbb{C}}
\newcommand{\realR}{\mathbb{R}}
\newcommand{\unitC}{\mathbb{T}}
\newcommand{\Id}{\mathbf{1}}
\newcommand{\K}{\mathbf{K}}
\newcommand{\E}{\mathbf{E}}
\newcommand{\GF}{\mathbf{GF}}
\newcommand{\Pcal}{\mathcal{P}}
\newcommand{\Hankel}{\mathcal{H}}

\newcommand{\sgn}{\operatorname{sgn}}
\newcommand{\low}{\operatorname{lower}}
\newcommand{\upp}{\operatorname{upper}}
\newcommand{\Res}{\mathop{\operatorname{Res}}}
\newcommand{\major}{\operatorname{major}}
\newcommand{\minor}{\operatorname{minor}}
\newcommand{\local}{\operatorname{local}}
\newcommand{\Pearcey}{\operatorname{Pearcey}}
\newcommand{\tac}{\operatorname{tac}}
\newcommand{\Er}{\operatorname{Er}}
\newcommand{\eqref}[1]{(\ref{#1})}
\makeatother

\begin{document}
\begin{frontmatter}

\title{Nonintersecting Brownian motions on~the~unit~circle}
\runtitle{Nonintersecting Brownian motions}

%
\begin{aug}
\author[A]{\fnms{Karl}~\snm{Liechty}\ead[label=e1]{kliechty@depaul.edu}\ead[label=u1,url]{http://math.depaul.edu/kliechty/}}
\and
\author[B]{\fnms{Dong}~\snm{Wang}\corref{}\thanksref{T2}\ead[label=e2]{matwd@nus.edu.sg}\ead[label=u2,url]{http://www.math.nus.edu.sg/\textasciitilde matwd/}}
\runauthor{K. Liechty and D. Wang}
\affiliation{DePaul University and National University of Singapore}
\address[A]{Department of Mathematical Sciences\\
DePaul University\\
Chicago, Illinois 60614 \\
USA\\
\printead{e1}\\
\printead{u1}}

\address[B]{Department of Mathematics\\
National University of Singapore\\
Singapore, 119076 \\
\printead{e2} \\
\printead{u2}}
\thankstext{T2}{Supported in part by the startup Grant R-146-000-164-133.}
\end{aug}
%

%
\received{\smonth{7} \syear{2014}}
%
\revised{\smonth{12} \syear{2014}}

%
\begin{abstract}
We consider an ensemble of $n$ nonintersecting Brownian particles on
the unit circle with diffusion parameter $n^{-1/2}$, which are
conditioned to begin at the same point and to return to that point
after time $T$, but otherwise not to intersect. There is a critical
value of $T$ which separates the subcritical case, in which it is
vanishingly unlikely that the particles wrap around the circle, and the
supercritical case, in which particles may wrap around the circle. In
this paper, we show that in the subcritical and critical cases the
probability that the total winding number is zero is almost surely 1 as
$n\to\infty$, and in the supercritical case that the distribution of
the total winding number converges to the discrete normal distribution.
We also give a streamlined approach to identifying the Pearcey and
tacnode processes in scaling limits. The formula of the tacnode
correlation kernel is new and involves a solution to a Lax system for
the Painlev\'{e} II equation of size 2 $\times$ 2. The proofs are based
on the determinantal structure of the ensemble, asymptotic results for
the related system of discrete Gaussian orthogonal polynomials, and a
formulation of the correlation kernel in terms of a double contour integral.
\end{abstract}

%
\begin{keyword}[class=AMS]
\kwd[Primary ]{60J65}
\kwd[; secondary ]{35Q15}
\kwd{42C05}
\end{keyword}
\begin{keyword}
\kwd{Nonintersecting Brownian motions}
\kwd{determinantal process}
\kwd{discrete orthogonal polynomial}
\kwd{tacnode process}
\kwd{Pearcey process}
\kwd{Riemann--Hilbert problem}
\kwd{double contour integral formula}
\end{keyword}
%
\end{frontmatter}

\section{Introduction}\label{sec1}

The probability models of nonintersecting Brownian motions have been
studied extensively in last decade; see \citeauthor{Tracy-Widom04}
(\citeyear{Tracy-Widom04,Tracy-Widom06}),
\citet{Adler-vanMoerbeke05}, \citet{Adler-Orantin-vanMoerbeke10}, \citet{Delvaux-Kuijlaars-Zhang11}, \citet{Johansson13},
\citet{Ferrari-Veto12}, \citet{Katori-Tanemura07} and \citet
{Comtet-Majumdar-Schehr08}, for example. These models are closely
related to random matrix theory and (multiple) orthogonal polynomials;
see \citeauthor{Bleher-Kuijlaars04}
(\citeyear{Bleher-Kuijlaars04,Bleher-Kuijlaars07}), \citet{Aptekarev-Bleher-Kuijlaars05}
and \citet{Kuijlaars10}, for example. One
interesting feature is that as the number of particles $n \to\infty$,
under proper scaling the nonintersecting Brownian motions models
converge to \emph{universal} processes, like the sine, Airy, Pearcey
and tacnode processes. These processes are called universal since they
appear in many other probability problems; see
\citeauthor{Okounkov-Reshetikhin03}
(\citeyear{Okounkov-Reshetikhin03,Okounkov07}), \citet{Johansson05}, \citet{Baik-Suidan07}, \citet{Adler-vanMoerbeke-Wang11}, \citet
{Adler-Ferrari-vanMoerbeke13} and \citet
{Adler-Johansson-vanMoerbeke14}, for example. Usually the models of
nonintersecting Brownian motions turn out to be the most convenient
ones to use for study of these universal processes. In particular, the
Airy process appears ubiquitously in the Kardar--Parisi--Zhang (KPZ)
universality class [\citet{Corwin11}], an important class of interacting
particle systems and random growth models. The analysis of
nonintersecting Brownian motions greatly improves the understanding of
the Airy process and the KPZ universality class; see \citet
{Corwin-Hammond11}. Here, we remark that if we consider the
nonintersecting Brownian motions on the real line, in the simplest
models the Pearcey process does not occur, and the tacnode process only
occurs in models with sophisticated parameters. Thus, the analysis
of these universal processes becomes increasingly more difficult.

Due to technical difficulties, most studies of the limiting local
properties of the nonintersecting Brownian motions concern models
defined on the real line. A~model of nonintersecting Brownian motions
on a circle was considered by Dyson as a dynamical generalization of
random matrix models [\citet{Dyson62}], and physicists and probabilists
have been interested in the nonintersecting Brownian motions on a
circle and their discrete counterparts for various reasons; see \citet
{Forrester90}, \citet{Hobson-Werner96} and \citet{Cardy03}, for example.
The simplest model of nonintersecting Brownian motions on a circle such
that the particles start and end at the same common point is shown to
be related to Yang--Mills theory on the sphere [\citet
{Forrester-Majumdar-Schehr11}, \citet
{Comtet-Forrester-Majumdar-Schehr13}] and the partition function
(a.k.a. reunion probability) shows an interesting phase transition
phenomenon closely related to the Tracy--Widom distributions in random
matrix theory.

In this paper, we show that the Pearcey and (symmetric) tacnode
processes mentioned above occur as the limits of the simplest model of
nonintersecting Brownian motions on a circle, and give a streamlined
method to analyze them. We also consider the total winding number of
the particles, a quantity that has no counterpart in the models defined
on the real line, and show that its limiting distribution in the
nontrivial case is the discrete normal distribution [\citet
{Szablowski01}], a natural through perhaps not well-known discretization
of the normal distribution. We also show that in the supercritical
case, the Pearcey process occurs if the model is conditioned to have
fixed total winding number. Although the sine and Airy processes also
naturally occur, we omit the discussion on them to shorten the paper. A
detailed discussion can be found in the preprint [\citet{Liechty-Wang14-1}].

Technically, the study of nonintersecting Brownian motions has been
carried out in two distinct ways: by double contour integral formula,
and by the Riemann--Hilbert problem. In the present work, we introduce
a mixed approach, using both a double integral formula and the
interpolation problem for discrete Gaussian orthogonal polynomials
[\citet
{Liechty12}], which are discrete orthogonal polynomials analogous to
Hermite polynomials. In this paper, we analyze the dependence of the
discrete Gaussian orthogonal polynomials on the translation of the
lattice, which encodes the information of the winding number of the
Brownian paths.

\subsection{Statement of main results}

Let $\T= \{ e^{i\theta} \in\compC\}$ be the unit circle. Suppose
$x_1, x_2, \ldots, x_n$ are $n$ particles in independent Brownian
motions on the unit circle with continuous paths and diffusion
parameter $n^{-1/2}$, that is,
%
\begin{equation}
x_k(t) = e^{{iB_k(t)}/{\sqrt{n}}},\qquad i = 1, 2, \ldots, n,
\end{equation}
where $B_k(t)$ are independent Brownian motions with diffusion
parameter $1$ starting from arbitrary places. The \emph{nonintersecting
Brownian motions on the circle with $n$ particles}, henceforth denoted
as NIBM in this paper, is defined by the particles $x_1, \ldots, x_n$
conditioned to have nonintersecting paths, that is, $x_1(t), \ldots,
x_n(t)$ are distinct for any $t$ between the starting time and the
ending time. In this paper, we concentrate on the simplest model of
NIBM, such that the $n$ particles start from the common point $e^{i
\cdot0}$ at the starting time $t = 0$, and end at the same common
point $e^{i \cdot0}$ at the ending time $t = T$. We denote this model
as $\mathrm{NIBM}_{0 \to T}$.

Throughout this paper, we represent a point in $\T$ by an angular
variable $\theta\in\R$ with $\theta=\theta+2\pi k$ ($k\in\Z$) if there
is no possibility of confusion, and use $\theta\in[-\pi, \pi)$ as the
principal value of the angle. Let $P(a; b; t)$ be the transition
probability density of one particle in Brownian motion on $\T$ with
diffusion parameter $n^{-1/2}$, starting from point $a\in\T$ and
ending at point $b\in\T$ after time $t>0$, which is
%
\begin{equation}
\label{eq:def_oneD_trans_prob} P(a; b; t) =\sqrt{\frac{n}{2\pi t}}\sum
_{k\in\Z} e^{-
{n(b-a+2\pi
k)^2}/{(2t)}}.
\end{equation}
Now consider the transition probability density of NIBM. Let $A_n=\{
a_1, \ldots,\break a_n\}$ and $B_n=\{b_1, \ldots, b_n\}$ be two sets of $n$
distinct points in $\T$ such that $-\pi\leq a_1 < a_2 < \cdots< a_n <
\pi$ and $-\pi\leq b_1 < b_2 < \cdots< b_n < \pi$, and denote by
$P(A_n; B_n; t)$ the transition probability density of NIBM with the
particles starting at the points $A_n$ and ending at the points $B_n$
after time $t$. Note that we do not require that the particle which
started at point $a_k$ ends at point $b_k$, but only that it ends at
point $b_j$ for some $j=1,\ldots, n$. For $\tau\in\R$, introduce the notation
%
\begin{equation}
\label{eq:tau_def_one_particle} P(a; b; t; \tau):= \sqrt{\frac{n}{2\pi t}}\sum
_{k \in\Z} e^{-
{n(b-a+2\pi k)^2}/{(2t)}}e^{2k\pi\tau i},
\end{equation}
which reduces to \eqref{eq:def_oneD_trans_prob} when $\tau=0$.
Introduce also the notation
%
\begin{equation}
\label{eq:def_hsgn} \epsilon(n) = %
\cases{ 0, &\quad $\mbox{if $n$ is odd},$
\vspace*{2pt}
\cr
\frac{1}{2} ,&\quad  $\mbox{if $n$ is even}.$} %
\end{equation}
A determinantal formula for $P(A_n; B_n; t)$ is then given in the
following proposition.

\begin{prop}\label{prop:trans_prob_determinant}
The transition probability density function $P(A_n;\break B_n; t)$ is given
by the determinant of size $n\times n$,
%
\begin{equation}
P(A_n; B_n; t)= \det \bigl(P \bigl(a_i;
b_j; t; \epsilon(n) \bigr) \bigr)_{i,j=1}^n.
\end{equation}
\end{prop}

This proposition follows immediately from the Karlin--McGregor formula
in the case that $n$ is odd. If $n$ is even then more care must be
taken to derive the formula, and in the limited knowledge of the
current authors
it has not appeared before in the literature. The proof
is presented in Section~\ref{tau-deformed transition probability
density of NIBM}.

Now we consider the model $\mathrm{NIBM}_{0 \to T}$. At a given time
$t\in[0,T]$, the
joint probability density function for the $n$ particles in $\mathrm
{NIBM}_{0 \to T}$ at
distinct points $-\pi\leq\theta_1 < \theta_2 < \cdots< \theta_n <
\pi
$ is given by
%
\begin{equation}
\label{eq:jpdf1} \mathop{\lim_{a_1, \ldots, a_n \to0 }}_{ b_1,\ldots, b_n \to0}
\frac
{P(A_n; \Theta_n; t)P(\Theta_n; B_n; T-t)}{P(A_n; B_n; T)} ,
\end{equation}
where $A_n=\{a_1, \ldots, a_n\}$, $B_n=\{b_1, \ldots, b_n\}$, and
$\Theta
_n = \{ \theta_1, \ldots, \theta_n \}$ describe the locations of the
$n$ particles at time $0$, $T$ and $t$, respectively. It is not
difficult to see that such a limit exists, and so that our model is
well defined (see Section~\ref{reunion_probability}).

The model $\mathrm{NIBM}_{0 \to T}$ is a determinantal process, meaning
that the
correlation functions of the particles may be described by a
determinantal formula [\citet{Soshnikov00}]. To define the determinantal
structure, fix $m$ times $0< t_1 < t_2 < \cdots<t_m<T$, and to each
time $t_i$, fix $k_i$ points on $\T$, $-\pi\leq\theta_1^{(i)}
<\theta
_2^{(i)}<\cdots<\theta_{k_i}^{(i)} < \pi$. The multi-time correlation
function is then defined as
%
\begin{eqnarray}
\label{eq:corr_function_defn_all_winding_number}
&& R^{(n)}_{0 \to T}\bigl(\theta^{(1)}_1,
\ldots, \theta^{(1)}_{k_1}; \ldots ; \theta
^{(m)}_1, \ldots, \theta^{(m)}_{k_m};
t_1, \ldots, t_m\bigr)\nonumber
\\
&&\qquad:=\lim_{\Delta x \to0} \frac{1}{(\Delta x)^{k_1 + \cdots+ k_m}} \Prob
\bigl( \mbox{there is a
particle in $\bigl[\theta^{(i)}_j, \theta ^{(i)}_j
+ \Delta x\bigr)$}\\
&&\hspace*{183pt} \mbox{for $j = 1, \ldots, k_i$ at time
$t_i$} \bigr).
\nonumber
\end{eqnarray}
Then there exists some kernel function $K_{t_i, t_j}(x,y)$ such that
%
\begin{eqnarray}
\label{eq:defn_Rmelon_special} &&R^{(n)}_{0 \to T}\bigl(\theta^{(1)}_1,
\ldots, \theta^{(1)}_{k_1}; \ldots ; \theta
^{(m)}_1, \ldots, \theta^{(m)}_{k_m};
t_1, \ldots, t_m\bigr)
\nonumber
\\[-8pt]
\\[-8pt]
\nonumber
&&\qquad= \det \bigl( K_{t_i, t_j}
\bigl(\theta^{(i)}_{l_i}, \theta^{(j)}_{l'_j}
\bigr) \bigr)_{i, j = 1, \ldots, m,
 l_i = 1, \ldots, k_i,  l'_j = 1,
\ldots, k_j};
\end{eqnarray}
see Section~\ref{subsec:tau_deformed_corr_functions}.

Intuitively, one can imagine the scenario of the model $\mathrm
{NIBM}_{0 \to T}$ as
follows. When the total time $T$ is small, it is very unlikely that the
particles will wrap around the circle before returning to $e^{i\cdot
0}$, and so the model is very close to the model of nonintersecting
Brownian bridges on the real line. For large $T$, the particles which
initially move in the positive direction and those which initially move
in the negative direction will eventually meet on the far side of the
circle, and the behavior of the model is very different. In this paper,
this heuristic argument is confirmed, and the critical value of $T$
which separates these two cases is pinpointed to be
%
\begin{equation}
\label{eq:first_defn_T_c} T_c=\pi^2.
\end{equation}
Accordingly, we divide the $\mathrm{NIBM}_{0 \to T}$ model into the
subcritical, critical
and supercritical cases, for $T < \pi^2$, $T = \pi^2$, and $T > \pi^2$,
respectively, as shown in Figure~\ref{fig:Global_picture}.

In the subcritical case $T<T_c$, the model is described asymptotically
by elementary functions. In the critical case $T=T_c$ and the
supercritical case $T>T_c$, the model is described asymptotically by
special functions: functions related to the Painlev\'{e} II equation
for $T=T_c$, and elliptic integrals for $T>T_c$. Let us define those functions.

\begin{figure}

\includegraphics{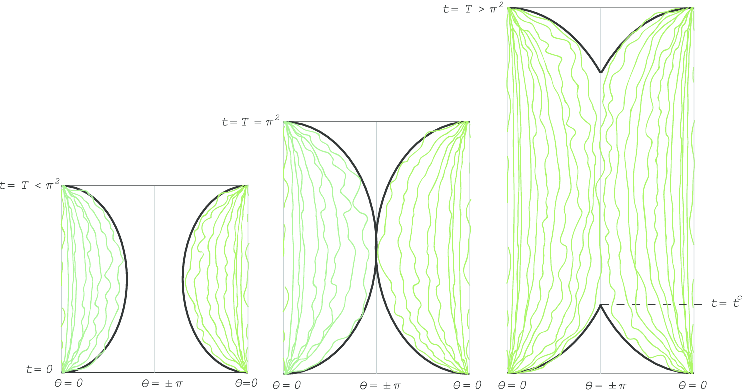}

\caption{Typical configurations of nonintersecting paths in the
subcritical (left), critical (middle) and supercritical (right) cases.
Time is on the vertical axis, and the angular variable $\theta$ on the
horizontal axis. At the initial time $t=0$ and the terminal time $t=T$,
the particles are at $\theta=0$, which is at both the left and right
ends of the figures. The far side of the circle, $\theta=\pm\pi$, is
marked by a light vertical line through the center of the figures. The
particles tend to stay within the thick curved lines. In the
supercritical case, the critical time $t^c$ is marked, when the
``leftmost'' and ``rightmost'' particles meet on the far side of the circle.}
\label{fig:Global_picture}
\end{figure}

\textit{Critical case: The Painlev\'{e} II equation, and the related
Lax pair}.
In the critical case, we consider the model $\mathrm{NIBM}_{0 \to T}$
in the scaling limit
%
\begin{equation}
\label{eq:T_scaling} T=\pi^2 \bigl(1-2^{-2/3}\sigma
n^{-2/3} \bigr),
\end{equation}
where $\sigma\in\R$ is a parameter.
In this case, the results of this paper involve a particular solution
to the Painlev\'e II equation, and a solution to a related Lax system.
Let us review these objects. The Hastings--McLeod solution [\citet
{Hastings-McLeod80}] to the homogeneous Painlev\'{e} II equation (PII) is
the solution to the differential equation
%
\begin{equation}
\label{cr2} q''(s)=sq(s)+2q(s)^3,
\end{equation}
which satisfies
%
\begin{equation}
\label{cr3} q(s)=\Ai(s) \bigl(1+o(1)\bigr)\qquad \mbox{as } s\to+\infty,
\end{equation}
where $\Ai(s)$ is the Airy function.
Let $q(s)$ be this particular solution to PII, and consider the $2
\times2$ matrix-valued solutions to the differential equation
%
\begin{equation}\qquad
\label{cr4} \frac{d}{d\zeta} \bolds\Psi(\zeta;s)= %
\pmatrix{ -4i
\zeta^2-i\bigl(s+2q(s)^2\bigr) & 4\zeta
q(s)+2iq'(s) \vspace *{2pt}
\cr
4\zeta q(s)-2iq'(s) &
4i\zeta^2 +i\bigl(s+2q(s)^2\bigr) } %
\bolds
\Psi(\zeta;s).
\end{equation}
This $2 \times2$ system was originally studied by \citet
{Flaschka-Newell80}. The differential equation \eqref{cr4}, together
with another one given in \eqref{equiv2}, form a Lax pair for the PII
equation, that is, the compatibility of the two differential equations
implies that $q(s)$ solves PII.
We will consider the particular solution to \eqref{cr4} which satisfies
%
\begin{equation}
\label{cr5} \bolds\Psi(\zeta; s)e^{i(({4}/{3}) \zeta^3 +s \zeta)\sigma_3} = I+O\bigl(
\zeta^{-1}\bigr),\qquad \zeta\to\pm\infty.
\end{equation}
The asymptotics \eqref{cr5} extend into the sectors $-\pi/3< \arg
\zeta<
\pi/3$, and $2\pi/3 < \arg\zeta< 4\pi/3$. Here, we note that the
uniqueness of the boundary value problem~\eqref{cr4} and \eqref{cr5} implies
%
\begin{equation}
\label{eq:symmetry_of_H-ML_solution} \Psi_{i, j}(-\zeta) =
\Psi_{3 - i, 3 - j}(\zeta),\qquad i, j
= 1, 2.
\end{equation}

\textit{Supercritical case: Elliptic integrals}.
In the supercritical case where $T > T_c = \pi^2$, we define a $t^c <
T/2$. To simplify the notation, we parametrize $T > \pi^2$ by $k \in
(0, 1)$. For each $k$, we have the elliptic integrals
%
\begin{eqnarray}
\label{eq:complete_elliptic} \K&:=& \K(k) = \int^1_0
\frac{ds}{\sqrt{(1 - s^2)(1 - k^2 s^2)}},
\nonumber
\\[-8pt]
\\[-8pt]
\nonumber
 \E&:= &\E(k) = \int^1_0
\frac{\sqrt{1 - k^2 s^2}}{\sqrt{1 - s^2}} \,ds.
\end{eqnarray}
We further define
%
\begin{equation}
\label{eq:def_tilde1} \tilde{k} := \frac{2\sqrt{k}}{1 + k},
\end{equation}
and denote
%
\begin{eqnarray}
\label{eq:def_tilde2} \tilde{\K} &:=& \K(\tilde{k}) = \int^1_0
\frac{ds}{\sqrt{(1 -
s^2)(1 -
\tilde{k}^2 s^2)}},
\nonumber
\\[-8pt]
\\[-8pt]
\nonumber
\tilde{\E} &:=& \E(\tilde{k}) = \int^1_0
\frac
{\sqrt{1 - \tilde{k}^2 s^2}}{\sqrt{1 - s^2}} \,ds.
\end{eqnarray}
$T$ is then parametrized as
%
\begin{equation}
\label{eq:parametrization_of_T_intro} T = 4\tilde{\K} \tilde{\E} = 4 \int^1_0
\frac{ds}{\sqrt{(1 -
s^2)(1 -
\tilde{k}^2 s^2)}} \int^1_0 \frac{\sqrt{1 - \tilde{k}^2
s^2}}{\sqrt{1 -
s^2}}
\,ds,
\end{equation}
where the well-definedness of the parametrization is given in Lemma~\ref
{lem:T_parametrized_by_k}, and $t^c$ is expressed as
%
\begin{equation}
\label{eq:T_c_in_introduction} t^c = \frac{4}{\tilde{k}^2} \tilde{\E} \bigl(\tilde{\E} -
\bigl(1 - \tilde {k}^2\bigr) \tilde{\K} \bigr) = 4 \int
^1_0 \frac{\sqrt{1 - \tilde{k}^2
s^2}}{\sqrt{1 - s^2}} \,ds \int
^1_0 \frac{\sqrt{1 - s^2}}{\sqrt{1 -
\tilde
{k}^2 s^2}} \,ds.
\end{equation}

The fundamental group of $\T$ has a canonical identification with $\Z$,
and so for any closed path on $\T$ we can define the \emph{winding
number} of the path as the integer representative of its homotopy
class. For a set of $n$ particles with continuous paths on $\T$ that
come back the initial position after some time, we can define their
\emph{total winding number} as the sum of the winding numbers of the
paths of the particles. The following theorem concerns the total
winding number of the particles in $\mathrm{NIBM}_{0 \to T}$. Let $q$ be
defined in terms
of the complete elliptic integral of the first kind as
%
\begin{equation}
\label{eq:def_elliptic_nome_intro} q := \exp \biggl(-\frac{\pi\mathbf K(\sqrt{1-k^2})}{2\mathbf
K(k)} \biggr) = \exp \biggl(-
\frac{\pi\mathbf K(\sqrt{1 - \tilde
{k}^2})}{\mathbf
K(\tilde{k})} \biggr),
\end{equation}
where $k$ and $\tilde k$ are related to $T$ via \eqref
{eq:complete_elliptic}--\eqref{eq:parametrization_of_T_intro}.

\begin{rmk} Note that we use the notation $q$ in two different
meanings. In the context of the critical asymptotics, $q$ is the
Hastings--McLeod solution to PII and is always written with its
argument $q(\sigma)$. In the context of the supercritical asymptotics, $q$
is written with no argument and represents the elliptic nome defined in
\eqref{eq:def_elliptic_nome_intro}. These are both standard notation,
and it should be clear throughout the paper to which object $q$ refers.
\end{rmk}

\begin{thmm}\label{winding_number_theorem}
In the $\mathrm{NIBM}_{0 \to T}$, as the number of particles $n \to
\infty$:
\begin{longlist}[(a)]
\item[(a)]
In the subcritical case $T<T_c=\pi^2$, the winding number is zero with
a probability that is exponentially close to 1. That is,
%
\begin{equation}
\label{wnsub} \mathbb{P}(\mbox{Total winding number equals } 0)= 1-\bigO
\bigl(e^{-cn}\bigr),
\end{equation}
where the constant $c>0$ may depend on $T$.
\item[(b)]
In the critical scaling \eqref{eq:T_scaling}, for any fixed $\sigma$,
%
\begin{eqnarray}
\label{wncrit} %
\mathbb{P}(\mbox{Total winding number
equals } 0)&=& 1-\frac{q(\sigma
)}{2^{1/3}n^{1/3}}+\frac{q(\sigma)^2}{2^{2/3}n^{2/3}}+\bigO\bigl(n^{-1}
\bigr) ,
\nonumber\\
\mathbb{P}(\mbox{Total winding number equals } 1 )&=& \mathbb{P}\bigl(\mbox{Total
winding number equals } (-1) \bigr)
\nonumber
\\[-4pt]
\\[-12pt]
\nonumber
&=& \frac{q(\sigma)}{2^{4/3}n^{1/3}}-\frac{q(\sigma
)^2}{2^{5/3}n^{2/3}}+\bigO \bigl(n^{-1}\bigr) ,
\\
\mathbb{P}\bigl(|\mbox{Total winding number}|> 1 \bigr)&=&\bigO\bigl(n^{-1}
\bigr).\hspace*{-20pt}\nonumber
\end{eqnarray}
\item[(c)]
For $T>T_c$ and for any $\omega\in\Z$,
%
\begin{equation}
\label{wnsuper} \mathbb{P}(\mbox{Total winding number equals } \omega)=
q^{\omega
^2} \sqrt {\frac{\pi}{2\tilde{\K}}} + \bigO\bigl(n^{-1}\bigr).
\end{equation}
\end{longlist}
\end{thmm}

The limiting distribution of the total winding number in the
supercritical case is the \emph{discrete normal distribution} defined
in \citet{Kemp97}, and the formula in the right-hand side of \eqref
{wnsuper} appears in \citet{Szablowski01}. See also \citet{Johnson-Kemp-Kotz05},
Section~10.8.3.

The Pearcey process is defined by the extended Pearcey kernel
[\citet{Tracy-Widom06},
Section~3],
%
\begin{equation}
\label{eq:Pearcey_kernel} K^{\Pearcey}_{s, t}(\xi, \eta) =
\tilde{K}^{\Pearcey}_{s, t}(\xi, \eta ) - \phi_{s, t}(\xi,
\eta),
\end{equation}
where
%
\begin{equation}
\label{eq:nonessential_Pearcey} \phi_{s, t}(\xi, \eta) = %
\cases{0, &\quad $\mbox{if } s \geq t$, \vspace*{2pt}
\cr
\displaystyle\frac{1}{\sqrt{2\pi(t - s)}}
e^{-{(\xi- \eta)^2}/{(2(t - s))}},
 &\quad $
\mbox{if $s < t$}$, } %
\end{equation}
and
%
\begin{equation}
\label{eq:int_formula_Pearcey} \tilde{K}^{\Pearcey}_{s, t}(\xi, \eta) =
\frac{i}{4 \pi^2} \oint _{\Sigma_P} \,dz \oint_{\Gamma_P} \,dw
\frac{e^{{z^4}/{4} + {s
z^2}/{2} + i\xi z}}{e^{{w^4}/{4} + {t w^2}/{2} + i\eta w}} \frac
{1}{z - w},
\end{equation}
where $\Sigma_P$ and $\Gamma_P$ are infinite, disjoint contours such
that the upper part of $\Sigma_P$ is from $e^{\pi i/4} \cdot\infty$ to
$e^{3\pi i/4} \cdot\infty$, the lower part of $\Sigma_P$ is from
$e^{5\pi i/4} \cdot\infty$ to $e^{7\pi i/4} \cdot\infty$, and
$\Gamma
_P$ is the leftward horizontal line. See Figure~\ref{fig:Sigma_P_Gamma_P} for the exact description. Our definition of the
Pearcey kernel is the same as that in \citet{Adler-Orantin-vanMoerbeke10},
Formula
1.2, up to a change of variables.

\begin{figure}

\includegraphics{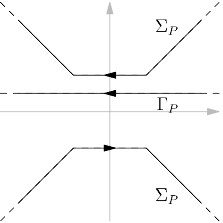}

\caption{The shape of contours $\Sigma_P$ and $\Gamma_P$. The upper
part of $\Sigma_P$ consists of the ray from $2 + 2i$ to $e^{\pi i/4}
\cdot\infty$, the line segment from to $2 + 2i$ to $-2 + 2i$, and the
ray from $-2 + 2i$ to $e^{3\pi i/4} \cdot\infty$. The lower part of
$\Sigma_P$ is the reflection of the upper part about the real axis.
$\Gamma_P$ is the horizontal line $\{ z = x + i \mid x \in\realR\}$.
Their orientations are shown in the figure.}
\label{fig:Sigma_P_Gamma_P}
\end{figure}

We now define the tacnode kernel. Denote by $\Psi_{ij}(\zeta;s)$ the
$(i,j)$ entry of the matrix $\bolds\Psi(\zeta;s)$ defined in \eqref{cr4}
and \eqref{cr5}.
It is convenient to also define the functions
%
\begin{eqnarray}
\label{tac18} f(u; s) &:=& %
\cases{ -\Psi_{12}(u;s), &\quad $
\mbox{if $\Im u>0$}$, \vspace*{2pt}
\cr
\Psi_{11}(u;s), &\quad $\mbox{if $\Im
u<0$}$, }
\nonumber
\\[-8pt]
\\[-8pt]
\nonumber
g(u,s) &:= &%
\cases{ -\Psi_{22}(u;s), & \quad $
\mbox{if $\Im u>0$},$ \vspace*{2pt}
\cr
\Psi_{21}(u;s), &\quad $\mbox{if $\Im
u<0$}$. } %
\end{eqnarray}
We then define the tacnode kernel as
%
\begin{equation}
\label{eq:tacnode_kernel} K^{\tac}_{s, t}(\xi, \eta; \sigma) =
\tilde{K}^{\tac}_{s, t}(\xi , \eta; \sigma) -
\phi_{s, t}(\xi, \eta),
\end{equation}
where $\phi_{s, t}(\xi, \eta)$ is as in \eqref
{eq:nonessential_Pearcey}, and
\begin{eqnarray}
\label{eq:essential_tacnode} \tilde{K}^{\tac}_{s, t}(\xi, \eta; \sigma)&:=&
\frac{1}{2\pi} \oint _{\Sg_T} \,du \oint_{\Sg_T} \,dv
e^{{s u^2}/{2} - {t v^2}/{2}} e^{-i(u\xi
- v\eta)}
\nonumber
\\[-8pt]
\\[-8pt]
\nonumber
&&{}\times\frac{f(u;\sigma)g(v;\sigma)-g(u;\sigma)f(v;\sigma
)}{2\pi i (u-v)}.
\end{eqnarray}
Here, $\Sg_T$ is a contour consisting of two pieces. One piece of $\Sg
_T$ lies entirely above the real line, and goes from $e^{\pi i/6} \cdot
\infty$ to $e^{5\pi i/6} \cdot\infty$. The other piece lies entirely
below the real line and goes from $e^{7\pi/6} \cdot\infty$ to
$e^{11\pi
/6} \cdot\infty$. See Figure~\ref{fig:Sigma_T} for the exact
description. The convergence of the integrals in \eqref
{eq:essential_tacnode} follows from the asymptotics~\eqref{cr5}. Let us
note that we could deform the two parts of the contour $\Sg_T$ to the
real line, and write \eqref{eq:essential_tacnode} as the sum of four
double integrals on $\R$. We prefer to write the integral on the
contour $\Sg_T$ because the integrand of \eqref{eq:essential_tacnode}
is in fact an $L^1$ function on $\Sg_T$, whereas convergence of the
integral over $\R$ is the result of rapid oscillations.

\begin{figure}

\includegraphics{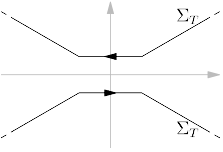}

\caption{The shape of contour $\Sg_T$. The upper part of $\Sigma_T$
consists of the ray from $\sqrt{3} + i$ to $e^{\pi i/6} \cdot\infty$,
the line segment from $\sqrt{3} + i$ to $-\sqrt{3} + i$, and the ray
from $-\sqrt{3} + i$ to $e^{5\pi i/6} \cdot\infty$. The lower part of
$\Sigma_T$ is the reflection of the upper part about the real axis. The
orientation is shown in the figure.}
\label{fig:Sigma_T}
\end{figure}

The convergence of $\mathrm{NIBM}_{0 \to T}$ to the universal processes
described above
is described in the following theorem.

\begin{thmm} \label{thmm:corr_kernel_intro}
In the $\mathrm{NIBM}_{0 \to T}$:
\begin{longlist}[(a)]
\item[(a)]
Assume $T>T_c$. There exists $d>0$ defined in \eqref{eq:Taylor_exp_upp_Pearcey}
such that when we scale $t_i$ and $t_j$ close to $t^c$, and $x$ and $y$
close to $-\pi$ as
%
\begin{eqnarray}
\label{eq:coefficients_Pearcey} t_i &=& t^c + \frac{d^2}{n^{{1}/{2}}}
\tau_i,\qquad t_j = t^c + \frac
{d^2}{n^{{1}/{2}}}
\tau_j,
\nonumber
\\[-8pt]
\\[-8pt]
\nonumber
 x& =& -\pi- \frac{d}{n^{
{3}/{4}}} \xi,\qquad y = -\pi-
\frac{d}{n^{{3}/{4}}} \eta,
\end{eqnarray}
the correlation kernel $K_{t_i, t_j}(x,y)$ has the limit
%
\begin{equation}
\lim_{n\to\infty} K_{t_i, t_j}(x,y) \biggl\llvert
\frac{dy}{d\eta} \biggr\rrvert = K^{\Pearcey}_{-\tau_j,-\tau_i}(\eta, \xi).
\end{equation}
\item[(b)]
Let $T$ be scaled close to $T_c=\pi^2$ as in \eqref{eq:T_scaling} with
$\sigma$ fixed, and let
%
\begin{equation}
\label{eq:defn_d_intro} d = 2^{-{5}/{3}}\pi.
\end{equation}
When we scale $t_i$ and $t_j$ close to $T/2$, and $x$ and $y$ close to
$-\pi$ as
%
\begin{eqnarray}
\label{tac6} t_i &=& \frac{T}{2} + \frac{d^2}{n^{{1}/{3}}}
\tau_i,\qquad t_j = \frac{T}{2} + \frac{d^2}{n^{{1}/{3}}}
\tau_j,
\nonumber
\\[-8pt]
\\[-8pt]
\nonumber
 x &=& -\pi- \frac
{d}{n^{{2}/{3}}} \xi,\qquad  y = -\pi-
\frac{d}{n^{{2}/{3}}} \eta,
\end{eqnarray}
the correlation kernel $K_{t_i, t_j}(x,y)$ has the limit
%
\begin{equation}
\label{tac17} \lim_{n\to\infty} K_{t_i, t_j}(x,y) \biggl\llvert
\frac{dy}{d\eta} \biggr\rrvert = K^{\tac}_{\tau_i, \tau_j}(\xi,\eta;
\sigma) = K^{\tac
}_{-\tau_j,
-\tau_i}(\eta, \xi; \sigma).
\end{equation}
\end{longlist}
\end{thmm}

\begin{rmk}
The identity $K^{\tac}_{\tau_i, \tau_j}(\xi,\eta; \sigma) =
K^{\tac
}_{-\tau_j, -\tau_i}(\eta, \xi; \sigma)$ in \eqref{tac17} is due
to the
symmetry of the kernel $K^{\tac}_{s, t}(\xi, \eta)$, which can be
checked by \eqref{eq:symmetry_of_H-ML_solution}.
\end{rmk}

In the supercritical case, we have finer result for the $\mathrm
{NIBM}_{0 \to T}$
conditioned to have fixed total winding number. Analogous to \eqref
{eq:corr_function_defn_all_winding_number}, we define the multi-time
correlation function for the $\mathrm{NIBM}_{0 \to T}$ with total
winding number $\omega
$ as
\begin{eqnarray}
&&\bigl(R^{(n)}_{0 \to T}\bigr)_{\omega}
\bigl(a^{(1)}_1, \ldots, a^{(1)}_{k_1};
\ldots; a^{(m)}_1, \ldots, a^{(m)}_{k_m};
t_1, \ldots, t_m\bigr)
\nonumber
\\[-8pt]
\\[-8pt]
\nonumber
&&\qquad:=\lim_{\Delta x \to0} \frac{1}{(\Delta x)^{k_1 + \cdots+ k_m}}
\Prob \left(
\begin{array}{c}
\mbox{there is a particle in $\bigl[a^{(i)}_j,
a^{(i)}_j + \Delta x\bigr)$}\\
\mbox{for $j = 1, \ldots,
k_i$ at time $t_i$,}\\
\mbox{and the total
winding number is $\omega$} %
\end{array}
 \right).\hspace*{-10pt}
\nonumber
\end{eqnarray}
If we consider the conditional $\mathrm{NIBM}_{0 \to T}$ such that the
total winding
number is fixed to be $\omega$, then the multi-time correlation function
of the conditional process should be
\begin{eqnarray}
&&\bigl(R^{(n)}_{0 \to T}\bigr)^{\sim}_{\omega}
\bigl(a^{(1)}_1, \ldots, a^{(1)}_{k_1};
\ldots; a^{(m)}_1, \ldots, a^{(m)}_{k_m};
t_1, \ldots, t_m\bigr)
\nonumber
\\[-8pt]
\\[-8pt]
\nonumber
&&\qquad :=\frac{(R^{(n)}_{0 \to T})_{\omega}(a^{(1)}_1, \ldots, a^{(1)}_{k_1};
\ldots;
a^{(m)}_1, \ldots, a^{(m)}_{k_m}; t_1, \ldots, t_m)}{\mathbb
{P}(\mathrm
{Total\ winding\ number\ equals\ } \omega)}.
\nonumber
\end{eqnarray}
Note that if the total winding number is fixed, then the conditional
$\mathrm{NIBM}_{0 \to T}$ is no longer a determinantal process. (The
reason is as
follows: In a determinantal process over time $[0, T]$, the movement of
particles between two times $t_1 < t_2 \in(0, T)$ only depends on the
positions of the particles at times $t_1$ and $t_2$, but not the
trajectories on $(0, t_1)$ or $(t_2, T)$. The conditional $\mathrm
{NIBM}_{0 \to T}$ with
fixed total winding number does not have this property.) Nevertheless,
we have results for the limiting $k$-correlation functions of the
process. The following theorem shows that with the condition of fixed
total winding number, the conditional $\mathrm{NIBM}_{0 \to T}$ has the
same local
limiting properties as the $\mathrm{NIBM}_{0 \to T}$ with free winding number.

\begin{thmm} \label{thmm:fixed_winding_number_correlations}
Assume $T>T_c=\pi^2$. Let $\omega$ be a fixed integer, $t_1, \ldots,
t_m \in(0, T)$ be times, and at each time $t_i$, let $x^{(i)}_1,
\ldots
, x^{(i)}_{k_i}$ be locations on $\T$ such that $k_1 + \cdots+ k_m =
k$. We consider the correlation function $(R^{(n)}_{0 \to T})^{\sim
}_{\omega} =
(R^{(n)}_{0 \to T})^{\sim}_{\omega}(x^{(1)}_1, \ldots,
x^{(1)}_{k_1}; \ldots;
x^{(m)}_1, \ldots, x^{(m)}_{k_m}; t_1, \ldots, t_m)$ in the conditional\break
$\mathrm{NIBM}_{0 \to T}$ with winding number $\omega$. Let
%
\begin{equation}
t_i = t^c + \frac{d^2}{n^{{1}/{2}}} \tau_i,\qquad
x^{(i)}_j = -\pi - \frac{d}{n^{{3}/{4}}} \xi^{(i)}_j,
\end{equation}
where $d$ is the same as in Theorem~\ref{thmm:corr_kernel_intro}\textup{(a)}.
The multi-time correlation function has
the limit
%
\begin{eqnarray}
\label{eq:corr_func_k_intro:cusp}&& \lim_{n\to\infty} \bigl(R^{(n)}_{0 \to T}
\bigr)^{\sim}_{\omega} \biggl( \frac
{d}{n^{{3}/{4}}}
\biggr)^k
\nonumber
\\[-8pt]
\\[-8pt]
\nonumber
&&\qquad = \det \bigl( K^{\Pearcey}_{-\tau_j, -\tau
_i} \bigl(
\xi^{(j)}_{l_j}, \xi^{(i)}_{l'_i} \bigr)
\bigr)_{{i, j = 1,
\ldots, m , l_i = 1, \ldots, k_i , l'_j = 1, \ldots, k_j}}.
\end{eqnarray}
\end{thmm}

\subsection{Comparison of $K^{\tac}$ with other tacnode kernels}
\label
{subsec:comparison_of_kernel_formulas}

The tacnode process was first studied by three different groups [\citet
{Adler-Ferrari-vanMoerbeke13}, \citet{Johansson13}, \citet
{Delvaux-Kuijlaars-Zhang11}], each using different methods and obtaining
different formulas for the tacnode process. The formulas obtained in
\citet{Adler-Ferrari-vanMoerbeke13} and \citet{Johansson13} each involve
Airy functions and related operators, whereas the formula of \citet
{Delvaux-Kuijlaars-Zhang11} involves a Lax system for the Painlev\'{e}
II equation of size $4 \times4$. As it turns out, the various matrix
entries of the $4\times4 $ Lax system appearing in \citet
{Delvaux-Kuijlaars-Zhang11} can be explicitly expressed in terms of
Airy functions and related operators [\citet{Delvaux13}] [see also \citet
{Kuijlaars14}], and the equivalence of the formulas in \citet
{Johansson13} and \citet{Delvaux-Kuijlaars-Zhang11} was recently proven
by Delvaux [\citet{Delvaux13}]. The equivalence of the two different Airy
formulas obtained in \citet{Johansson13} and \citet
{Adler-Ferrari-vanMoerbeke13} was proved in \citet
{Adler-Johansson-vanMoerbeke14}, although the proof is somewhat
indirect in that it relies on computing the limiting kernel from a
particular model in two different ways.

Indeed the formula for the tacnode kernel obtained in the $\mathrm
{NIBM}_{0 \to T}$ is
equivalent to the existing formulas. In order to state this equivalence
precisely, we define the kernel $\mathcal{L}^{\tac}$ obtained in
\citet
{Johansson13}, using some notation which was introduced in \citet
{Delvaux13} and \citet{Baik-Liechty-Schehr12}.
Let $\mathbf B_s$ be the integral operator defined in \citet{Baik-Liechty-Schehr12},
Formula
(3), which is denoted as $A_{\sigma}$ in \citet{Delvaux13},
Formula (4.1), acting on $L^2[0,\infty)$ with kernel
%
\begin{equation}
\mathbf B_s(x,y)=\Ai(x+y+s),
\end{equation}
and let $\mathbf A_s:=\mathbf B_s^2$ be the Airy operator, which is
defined in \citet{Baik-Liechty-Schehr12}, Formula (17) and is denoted as
$K_{\Ai, \sigma}$ in \citet{Delvaux13}, Formula (4.2). Define the
functions $Q_s$ and $R_s$ as in \citet{Baik-Liechty-Schehr12}, Formula~(18)
%
\begin{equation}
\label{equiv13} Q_s:=(\mathbf1-\mathbf A_s)^{-1}
\mathbf B_s \delta_0 ,\qquad R_s:=(\mathbf1-
\mathbf A_s)^{-1} \mathbf A_s
\delta_0 ,
\end{equation}
where the delta function $\delta_0$ is defined such that
%
\begin{equation}
\int_{[0,\infty)} f(x) \delta_0(x) \,dx = f(0),
\end{equation}
for functions $f(x)$ which are right-continuous at zero. Define also
the function
%
\begin{equation}
\label{equiv16} b_{\tau, z, \sigma}(x):=e^{-({2}/{3})\tau^3-\tau z - 2^{1/3} \tau
x -
2^{-2/3} \tau\sigma}\Ai\bigl(2^{1/3} x
+ z + 2^{-2/3} \sigma+ \tau^2\bigr) ,\hspace*{-10pt}
\end{equation}
which was introduced in \citet{Delvaux13}, Formula (2.16). Note that our
$b_{\tau, z, \sigma}(x)$ is equivalent to $b_{\tau, z}(x) = \tilde
{b}_{\tau, -z}(x)$ in \citet{Delvaux13}, Formula (2.16) with $\lambda
= 1$.
Then the symmetric tacnode kernel obtained in \citet{Johansson13} is
given by
%
\begin{equation}
\label{equiv20} \mathcal{L}_{\tac}(u,v;\sigma,\tau_1,
\tau_2)=\tilde{\mathcal {L}}_{\tac
}(u,v;\sigma,
\tau_1, \tau_2) - \phi_{2\tau_1, 2\tau_2}(u,v),
\end{equation}
where $\phi_{s, t}(u,v)$ is defined in \eqref{eq:nonessential_Pearcey}
and by \citet{Delvaux13}, Formula (2.29),
%
\begin{eqnarray}
\label{equiv20a}&&\tilde{\mathcal{L}}_{\tac}(u,v;\sigma,\tau_1,
\tau_2)
\nonumber
\\
&&\qquad=\frac{1}{2^{2/3}} \int_\sigma^\infty
\bigl(\hat{p}_1(u;s,\tau_1)\hat{p}_1(v;s,-
\tau _2)\\
&&\hspace*{78pt}{}+\hat {p}_1(-u;s,\tau_1)
\hat{p}_1(-v;s,-\tau_2) \bigr) \,ds,\nonumber
\end{eqnarray}
and the function $\hat{p}_1(z;s,\tau)$ is equivalent to $\hat{p}_1(z;
s, \tau)$ and $\hat{p}_2(-z; s, \tau)$ defined in \citet{Delvaux13},
Formula
(2.26), with $\lambda= 1$, and by \citet{Delvaux13},
Lemmas 4.2 and
4.3, it has the expression,
%
\begin{equation}
\label{equiv19} \hat{p}_1(z; s,\tau) := \langle b_{\tau, -z,s},
R_s+\delta_0 \rangle_0-\langle
b_{\tau
, z,s}, Q_s \rangle_0,
\end{equation}
where $\langle\cdot, \cdot\rangle_0$ is the inner product on
$L^2[0,\infty)$.
The kernels $\mathcal{L}_{\tac}$ and $K^{\tac}$ are related in the
following proposition.

\begin{prop}\label{4541521548452154}
%
\begin{equation}
\label{equiv21} K_{\tau_i, \tau_j}^{\tac}(\xi,\eta; \sigma)=2^{-
{2}/{3}}
\mathcal {L}_{\tac}\bigl(2^{-{2}/{3}}\xi,2^{-{2}/{3}}\eta; \sigma,
2^{-
{7}/{3}}\tau_i, 2^{-{7}/{3}}\tau_j\bigr).
\end{equation}
\end{prop}

The proof of this proposition is given in Appendix~\ref{sec:tacnode_equivalence}.

\subsection{Organization of the paper}

In Section~\ref{sec:algebra}, we derive the exact formulas for the
transition probability density of NIBM, the so-called reunion
probability of $\mathrm{NIBM}_{0 \to T}$, and the correlation kernel of
$\mathrm{NIBM}_{0 \to T}$. We also
derive the $\tau$-deformed version of the formulas to analyze the
conditional $\mathrm{NIBM}_{0 \to T}$ with fixed total winding number.
In Section~\ref{sec:asymptotics}, we summarize the results about discrete Gaussian
orthogonal polynomials that are necessary for the asymptotic analysis
in this paper. In Section~\ref{sec:winding_number}, we prove Theorem~\ref{winding_number_theorem}. In Section~\ref{sec:correlation}, we
prove Theorems \ref{thmm:corr_kernel_intro} and \ref
{thmm:fixed_winding_number_correlations}. Section~\ref{sec:Riemann-Hilbert} is on the interpolation problem and
Riemann--Hilbert problem associated to Gaussian discrete orthogonal
polynomials, and we prove there the technical results stated in Section~\ref{sec:asymptotics}. Appendix~\ref{sec:critical_pt} contains
technical results needed in the asymptotic analysis of Section~\ref{sec:correlation}, and Appendix~\ref{sec:tacnode_equivalence} gives a
proof of Proposition~\ref{4541521548452154}.

\section{Nonintersecting Brownian motion on the unit circle and
discrete Gaussian orthogonal polynomials} \label{sec:algebra}

In this section, we derive the transition probability density of NIBM,
and the joint correlation function and the correlation kernel of
$\mathrm{NIBM}_{0 \to T}$. For all the probabilistic quantities, we derive the $\tau$-deformed
versions, which have no direct probabilistic meaning, but are
generating functions of the corresponding probabilistic quantities with
fixed offset/winding number.

\subsection{\texorpdfstring{$\tau$}{tau}-deformed transition probability density of NIBM}
\label{tau-deformed transition probability density of NIBM}

Let $P(a; b; t)$ be the transition probability density of one particle
in Brownian motion on $\T$ with diffusion parameter $n^{-1/2}$,
starting from point $a$ and ending at point $b$ after time $t$ as given
in \eqref{eq:def_oneD_trans_prob}. For $n$ labeled particles in NIBM
starting at $\vec{a} = (a_1, \ldots, a_n)$ and ending\vspace*{1pt} at $\vec{b} =
(b_1, \ldots, b_n)$ after time $t$, we denote the transition
probability density $P(\vec{a}; \vec{b}; t)$. By labeled particles, we
mean that the particle beginning at the point $a_j$ must end at the
point $b_j$ for each $j=1,\ldots,n$. Since the Brownian motion on $\T$
is a stationary strong Markov process with continuous transition
probability density, we apply the celebrated Karlin--McGregor formula
[\citet{Karlin-McGregor59}, Theorem~1 and assertion D], and have
%
\begin{eqnarray}
\label{eq:Karlin_McGregor} \sum_{\sigma\in S_n} \sgn(\sigma) P \bigl(
\vec{a}; \vec{b}(\sigma); t \bigr) = \det \bigl[P(a_i;
b_j; t) \bigr]^n_{i, j = 1}
\nonumber
\\[-8pt]
\\[-8pt]
\eqntext{\mbox {where } \vec{b}(
\sigma) = (b_{\sigma(1)}, \ldots, b_{\sigma(n)}).}
\end{eqnarray}
Below we assume that $-\pi\leq a_1 < a_2 < \cdots< a_n < \pi$ and
$-\pi\leq b_1 < b_2 < \cdots< b_n < \pi$. Then $P(\vec{a}; \vec
{b}(\sigma); t)$ is nonzero only if $\sigma$ is a cyclic permutation.
For $\ell\in\{1, \ldots, n \}$, we use the notation $[\ell]$ to
denote the cyclic permutation which shifts by~$\ell$. That is, $[\ell]
\in\intZ/n\intZ\subseteq S_n$ acts on the set $\{1,\ldots, n\}$ as
$[\ell](k) = k + \ell$ or $k + \ell- n$ in $\{1, \ldots, n \}$. Hence,
\eqref{eq:Karlin_McGregor} becomes
%
\begin{equation}
\label{eq:cyclic_Karlin_McGregor} \sum_{[\ell] \in\intZ/n\intZ\subseteq S_n} \sgn\bigl([\ell]\bigr) P
\bigl(\vec {a}; \vec{b}\bigl([\ell]\bigr); t \bigr) = \det \bigl[P(a_i;
b_j; t) \bigr]^n_{i,
j = 1}.
\end{equation}
Now let $A_n=\{a_1, \ldots, a_n\}$ and $B_n=\{b_1, \ldots, b_n\}$ be two
unlabeled sets of points in $\T$, and let $P(A_n; B_n; t)$ be the
transition probability for NIBM on $\T$ with the particles starting at
the points $A_n$ and ending at the points $B_n$, as described in the
paragraph preceding \eqref{eq:tau_def_one_particle}. Then $P(A_n; B_n;
t)$ is obtained from $P(\vec{a}; \vec{b}(\sigma); t)$ via the relation
%
\begin{equation}
\label{eq:cyclic_sum} P(A_n; B_n; t) = \sum
_{\sigma\in S_n} P \bigl(\vec{a}; \vec {b}(\sigma ); t \bigr) = \sum
_{[\ell] \in\intZ/n\intZ\subseteq S_n} P \bigl(\vec {a}; \vec{b}\bigl([\ell]\bigr); t
\bigr).
\end{equation}

In the case that $n$ is odd, we have $\sgn([\ell]) = 1$ for all
$[\ell]
\in\intZ/n\intZ$, and then \eqref{eq:cyclic_sum} and \eqref
{eq:cyclic_Karlin_McGregor} yield
%
\begin{equation}
\label{eq:K_MG_odd_case} P(A_n; B_n; t) = \det
\bigl[P(a_i; b_j; t) \bigr]^n_{i, j = 1}.
\end{equation}
In the case that $n$ is even, the situation is more complicated. The
determinantal formula of $P(A_n; B_n; t)$ has not appeared before in
the literature as far as the current authors can tell, but a discrete
analogue was solved by \citet{Fulmek04}. We summarize Fulmek's result
below, and take the continuum limit to obtain the result for NIBM.

Consider the cylindrical lattice $\intZ_M \times\intZ= \{ ([m], n)
\mid m = -M/2, \ldots, M/2 - 1, n \in\intZ\}$, where $M$ is assumed
to be even, and we take the canonical representation for $\intZ_M$ to
be the integers between (and including) $-M/2$ and $M/2 - 1$.
We define a step to the left as the edge from $([m], n)$ to $([m - 1],
n + 1)$, and a step to the right as the edge from $([m], n)$ to $([m +
1], n + 1)$. We assign weight the $x$ to each step to the left and
weight $y$ to each step to the right, so that
%
\begin{equation}\quad
\label{eq:weight_of_each_step} w(e) := %
\cases{ x, &\quad $\mbox{if $e = \bigl[\bigl([m], n
\bigr) \to\bigl([m - 1], n + 1\bigr)\bigr]$ is a step to the left}$, \vspace*{2pt}
\cr
y, &\quad $\mbox{if $e = \bigl[\bigl([m], n\bigr) \to\bigl([m + 1], n + 1\bigr)
\bigr]$ is a step to the right}$. }\hspace*{-20pt} %
\end{equation}
A path on the lattice is defined as a sequence of adjacent steps,
either to the left or to the right. We define the weight of a path as
the product of the weights of its edges, so that
%
\begin{equation}
w\bigl(p = (e_1, \ldots, e_N)\bigr) := \prod
^N_{i = 1} w(e_i),
\end{equation}
and for an arbitrary $n$-tuple of paths $(p_1, \ldots, p_n)$, define
its weight as $w ((p_1,\break \ldots,  p_n) ) = \prod^n_{i = 1} w(p_i)$.
Furthermore, for a set of objects whose weights are defined, we define
the generating function of these weighted objects as the sum of their
weights, so that
%
\begin{equation}
\GF(A) := \sum_{a \in A} w(a).
\end{equation}
Let $-M/2 \leq\alpha_1 < \alpha_2 < \cdots< \alpha_n < M/2$ and $M/2
\leq\beta_1 < \beta_2 < \ldots< \beta_n < M/2$ such that $\al_i,
\be
_i$ are all even, and $N$ be an even integer. We denote $\Pcal(\alpha
_i; \beta_j; N)$ as the set of paths connecting $([\alpha_i], 0)$ and
$([\beta_j], N)$. For any $\sigma\in S_n$, denote $\Pcal(\vec
{\alpha};
\vec{\beta}(\sigma); N)$ as the set of the $n$-tuples of
nonintersecting paths $(p_1, \ldots, p_n)$ such that $p_i$ connects
$([\alpha_i], 0)$ and $([\beta_{\sigma(i)}], N)$.

The celebrated Lindstr\"{o}m--Gessel--Viennot formula [\citet{lindstrom73},
\citet{Gessel-Viennot85}] yields that
%
\begin{eqnarray}
\label{eq:Lindstrom-Gessel-Viennot_original} %
&&\sum_{\sigma\in\intZ/n\intZ\subseteq S_n} \sgn(\sigma)
\GF \bigl(\Pcal \bigl(\vec{\alpha}; \vec{\beta}(\sigma); N \bigr) \bigr) \nonumber\\
&&\qquad= \sum
_{[\ell] \in S_n} \sgn\bigl([\ell]\bigr) \GF \bigl(\Pcal \bigl(
\vec {\alpha}; \vec{\beta}\bigl([\ell]\bigr); N \bigr) \bigr)
\\
&&\qquad= \det \bigl(\GF \bigl(\Pcal(\alpha_i; \beta_j; N)
\bigr) \bigr)^n_{i,
j = 1},\nonumber %
\end{eqnarray}
where in the first identity we have used that there are no
nonintersecting paths connecting $([\alpha_i], 0)$ and $([\beta
_{\sigma
(i)}], N)$ for all $i$ unless $\sigma$ is a cyclic permutation.

With the weights $x = y = 1/2$, we find that $\GF (\Pcal(\alpha_i;
\beta_j; N) )$ is the probability that a random walker on $\intZ_M$
that starts at $[\alpha_i]$ will end at $[\beta_j]$ after time~$N$.
Similarly $\GF(\Pcal(\vec{\alpha}; \vec{\beta}(\sigma); N))$ is the
probability that $n$ labeled vicious walkers (i.e., their paths do not
intersect) on $\intZ_M$ which start at $[\alpha_1], \ldots, [\alpha_n]$
will end at $[\beta_{\sigma(1)}], \ldots, [\beta_{\sigma(n)}]$,
respectively. By Donsker's theorem [\citet{Durrett10}] the path of a
random walk converges to the path of Brownian motion in the sense of
weak convergence as the step length becomes small and the number of
steps becomes large. Similarly, the paths of $n$ vicious walkers on the
circle converge to the paths of NIBM in the weak sense. A rigorous
proof of this intuitively clear convergence result, together with a
bound of convergence rate, is given by \citet{Baik-Suidan07} in the
setting of nonintersecting Brownian motion on the real line. We do not
repeat the proof here. One consequence of the convergence is the
following convergence of the transition probability density. Let $M, N
\to\infty$ such that
%
\begin{equation}
\label{eq:Brownian_scaling} \frac{\alpha_i}{M} \to\frac{a_i}{2\pi},\qquad
 \frac{\beta_i}{M} \to
\frac{b_i}{2\pi}, \qquad
\frac{N}{M^2} \to\frac{t}{4\pi^2 n},
\end{equation}
and the arrays of $a_i$'s and $b_i$'s are distinct, respectively. Then
%
\begin{eqnarray}
\label{eq:Lindstrom_Gessel_Viennot} \frac{M}{4\pi} \GF \bigl(\Pcal(\alpha_i;
\beta_j; N) \bigr) &\to& P(a_i, b_j; t)\quad
\mbox{and}
\nonumber
\\[-8pt]
\\[-8pt]
\nonumber
\biggl( \frac{M}{4\pi} \biggr)^n \GF \bigl(\Pcal \bigl(
\vec{\alpha}; \vec{\beta}(\sigma); N \bigr) \bigr) &\to& P \bigl(\vec{a}; \vec{b}(
\sigma); t \bigr),
\end{eqnarray}
and the discrete identity \eqref{eq:Lindstrom_Gessel_Viennot} implies
\eqref{eq:cyclic_Karlin_McGregor} as the continuous limit.

We now introduce the phase parameter $\tau$, and consider
%
\begin{equation}
\label{eq:weight_with_tau} x =
\frac{1}{2} e^{-({2\pi i}/{M})\tau},\qquad y = \frac{1}{2}
e^{({2\pi i}/{M})\tau}.
\end{equation}
To analyze the information carried by $\tau$, we recall the \emph
{offset} of the trajectory of a particle moving on $\T$. Suppose a
particle $\theta$ moves on $\T$ such that $\theta(t_1) = e^{ai}$ and
$\theta(t_2) = e^{bi}$ where $a, b \in[-\pi, \pi)$, and the trajectory
of $\theta$ is expressed as $\theta(t) = e^{ix(t)}$ where $x(t)\dvtx [t_1,
t_2] \rightarrow\R$ is continuous for $t \in[t_1, t_2]$. Then the
offset of the trajectory of $\theta$ is defined as $[(x(t_2) - x(t_1))
- (b - a)]/(2\pi)$. If $a = b$, the offset is more commonly called the
\emph{winding number}. To consider the path on the lattice $\intZ_M
\times\intZ$, we identity the first coordinate $[m_1] \in\intZ_M$ as
the discrete point $e^{2m_1\pi i/M}$ on $\T$, and consider the second
coordinate $m_2 \in\intZ$ as the discrete time $4\pi^2 n m_2/M^2$.
Then a path on the lattice connecting $([\alpha_i], 0)$ and $([\beta
_j], N)$ is identified as a trajectory of a particle $\theta$ on $\T$
such that $\theta(0) = e^{{2\alpha_i \pi i}/{M}}$, $\theta(4\pi^2
nN/M^2) = e^{{2\beta_j \pi i}/{M}}$, and $\theta(t) = e^{ix(t)}$
where $x(t)$ is continuous on $[0, 4\pi^2 nN/M^2]$. Furthermore, we can
require $x(0) = \frac{2\alpha_i \pi}{M}$ and $x(4\pi^2 nN/M^2) =
\frac
{2\beta_j \pi}{M} + 2\pi o$ where $o \in\intZ$. Then we say that $o$
is the offset of the path.

Express
%
\begin{equation}
\label{eq:separation_by_offset_one_particle} \Pcal(\alpha_i; \beta_j; N) = \bigcup
_{o \in\intZ} \Pcal _o(\alpha_i;
\beta_j; N),
\end{equation}
where
%
\begin{equation}
\Pcal_o(\alpha_i; \beta_j; N) = \bigl\{
\mbox{paths connecting $\bigl([\alpha _i], 0\bigr)$ and $\bigl([
\beta_j], N\bigr)$ with offset $o$} \bigr\}.\hspace*{-20pt}
\end{equation}
Then the paths in $\Pcal_o(\alpha_i; \beta_j; N)$ on the lattice
$\intZ
_M \times\intZ$ have a canonical 1--1 correspondence with paths on
$\intZ\times\intZ$ that connect $(\alpha_i, 0)$ and $(\beta_j + oM,
N)$ and are made of adjacent steps either to the left or to the right.
Here, by steps to the left (resp., to the right), we mean edges
connecting $(m_1, m_2)$ and $(m_1 - 1, m_2 + 1)$ [resp., edges
connecting $(m_1, m_2)$ and $(m_1 + 1, m_2 + 1)$].

Letting
%
\begin{eqnarray}
\label{eq:transition_prob_ZxZ} \Prob_o(\alpha_i; \beta_j;
N) &:=& \mbox{transition probability of random walk on $\intZ$ from $
\alpha_i$}
\nonumber
\\[-8pt]
\\[-8pt]
\nonumber
&&\mbox{to $\beta_j + oM$ after time $N$},
\end{eqnarray}
we have that
%
\begin{eqnarray}
\label{eq:decomp_of_one_particle_trans_prob_offset}
\GF \bigl(\Pcal(\alpha_i; \beta_j; N)
\bigr) &=& \sum_{o \in\intZ
} \GF \bigl(\Pcal_o(
\alpha_i; \beta_j; N) \bigr)
\nonumber
\\[-8pt]
\\[-8pt]
\nonumber
& =& \sum
_{o \in\intZ} \Prob _o(\alpha_i;
\beta_j; N) e^{(\beta_j - \alpha_i){2\tau\pi
i}/{M} +
2o\tau\pi i}.
\end{eqnarray}

Consider $n$ nonintersecting paths that connect $ ([\alpha_i],
0 )$ to $ ([\beta_i], N )$, respectively, for $i = 1,
\ldots, n$. We find that the total offset of these paths has to be $kn$
($k \in\intZ$), since all the paths have the same offset. Similarly,
letting $\sigma= [\ell] \in\intZ/n\intZ$, the total offset of $n$
nonintersecting paths that connect $([\alpha_i], 0)$ to $[\beta
_{\sigma
(i)}], N)$, respectively, for $i = 1, \ldots, n$ has to be $kn + \ell$
($k \in\intZ$). Similar to \eqref
{eq:separation_by_offset_one_particle}, we write for $\sigma= [\ell]$,
%
\begin{equation}
\Pcal \bigl(\vec{\alpha}; \vec{\beta}\bigl([\ell]\bigr); N \bigr) = \bigcup
_{o
\in n\intZ+ \ell} \Pcal_o \bigl(\vec{\alpha}; \vec{
\beta}\bigl([\ell]\bigr); N \bigr),
\end{equation}
where
%
\begin{eqnarray}
&&\Pcal_o\bigl(\vec{\alpha}; \vec{\beta}\bigl([\ell]\bigr); N\bigr)\hspace*{-20pt}\nonumber 
\\
&&\qquad:=
\bigl\{ \mbox {$n$-tuples of nonintersecting paths connecting $\bigl([
\alpha_i], 0\bigr)$ to $\bigl([\beta_{[\ell](i)}], N\bigr)$}\hspace*{-20pt}
\\
&&\hspace*{205pt}\mbox{($i = 1, \ldots, n$) with total offset $o$} \bigr\}.\hspace*{-20pt}\nonumber
\end{eqnarray}
Then, similar to the paths in $\Pcal_o(\alpha_i; \beta_j; N)$, the
$n$-tuples of nonintersecting paths in $\Pcal_o(\alpha_1, \ldots,
\alpha_n; \beta_{[\ell](1)}, \ldots, \beta_{[\ell](n)}; N)$ on the
lattice $\intZ_M \times\intZ$ have the canonical 1--1 correspondence
with the $n$-tuples of paths $(x_1(t), \ldots, x_n(t))$ on $\intZ
\times\intZ$ such that they connect $(\alpha_1, 0)$ to $(\beta
_{\ell+
1} + knM, N), \ldots, (\alpha_{n - \ell}, 0)$ to $(\beta_n + knM, N)$,
$(\alpha_{n - \ell+ 1}, 0)$ to $(\beta_1 + k(n + 1)M, N)$, \ldots,
$(\alpha_n, 0)$ to $(\beta_\ell+ k(n + 1)M, N)$, respectively, and
satisfy $x_n(t) - x_1(t) < M$ for all $t = 0, \ldots, N$. Similar to
\eqref{eq:transition_prob_ZxZ}, let us denote
%
\begin{eqnarray}
\label{eq:defn_Prob_o_two_time} 
\Prob_o\bigl(\vec{\alpha}; \vec{\beta}\bigl([\ell]
\bigr); N\bigr) &:=& \mbox{transition probability of $n$ vicious walkers
$x_1(t), \ldots, x_n(t)$}\hspace*{-10pt}\nonumber
\\
&&{}\mbox{on $\intZ$ such that }
\mbox{$x_i(0) = \alpha_i$, $\displaystyle x_i(N) =
\beta_{[\ell](i)} + \biggl[\frac{o
+ i - 1}{n} \biggr]M$}\hspace*{-10pt}
\\
&&\mbox{and
$x_n(t) - x_1(t) < M$ for all $t = 0, \ldots , N$}.\hspace*{-10pt}
\nonumber
\end{eqnarray}
Then, similar to \eqref{eq:decomp_of_one_particle_trans_prob_offset}, we
have that
%
\begin{eqnarray}
\label
{eq:decomp_of_nonintersecting_particles_trans_prob_offset}
\GF \bigl(\Pcal \bigl(\vec{\alpha}; \vec{\beta}\bigl([\ell]
\bigr); N \bigr) \bigr) &=& \sum_{o \in n\intZ+ \ell} \GF \bigl(
\Pcal_o \bigl(\vec{\alpha}; \vec{\beta}\bigl([\ell]\bigr); N \bigr)
\bigr)
\nonumber
\\[-8pt]
\\[-8pt]
\nonumber
&=& \sum_{o \in n\intZ+ \ell} \Prob_o \bigl(\vec{
\alpha}; \vec {\beta }\bigl([\ell]\bigr); N \bigr) e^{\sum^n_{k = 1}
 (\beta_k - \alpha_k)
{2\tau
\pi i}/{M} + 2o\tau\pi i}. %
\end{eqnarray}

Note that if $n$ is even and $[\ell] \in\intZ/n\intZ\subseteq S_n$,
then for any $k \in\intZ$, $\sgn([\ell]) = (-1)^{kn + \ell}$.
Thus, by
\eqref{eq:decomp_of_one_particle_trans_prob_offset} and \eqref
{eq:decomp_of_nonintersecting_particles_trans_prob_offset}, the
determinantal identity \eqref{eq:Lindstrom-Gessel-Viennot_original} implies
%
\begin{eqnarray}
\label{eq:LGV_with_offset} %
&& e^{\sum^n_{k = 1} (\beta_k - \alpha_k)( {2\tau\pi i}/{M})} \sum_{o \in\intZ}
\Prob_o \bigl(\vec{\alpha}; \vec{\beta}\bigl([o \bmod n]\bigr); N
\bigr) (-1)^o e^{2o\tau\pi i}\nonumber
\\
&&\qquad= \sum_{o \in\intZ} (-1)^o \GF \bigl(
\Pcal_o \bigl( \vec{\alpha}; \vec{\beta}\bigl([o \bmod n]\bigr); N
\bigr) \bigr)
\\
&&\qquad= \det \biggl(\sum_{o \in\intZ} \Prob_o(
\alpha_i; \beta_j; N) e^{(\beta_j - \alpha_i)
({2\tau\pi i}/{M}) + 2o\tau\pi i}
\biggr)^n_{i, j = 1}.\nonumber %
\end{eqnarray}
In the scaling limit $M, N \to\infty$ given in \eqref
{eq:Brownian_scaling} with distinct arrays of $a_i$'s and $b_i$'s,
respectively, the random walk converges to Brownian motion with
diffusion parameter $n^{-1/2}$. Therefore, analogous to \eqref
{eq:Lindstrom_Gessel_Viennot} we obtain
%
\begin{eqnarray}
\label{eq:defn_of_P_density_o} \frac{M}{4\pi} \Prob_o(\alpha_i;
\beta_j; N)& \to&\frac{\sqrt
{n}}{\sqrt
{2\pi t}} e^{-{n(b_j - a_i + 2o\pi)^2}/{(2t)}} \quad\mbox{and}
\nonumber
\\[-8pt]
\\[-8pt]
\nonumber
\biggl( \frac{M}{4\pi} \biggr)^n \Prob_o\bigl(\vec{
\alpha}; \vec {\beta}\bigl([o \bmod n]\bigr); N\bigr)& \to& P_o(A_n;
B_n; t),
\end{eqnarray}
where $P_o(A_n; B_n; t)$ is the transition probability of NIBM with
fixed offset $o$, defined as
%
\begin{eqnarray}
&&P_o(A_n; B_n; t)
\nonumber
\\[-6pt]
\\[-8pt]
\nonumber
&&\qquad:= \lim
_{\Delta x \to0} \frac{1}{(\Delta x)^n} \Prob \left( %
\begin{array}{c}
\mbox{$n$ particles in NIBM start at $a_1, \ldots, a_n$}\\
\mbox{and after time $t$ end in}\\
\mbox{$[b_1, b_1 +
\Delta x), \ldots, [b_n, b_n + \Delta x)$ with total
offset $o$} %
\end{array}
 \right).\hspace*{-20pt}
\end{eqnarray}

Denote
%
\begin{equation}
\label{eq:defnPAnBnttau} P(A_n; B_n; t; \tau) := \det
\bigl(P(a_i; b_j; t; \tau) \bigr)^n_{i, j = 1},
\end{equation}
where $P(a; b; t; \tau)$ is defined in \eqref{eq:tau_def_one_particle}.
We now take \eqref{eq:LGV_with_offset} in the scaling limit \eqref
{eq:Brownian_scaling}, and derive that if $n$ is even
%
\begin{eqnarray}
\label{eq:decompPoABteven}&& e^{\sum^n_{k = 1} (b_k - a_k)\tau i}
 \sum_{o \in\intZ}
P_o(A_n; B_n; t) (-1)^o
e^{2o\tau\pi i}
\nonumber
\\[-8pt]
\\[-8pt]
\nonumber
&&\qquad = e^{\sum^n_{k = 1} (b_k - a_k)\tau i}P(A_n; B_n; t; \tau).
\end{eqnarray}
With $\tau=1/2$, \eqref{eq:decompPoABteven} implies
%
\begin{equation}
\label{eq:trans_prob_even} P(A_n; B_n; t) = \sum
_{o \in\intZ} P_o(A_n; B_n; t) =
P \biggl(A_n; B_n; t; \frac{1}{2} \biggr),
\end{equation}
for $n$ even. For $n$ odd, we have a similar formula in \eqref
{eq:K_MG_odd_case}, which can be written as
%
\begin{equation}
\label{eq:trans_prob_odd} P(A_n; B_n; t) = \sum
_{o \in\intZ} P_o(A_n; B_n; t) =
P(A_n; B_n; t; 0).
\end{equation}
The two formulas \eqref{eq:trans_prob_even} and \eqref
{eq:trans_prob_odd} are combined to give Proposition~\ref
{prop:trans_prob_determinant}.

In what follows we consider $P(A_n; B_n; t; \tau)$ for a general $\tau
\in\R$. To get the transition probability density for NIBM, we simply
let $\tau= 0$ or $\tau= 1/2$ depending on the parity of the number of
particles. One advantage of working with $P(A_n; B_n; t; \tau)$ with
general $\tau$ is that $P(A_n; B_n; t; \tau)$ is a generating function
for $P_o(A_n; B_n; t)$. We call $P(A_n; B_n; t; \tau)$ the \emph
{$\tau
$-deformed transition probability density} of NIBM.

\subsection{\texorpdfstring{$\tau$}{tau}-deformed reunion probability}\label{reunion_probability}

Now we consider the limiting case that $a_1, \ldots, a_n$ are close to
$0$ and/or $b_1, \ldots, b_n$ are close to $0$. In the case that $a_i
\to0$ and $b_i$ are fixed and distinct, by l'H\^{o}pital's rule,
%
\begin{eqnarray}
\label{eq:det_a_to_0} P(A_n; B_n; t; \tau) &=&
\frac{\prod_{1 \leq j < k \leq n} (a_k -
a_j)}{\prod^{n - 1}_{j = 0} j!}
\nonumber
\\[-8pt]
\\[-8pt]
\nonumber
&&{}\times
\det \biggl( \frac{d^{j -
1}}{dx^{j - 1}} P(x; b_k; t; \tau)
\bigg\vert_{x = 0} \biggr) \bigl(1 + \bigO \bigl(\max\bigl(\vert a_i
\vert\bigr) \bigr) \bigr).
\end{eqnarray}
Similarly, in the case that $b_i \to0$ and $a_i$ are fixed and distinct,
%
\begin{eqnarray}
\label{eq:det_b_to_0} P(A_n; B_n; t; \tau)& =& \frac{\prod_{1 \leq j < k \leq n} (b_k -
b_j)}{\prod^{n - 1}_{j = 0} j!}
\nonumber
\\[-8pt]
\\[-8pt]
\nonumber
&&{}\times
\det \biggl( \frac{d^{j -
1}}{dx^{j - 1}} P(a_k; x; t; \tau)
\bigg\vert_{x = 0} \biggr) \bigl(1 + \bigO \bigl(\max\bigl(\vert b_i
\vert\bigr) \bigr) \bigr).
\end{eqnarray}
In the case that both $a_i \to0$ and $b_i \to0$, we define
%
\begin{equation}
R_n(t; \tau) = \det \biggl( \frac{d^{j + k - 2}}{dx^{j + k - 2}} P(0; x; t; \tau)
\bigg\vert_{x = 0} \biggr),
\end{equation}
and have the \emph{$\tau$-deformed reunion probability}
%
\begin{eqnarray}
\label{eq:PABttaua0b0}&& P(A_n; B_n; t; \tau)
\nonumber
\\[-8pt]
\\[-8pt]
\nonumber
&&\qquad=
\frac{\prod_{1 \leq j < k \leq n} (a_j -
a_k)(b_k - b_j)}{\prod^{n - 1}_{j = 0} j!^2}
R_n(t; \tau) \bigl(1 + \bigO \bigl(\max\bigl(\vert a_i
\vert, \vert b_i \vert\bigr) \bigr) \bigr).
\end{eqnarray}
The transition probability density $P(A_n; B_n; t; \epsilon(n))$ of the
particles in NIBM with starting point $a_i \to0$ and ending point $b_i
\to0$ is called the \emph{reunion probability} in \citet
{Forrester-Majumdar-Schehr11}. In \citet{Forrester-Majumdar-Schehr11},
the \emph{normalized reunion probability} is defined in the setting of
our paper as
%
\begin{equation}
\tilde{G}_n(L) = \frac{ ({2\pi}/{L}  )^{2n^2} R_n
 (
{4\pi^2 n}/{L^2}, \epsilon(n)  )}{\lim_{t \to0} t^{n^2} R_n(nt,
\epsilon(n))}.
\end{equation}
Note that the normalized reunion probability is not real probability
since it can exceed $1$.

In our paper, we are interested in the $\tau$-deformed transition
probability $P(A_n; B_n; t; \tau)$ and $R_n(t; \tau)$ because they
contain information on the total winding number in NIBM with common
starting point and the same common ending point. By \eqref
{eq:PABttaua0b0}, as $a_1, \ldots, a_n \to0$ and $b_1,\ldots, b_n
\to0$,
%
\begin{eqnarray}
P_{\omega}(A_n; B_n; t)& = &\frac{\prod_{1 \leq j < k \leq n} (a_j -
a_k)(b_k - b_j)}{\prod^{n - 1}_{j = 0} j!^2}
e^{2\pi\epsilon(n)
\omega i} R_{n, \omega}(t)
\nonumber
\\[-8pt]
\\[-8pt]
\nonumber
&&{}\times \bigl(1 + \bigO \bigl(\max\bigl(\vert
a_i \vert, \vert b_i \vert\bigr) \bigr) \bigr),
\end{eqnarray}
where $R_{n, \omega}(t)$ is defined as
%
\begin{equation}
\label{Rno_def} R_{n, \omega}(t) = \int^1_0
R_n(t; \tau) e^{-2\omega\tau\pi i} \,d\tau.
\end{equation}
Note that the ratio
%
\begin{equation}
\label{eq:R_n_ratio} \frac{e^{2\pi\epsilon(n) \omega i} R_{n, \omega}(t)}{R_n(t;
\epsilon(n))} = \mathop{\lim_{a_1, \ldots, a_n \to0 }}_{ b_1,\ldots, b_n \to0}
\frac
{P_{\omega}(A_n; B_n; t)}{P(A_n; B_n; t)},
\end{equation}
is the probability that the total winding number of the $n$ particles
in NIBM starting at a common point and ending at the same common point
is $\omega$.

To evaluate $R_n(t; \tau)$ and the determinants on the right-hand sides
of \eqref{eq:det_a_to_0} and~\eqref{eq:det_b_to_0}, we consider the
Fourier series of entries of these determinants. Introduce the lattice
%
\begin{equation}
L_{n, \tau}:= \biggl\{ \frac{k + \tau}{n} \bigg\mid k \in\intZ \biggr\}.
\end{equation}
By the Poisson resummation formula, we find
%
\begin{eqnarray}
\label{eq:formulaofPabttau} %
P(a; \theta; t; \tau) &=& \frac{\sqrt{n}}{\sqrt{2\pi t}} \sum
_{l
\in
\intZ} e^{-{n(\theta- a + 2l\pi)^2}/{(2t)}} e^{2l\pi\tau i}
\nonumber\\
&=& \frac{\sqrt{n}}{\sqrt{2\pi t}} \sum_{k \in\intZ} \int
^{\infty
}_{-\infty} e^{-{n(\theta- a+2\xi\pi)^2}/{(2t)}} e^{-2\pi i\xi
(k-\tau
)} \,d
\xi
\nonumber
\\[-8pt]
\\[-8pt]
\nonumber
&=& \frac{1}{2\pi} \sum_{k \in\intZ} e^{-{t(k - \tau)^2}/{(2n)} }
e^{i(\theta-a)(k-\tau)}
\\
&=& \frac{1}{2\pi} \sum_{x \in L_{n, \tau}} e^{-{tn x^2}/{2}}
e^{-inx(\theta- a)}. \nonumber%
\end{eqnarray}
It follows that
%
\begin{equation}
\label{eq:derivative_of_P_1st_variable} \frac{d^j}{d\theta^j}
 P(a; \theta; t; \tau) = \frac{(-ni)^j}{2\pi
}
\sum_{x \in L_{n, \tau}} x^j e^{-{tn x^2}/{2}}
e^{-inx(\theta- a)}.
\end{equation}
Similarly,
%
\begin{eqnarray}
\label{eq:derivative_of_P_2nd_variable} P(\theta; b; t; \tau)
&= &\frac{1}{2\pi} \sum
_{x \in L_{n, \tau}} e^{-{tn x^2}/{2}} e^{inx(\theta- b)},
\\
\frac{d^j}{d\theta^j} P(\theta; b; t; \tau) &=& \frac{(ni)^j}{2\pi} \sum
_{x \in L_{n, \tau}} x^j e^{-{tn x^2}/{2}} e^{inx(\theta- b)},
\end{eqnarray}
and in particular
%
\begin{equation}
\frac{d^j}{d\theta^j} P(0; \theta; t; \tau)\bigg \vert _{\theta= 0} =
\frac{(-ni)^j}{2\pi} \sum_{x \in L_{n, \tau}} x^j
e^{-
{tn x^2}/{2}}.
\end{equation}

Now setting $t=T$, we find that
%
\begin{eqnarray}
\label{eq:defnRttau} R_n(T; \tau) = (-1)^{{n(n - 1)}/{2}}
\frac{n^{n^2}}{(2\pi)^n} \Hankel_n(T; \tau)
\nonumber
\\[-8pt]
\\[-8pt]
\eqntext{\mbox{where }
\displaystyle \Hankel_n(T; \tau) := \det \biggl( \frac{1}{n} \sum
_{x \in L_{n, \tau}} x^{j + k - 2} e^{-{Tn x^2}/{2}}
\biggr)^n_{j, k = 1}.}
\end{eqnarray}
Note that $\Hankel_n(t; \tau)$ is the Hankel determinant with respect
to the discrete measure on the lattice $L_{n,\tau}$,
%
\begin{equation}
\frac{1}{n} \sum_{y \in L_{n, \tau}} e^{-Tn x^2/2}
\delta(x-y).
\end{equation}

\begin{rmk}
Formula \eqref{eq:defnRttau} was obtained in \citet
{Forrester-Majumdar-Schehr11} and \citet
{Comtet-Forrester-Majumdar-Schehr13} with $\tau= 0$ and more recently
in \citet{Dupic-PerezCastillo13} with $\tau= \epsilon(n)$. We note that
the $\mathrm{NIBM}_{0 \to T}$ model is related to Yang--Mills theory on
the sphere, as
shown in \citet{Forrester-Majumdar-Schehr11}, and a similar formula was
derived in the Yang--Mills theory setting in \citet{Douglas-Kazakov93}
with $\tau= \epsilon(n)$.
\end{rmk}

By a standard result for Hankel determinants, we can express $\Hankel
_n(T; \tau)$ using the \emph{discrete Gaussian orthogonal polynomials}.
Let $p^{(T; \tau)}_{n, j}(x)$ be the monic polynomial of degree $j$
that satisfies
%
\begin{equation}
\label{eq:defn_of_discrete_Gaussian_OP} \frac{1}{n} \sum_{x \in L_{n, \tau}}
p^{(T; \tau)}_{n, j}(x) p^{(T;
\tau
)}_{n, k}(x)e^{-Tn x^2/2}
= 0 \qquad\mbox{if $j \neq k$}.
\end{equation}
We then have [see e.g., \citet{Bleher-Liechty13}, Proposition~2.2.2],
%
\begin{equation}
\Hankel_n(T; \tau) = \prod^{n - 1}_{j = 0}
h^{(T; \tau)}_{n,j},
\end{equation}
where
%
\begin{equation}
\label{eq:defn_of_h_nk} h^{(T; \tau)}_{n, k} := \frac{1}{n} \sum
_{x \in L_{n, \tau}} p^{(T;
\tau
)}_{n, k}(x)^2
e^{-Tn x^2/2}.
\end{equation}
The orthogonal polynomials \eqref{eq:defn_of_discrete_Gaussian_OP}
satisfy the three term recurrence equation [see \citet{Szego75}],
%
\begin{equation}
\label{eq:three_term_recurrence} xp^{(T; \tau)}_{n, j}(x)=p^{(T; \tau)}_{n, j+1}(x)+
\be^{(T; \tau)}_{n,
j}p^{(T; \tau)}_{n, j}(x)+ \bigl(
\gamma^{(T; \tau)}_{n, j} \bigr)^2p^{(T;
\tau)}_{n, j-1}(x),
\end{equation}
where $\{\be^{(T; \tau)}_{n, j}\}_{j=0}^\infty$ is a sequence of real
constants, and
%
\begin{equation}
\label{eq:defn_of_gamma_nk} \gamma^{(T; \tau)}_{n, j} := \biggl(
\frac{h^{(T; \tau)}_{n,
j}}{h^{(T; \tau
)}_{n, j-1}} \biggr)^{1/2}.
\end{equation}

\subsection{\texorpdfstring{$\tau$}{tau}-deformed multi-time correlation functions}\label
{subsec:tau_deformed_corr_functions}

Next, we consider the joint probability density of $n$-particles in
NIBM at times $t_1, \ldots, t_m$ such that $0 < t_1 < \cdots< t_m < T$
with the initial condition that they start from the common position $0
\in[-\pi, \pi) = \unitC$ at time $0$ and end at the same common
position at $T$. That is, we consider the joint probability density in
$\mathrm{NIBM}_{0 \to T}$. We also want to extract the information of
joint probability
density for each fixed total offset/winding number of the
$n$-particles. Thus, we consider the $\tau$-deformed joint probability
density function for the Brownian particles. This density function is
the one given in \eqref{eq:jpdf1} in the physical setting. In order to
get the $\tau$-deformed version, we start with the discrete model as in
Section~\ref{tau-deformed transition probability density of NIBM}.

Let $N_0 = 0 < N_1 < \cdots< N_m < N_{m + 1} = N$ be even integers and
$\alpha^{(k)}_i$ be even integers for $k = 0, \ldots, m + 1$ and $i =
1, \ldots, n$ such that for all $k = 0, \ldots, m + 1$,
%
\begin{equation}
-\frac{M}{2} \leq\alpha^{(k)}_1 <
\alpha^{(k)}_2 < \cdots< \alpha ^{(k)}_n
< \frac{M}{2}.
\end{equation}
Let $\sigma_1, \ldots, \sigma_{m + 1} \in S_n$ be permutations. Denote
$\Pcal(\vec{\alpha}^{(0)}; \vec{\alpha}^{(1)}(\sigma_1); \ldots;\break
\vec
{\alpha}^{(m + 1)}(\sigma_{m + 1}); N_1; \ldots; N_{m + 1})$ be the set
of $n$-tuples of nonintersecting paths $(p_1, \ldots, p_n)$ such that
$p_i$ connects $([\alpha^{(0)}_i], 0)$, $([\alpha^{(1)}_{\sigma_1(i)}],
N_1), \ldots, ([\alpha^{(m + 1)}_{\sigma_{m + 1}(i)}],\break  N_{m + 1})$
successively, and denote $\Pcal^{(\sigma)}(\vec{\alpha}^{(0)};
\ldots;
\vec{\alpha}^{(m)}; \vec{\alpha}^{(m + 1)}; N_1; \ldots; N_{m +
1})$ as
the union of $\Pcal(\vec{\alpha}^{(0)}; \vec{\alpha}^{(1)}(\sigma_1);
\ldots; \vec{\alpha}^{(m)}(\sigma_m); \vec{\alpha}^{(m +
1)}(\sigma);
N_1; \ldots; N_{m + 1})$ for all $\sigma_1, \ldots, \sigma_m \in S_n$.
Note that we only need to consider cyclic permutations $\sigma_k \in
\intZ/n\intZ\subseteq S_n$ due to the nonintersecting assumption.
Using the Lindstr\"{o}m--Gessel--Viennot formula repeatedly, we have,
as a
generalization of \eqref{eq:Lindstrom-Gessel-Viennot_original},
%
\begin{eqnarray}
\label{eq:LGV_formula_multtime} %
&& \sum_{[\ell] \in\intZ/n\intZ\subseteq S_n} \sgn\bigl([
\ell]\bigr) \GF \bigl( \Pcal^{[\ell]}\bigl(\vec{\alpha}^{(0)};
\vec{\alpha}^{(1)}; \ldots; \vec {\alpha}^{(1)}; \vec{
\alpha}^{(m + 1)}; N_1; \ldots; N_{m + 1}\bigr) \bigr)\nonumber
\\
&&\qquad= \sum_{\sigma_1, \ldots, \sigma_m, [\ell] \in\intZ/n\intZ
\subseteq S_n} \sgn\bigl([\ell]\bigr) \GF \bigl(
\Pcal\bigl(\vec{\alpha}^{(0)}; \vec {\alpha}^{(1)}(
\sigma_1); \ldots; \vec{\alpha}^{(m)}(\sigma_m);
\nonumber
\\[-8pt]
\\[-8pt]
\nonumber
&&\hspace*{190pt}\vec {\alpha}^{(m + 1)}\bigl([\ell]\bigr); N_1; \ldots;
N_{m + 1}\bigr) \bigr)
\\
&&\qquad= \prod^{m + 1}_{k = 1} \det \bigl( \GF\bigl(
\Pcal\bigl(\alpha^{(k - 1)}_i; \alpha^{(k)}_j;
N_k - N_{k - 1}\bigr)\bigr) \bigr)^n_{i, j = 1}.\nonumber
\end{eqnarray}

Let the weight for each step in \eqref{eq:weight_of_each_step} be given
by $x = e^{-2\pi\tau i/M}/2$ and $y = e^{2\pi\tau i/M}/2$ as in
\eqref
{eq:weight_with_tau}. Similar to \eqref{eq:defn_Prob_o_two_time},
suppose $o = kn + \ell$ where $\ell= 0, \ldots, n - 1$, we denote
%
\begin{eqnarray}
&&\Prob_o\bigl(\vec{\alpha}^{(0)}; \ldots; \vec{
\alpha}^{(m)}; \vec {\alpha }^{(m + 1)}; N_1; \ldots;
N_m; N_{m + 1}\bigr)\nonumber\\
&&\qquad := \mbox{transition probability of $n$
vicious walkers}\nonumber
\\
&&\quad\qquad\mbox{$x_1(t), \ldots, x_n(t)$ on $\intZ$ such that
$x_i(0) = \alpha ^{(0)}_i$,}
\nonumber
\\[-8pt]
\\[-8pt]
\nonumber
&&\qquad\quad\mbox{$\displaystyle x_i(N_{m + 1}) = \alpha^{(m + 1)}_{[\ell](i)} +
\biggl[\frac
{o + i - 1}{n} \biggr]M$,}
\\
&&\quad\qquad\mbox{$x_i(N_j) = \alpha^{(j)}_l
+ c^{(j)}_l M$ for some $l = 1, \ldots , n$ and
$c^{(j)}_l \in\intZ$,}\nonumber
\\
&&\quad\qquad\mbox{and $x_n(t) - x_1(t) < M$ for all $t = 0, \ldots,
N$.}
\nonumber
\end{eqnarray}
Then, similar to \eqref
{eq:decomp_of_nonintersecting_particles_trans_prob_offset}, we have
%
\begin{eqnarray}
\label{eq:decomp_of_GF_multtime}
&&\GF\bigl(\Pcal^{[\ell]}\bigl(\vec{\alpha}^{(0)};
\vec{\alpha}^{(1)}; \ldots; \vec {\alpha}^{(1)}; \vec{
\alpha}^{(m + 1)}; N_1; \ldots; N_{m + 1}\bigr)\bigr)\nonumber\\
&&\qquad =
\sum_{o \in n\intZ+ \ell} \Prob_o\bigl(\vec{
\alpha}^{(0)}; \ldots; \vec {\alpha}^{(m)}; \vec{
\alpha}^{(m + 1)}; N_1; \ldots; N_{m + 1}\bigr)\\
&&\hspace*{30pt}\qquad\quad{}\times e^{\sum
^n_{k = 1} (\alpha^{(m + 1)}_k - \alpha^{(0)}_k) ({2\tau\pi
i}/{M}) +
2o\tau\pi i}.
\nonumber
\end{eqnarray}

In the limit that $M, N \to\infty$ such that analogous to \eqref
{eq:Brownian_scaling},
%
\begin{equation}
\frac{\alpha^{(j)}_i}{M} \to\frac{a^{(j)}_i}{2\pi}, \qquad\frac
{N_j}{M^2} \to
\frac{t_j}{4\pi^2 n},
\end{equation}
where $0 = t_0 < t_1 < \cdots< t_{m + 1} = T$, and $-\pi\leq
a^{(j)}_1 < \cdots< a^{(j)}_n < \pi$ for each $j = 0, \ldots, m + 1$,
we obtain, similar to \eqref{eq:defn_of_P_density_o},
%
\begin{eqnarray}
&&\biggl( \frac{M}{4\pi} \biggr)^{mn} \Prob_o\bigl(
\vec{\alpha}^{(0)}; \ldots; \vec{\alpha}^{(m)}; \vec{
\alpha}^{(m + 1)}; N_1; \ldots; N_{m +
1}\bigr)
\nonumber
\\[-8pt]
\\[-8pt]
\nonumber
&&\qquad\to
P_o\bigl(A^{(0)}; \ldots; A^{(m + 1)};
t_1; \ldots; t_{m + 1}\bigr),
\end{eqnarray}
where
%
\begin{eqnarray}
&&P_o\bigl(A^{(0)}; \ldots; A^{(m + 1)};
t_1; \ldots; t_{m + 1}\bigr)\nonumber \\
&&\qquad:= \lim_{\Delta x \to0}
\frac{1}{(\Delta x)^{mn}}
\\
&&\qquad\quad{}\times\Prob \left( %
\begin{array}{c}
\mbox{$n$ particles in NIBM start at
$a^{(0)}_1, \ldots, a^{(0)}_n$ at
time $0$},\\
\mbox{stay in $\bigl[a^{(k)}_1, a^{(k)}_1
+ \Delta x\bigr), \ldots,\bigl[a^{(k)}_n,
a^{(k)}_n + \Delta x\bigr)$ at time $t_k$}\\
\mbox{($k = 1,
\ldots , m + 1$) with total offset $o$ at time $t_{m + 1}$.}
\end{array}
 \right).
\nonumber
\end{eqnarray}
Thus, similar to \eqref{eq:decompPoABteven}, equations \eqref
{eq:decomp_of_GF_multtime} and \eqref{eq:LGV_formula_multtime} imply
that the \emph{$\tau$-deformed joint transition probability density} of
$n$ particles in NIBM is [here $\epsilon(n)$ accommodates both even and
odd $n$]
%
\begin{eqnarray}
&&\sum_{o \in\intZ} P_o\bigl(A^{(0)};
\ldots; A^{(m + 1)}; t_1; \ldots; t_{m
+ 1}\bigr)
e^{2\pi\epsilon(n) oi} e^{2o\tau\pi i}
\nonumber
\\[-8pt]
\\[-8pt]
\nonumber
&&\qquad = \prod^{m + 1}_{j = 1}
P\bigl(A^{(j - 1)}; A^{(j)}; t_j - t_{j - 1};
\tau\bigr),
\end{eqnarray}
where $P(A^{(j - 1)}; A^{(j)}; t_j - t_{j - 1}; \tau)$ is defined by
\eqref{eq:defnPAnBnttau} with $A_n, B_n$ replaced by $A^{(j - 1)},
A^{(j)}$. Letting $\tau= \epsilon(n)$, we have the \emph{joint transition
probability density} in NIBM, which is the sum of all $P_o(A^{(0)};
\ldots; A^{(m + 1)}; t_1; \ldots; t_{m + 1})$, expressed as
%
\begin{eqnarray}
&&\sum_{o \in\intZ} P_o \bigl(A^{(0)};
\ldots; A^{(m + 1)}; t_1; \ldots; t_{m + 1} \bigr)
\nonumber
\\[-8pt]
\\[-8pt]
\nonumber
&&\qquad=
\prod^{m + 1}_{j = 1} P \bigl(A^{(j - 1)};
A^{(j)}; t_j - t_{j - 1}; \epsilon(n) \bigr).
\end{eqnarray}

In the limiting case $a^{(0)}_i \to0$ and/or $a^{(m + 1)}_i \to0$, we
have the result similar to \eqref{eq:det_a_to_0}, \eqref{eq:det_b_to_0}
and \eqref{eq:PABttaua0b0}. For $\mathrm{NIBM}_{0 \to T}$ we are
interested in the
ratio between the $\tau$-deformed transition probability density of the
particles from $A^{(0)}$ to $A^{(1)}, \ldots, A^{(m + 1)}$ successively
and the $\tau$-deformed transition probability (i.e., the $\tau
$-deformed reunion probability) of the particles from $A^{(0)}$ to
$A^{(m + 1)}$, as $a^{(0)}_i \to0, a^{(m + 1)}_i \to0$. After
changing the notation $t_{m + 1}$ into $T$, we have the \emph{$\tau
$-deformed joint probability density} in $\mathrm{NIBM}_{0 \to T}$,
%
\begin{eqnarray}
\label{eq:Pmelon} %
&& P_{0 \to T}\bigl(A^{(1)}, \ldots,
A^{(m)}; t_1, \ldots, t_m; \tau\bigr)
\nonumber\\
&&\qquad:= \lim_{a^{(0)}_i \to0, a^{(m + 1)}_i \to0} \frac{\prod^{m +
1}_{j = 1} P(A^{(j - 1)}; A^{(j)}; t_j - t_{j - 1}; \tau)}{P(A^{(0)};
A^{(m + 1)}; t_{m + 1}; \tau)}
\nonumber\\
&&\qquad= \frac{1}{R_n(T; \tau)} \det \biggl( \frac{d^{j - 1}}{dx^{j
- 1}} P\bigl(x;
a^{(1)}_k; t_1; \tau\bigr) \bigg\vert_{x = 0}
\biggr)_{j,k=1}^n \\
&&\qquad\quad{}\times \det \biggl( \frac{d^{j - 1}}{dx^{j - 1}} P
\bigl(a^{(m)}_k; x; T - t_m; \tau\bigr)
\bigg\vert_{x = 0} \biggr)_{j,k=1}^n
\nonumber\\
&&\qquad\quad{} \times\prod^m_{j = 2} P
\bigl(A^{(j
- 1)}; A^{(j)}; t_j - t_{j - 1};
\tau \bigr). \nonumber%
\end{eqnarray}
Note that for any $\tau$, the denominator $R_n(T; \tau)$ is a nonzero
real number, by \eqref{eq:defnRttau} and \eqref{eq:defn_of_h_nk}.
With $\tau= \epsilon(n)$, $P_{0 \to T}(A^{(1)}, \ldots, A^{(m)}; t_1,
\ldots,
t_m; \epsilon(n))$ gives the \emph{joint transition probability density}
of particles in $\mathrm{NIBM}_{0 \to T}$. With the help of Fourier
expansion, $P_{0 \to T}
(A^{(1)}, \ldots, A^{(m)}; t_1, \ldots, t_m; \tau)$ yields the
conditional joint transition probability density with fixed total
winding number. To be precise, we have
%
\begin{eqnarray}
&&\frac{R_n(T; \tau)}{R_n(T, \epsilon(n))} P_{0 \to T}\bigl(A^{(1)}, \ldots,
A^{(m)}; t_1, \ldots, t_m; \tau\bigr)
\nonumber
\\[-8pt]
\\[-8pt]
\nonumber
&&\qquad = \sum
_{\omega\in\intZ} (P_{0
\to T} )_{\omega}
\bigl(A^{(1)}, \ldots, A^{(m)}; t_1, \ldots,
t_m\bigr) e^{2\pi
\omega
(\tau- \epsilon(n))i},
\end{eqnarray}
where
%
\begin{eqnarray}
&&(P_{0 \to T})_{\omega}\bigl(A^{(1)}, \ldots,
A^{(m)}; t_1, \ldots, t_m\bigr)
\nonumber
\\[-4pt]
\\[-12pt]
\nonumber
&&\qquad=\lim_{\Delta x \to0} \frac{1}{(\Delta x)^{mn}} \Prob \left(
\begin{array}{c}
\mbox{$n$ particles in $\mathrm{NIBM}_{0 \to T}$ with
total winding}
\\
\mbox{number $\omega$, there is a particle
in}\\
\mbox{ $\bigl[a^{(i)}_j, a^{(i)}_j + \Delta x
\bigr)$ at time $t_i$} %
\end{array}
 \right).
\end{eqnarray}
Note that although $P_{0 \to T}(A^{(1)}, \ldots, A^{(m)}; t_1, \ldots
, t_m;
\tau)$ may not be nonnegative-valued, it is normalized in the sense
that total integral over all possible positions of $a^{(k)}_j$ is $1$.

By \eqref{eq:Pmelon}, we find that $P_{0 \to T}(A^{(1)}, \ldots, A^{(m)};
t_1, \ldots, t_m; \tau)$ has properties similar to the joint
probability density function of a determinantal process, and thus is
characterized by a reproducing kernel. We apply the Eynard--Mehta
theorem [\citet{Eynard-Mehta98}], to $P_{0 \to T}(A^{(1)}, \ldots, A^{(m)};
t_1, \ldots,\break t_m; \tau)$, following the notational conventions in
\citet
{Borodin-Rains05}.

Denote for $k = 1, \ldots, m - 1$ and $j = 1, \ldots, n$,
%
\begin{eqnarray}
W_k(x, y)& :=& P(x; y; t_{k + 1} - t_k; \tau),
\label{eq:W_k_kernel}
\\
\phi_j(x) &:= &\mbox{linear combination of } \biggl\{
\frac
{d^{l}}{dy^{l}} P(y; x; t_1; \tau) \bigg\vert_{y = 0} \biggr\}
\nonumber
\\[-8pt]
\\[-8pt]
\eqntext{\mbox{for } l =  0, \ldots, j-1 ,}
\\
\psi_j(x)& :=& \mbox{linear combination of } \biggl\{
\frac
{d^{l}}{dy^{l}} P(x; y; T - t_m; \tau) \bigg\vert_{y = 0}
\biggr\}
\nonumber
\\[-8pt]
\\[-8pt]
\eqntext{\mbox{for } l = 0, \ldots, j-1 ,}
\end{eqnarray}
where we suppress the dependence on $\tau$, and the concrete formulas
for $\phi_j(x)$ and $\psi_j(x)$ are to be fixed later in \eqref
{eq:phi_and_psi_in_polynomials} and \eqref
{eq:explicit_expression_of_f_g}. Then we define the operator $\Phi:
L^2(\unitC) \to\ell^2(n)$ as
%
\begin{equation}
\Phi\bigl(f(\theta)\bigr) = \biggl( \int^{\pi}_{-\pi}
f(\theta) \phi _1(\theta) \,d\theta, \ldots, \int^{\pi}_{-\pi}
f(\theta) \phi_n(\theta) \,d\theta \biggr)^T,
\end{equation}
the operator $\Psi: \ell^2(n) \to L^2(\unitC)$ as
%
\begin{equation}
\Psi\bigl((v_1, \ldots, v_n)^T\bigr) = \sum
^n_{k = 1} v_k
\psi_k(\theta),
\end{equation}
and define the operator $W_k\dvtx L^2(\unitC) \to L^2(\unitC)$ by the
kernel function $W_k(x, y)$ in~\eqref{eq:W_k_kernel}. Furthermore, we
define the operators
%
\begin{eqnarray}
\label{eq:Wi,jwithcirc} W_{[i, j)} &:=& %
\cases{ W_i \cdots
W_{j - 1}, &\quad $\mbox{for $i < j$}$, \vspace*{2pt}
\cr
\Id,& \quad$\mbox{for $i =
j$}$, \vspace*{2pt}
\cr
0, &\quad $\mbox{for $i >j$}$, } \qquad
\mbox{and}
\nonumber
\\[-8pt]
\\[-8pt]
\nonumber
\WWW_{[i, j)} &:=& %
\cases{ W_i \cdots
W_{j - 1}, & \quad$\mbox{for $i < j$}$, \vspace*{2pt}
\cr
0, &\quad $\mbox{for $i \geq
j$}$. }
\end{eqnarray}
%
We also define the operator $M\dvtx \ell^2(n) \to\ell^2(n)$ as
%
\begin{equation}
\label{eq:defn_M} M := \Phi W_{[1, m)} \Psi,
\end{equation}
which is represented by the $n \times n$ matrix
%
\begin{eqnarray}
M_{ij}& = &\int\cdots\int_{\unitC^m}
\phi_i(\theta_1) W_1(\theta_1,
\theta_2) \cdots
\nonumber
\\[-8pt]
\\[-8pt]
\nonumber
&&\hspace*{40pt}{}\times W_{m - 1}(\theta_{m - 1},
\theta_m) \psi_j(\theta_m) \,d
\theta_1 \cdots d\theta_m.
\end{eqnarray}
Then for any $k_1, \ldots, k_m \leq n$, we define the \emph{$\tau
$-deformed joint correlation function} as
%
\begin{eqnarray}
\label{eq:defn_tau_deformed_corr_func} &&\hspace*{-4pt}R^{(n)}_{0 \to T}\bigl(a^{(1)}_1,
\ldots, a^{(1)}_{k_1}; \ldots; a^{(m)}_1,
\ldots, a^{(m)}_{k_m}; t_1, \ldots,
t_m; \tau\bigr)\nonumber\\
&&\hspace*{-6pt}\qquad = \prod^m_{j = 1}
\frac
{n!}{(n -
k_j)!}
\nonumber
\\[-8pt]
\\[-8pt]
\nonumber
&&\hspace*{-6pt}\qquad\quad\times\int_{[-\pi, \pi)^{mn - (k_1 + \cdots+ k_m)}} P_{0 \to T}\bigl(A^{(1)},
\ldots, A^{(m)};\\
&&\hspace*{152pt}{} t_1, \ldots, t_m; \tau\bigr) \,d
a^{(1)}_{k_1 + 1} \cdots d a^{(1)}_n\, d
a^{(2)}_{k_2 + 1} \cdots d a^{(m)}_n,\hspace*{-20pt}
\nonumber
\end{eqnarray}
and the Eynard--Mehta theorem gives the determinantal formula
%
\begin{eqnarray}
\label{eq:defn_Rmelon}&& R^{(n)}_{0 \to T}\bigl(a^{(1)}_1,
\ldots, a^{(1)}_{k_1}; \ldots; a^{(m)}_1,
\ldots, a^{(m)}_{k_m}; t_1, \ldots,
t_m; \tau\bigr)
\nonumber
\\[-8pt]
\\[-8pt]
\nonumber
&&\qquad = \det \bigl( K_{t_i,
t_j}
\bigl(a^{(i)}_{l_i}, a^{(j)}_{l'_j} \bigr)
\bigr)_{{i, j = 1,
\ldots, m, l_i = 1, \ldots, k_i , l'_j = 1, \ldots, k_j}} ,
\end{eqnarray}
where the \emph{$\tau$-deformed correlation kernel} is defined as
%
\begin{eqnarray}
\label{eq:two_parts_of_K} K_{t_i, t_j}(x, y)& =& \tilde{K}_{t_i, t_j}(x, y) -
\WWW_{[i, j)} \quad\mbox{and}
\nonumber
\\[-8pt]
\\[-8pt]
\nonumber
 \tilde{K}_{t_i, t_j}(x, y) &=&
W_{[i, m)} \Psi M^{-1} \Phi W_{[1, j)}.
\end{eqnarray}

\begin{rmk} \label{rmk:correlation_kernel}
The kernel $K_{t_i, t_j}(x, y)$ depends on $\tau$, but we suppress it
for notational simplicity. If we let $\tau= \epsilon(n)$, we obtain the
\emph{correlation kernel} for $\mathrm{NIBM}_{0 \to T}$ in \eqref
{eq:defn_Rmelon_special}.
\end{rmk}

Our next task is to find an expression for $\tilde{K}_{t_i, t_j}(x, y)$
which is convenient for analysis. We note that by \eqref
{eq:formulaofPabttau}, \eqref{eq:derivative_of_P_1st_variable} and
\eqref{eq:derivative_of_P_2nd_variable},
%
\begin{eqnarray}
\label{eq:phi_and_psi_in_polynomials} \phi_j(x) &=& \sum_{k \in\intZ+\tau}
f_{j - 1}(k) e^{-{t_1
k^2}/{(2n)}} e^{-ikx},
\nonumber
\\[-8pt]
\\[-8pt]
\nonumber
 \psi_j(x)& =&
\sum_{k \in\intZ+\tau} g_{j -
1}(k) e^{-{(T - t_m)k^2}/{(2n)}}
e^{ikx},
\end{eqnarray}
where $f_i, g_i$ are polynomials of degree exactly $i$ (with possibly
complex coefficients). Note that $W_j(x, y)$ depends only on $x - y$,
and so we can write $W_j(x, y) = h_j(x - y)$. Thus, we see that $W_j$
is a convolution operator,
%
\begin{equation}
(W_j f) (x) = \int^{\pi}_{-\pi}
h_j(x - y)f(y) \,dy =: (h_j * f) (x),
\end{equation}
where by \eqref{eq:W_k_kernel} and \eqref{eq:formulaofPabttau},
%
\begin{equation}
\label{eq:h_hat_formula} h_j(x) = \sum_{k\in\intZ+\tau}
\hat{h}_j(k) e^{ikx} ,\qquad \hat{h}_j(k) =
\frac{1}{2\pi} e^{-{(t_{j + 1} - t_j)k^2}/{(2n)}}.
\end{equation}
Here, and in what follows, we use the notation $\hat{h}(k)$ for the
$k$th coefficient in the \emph{$\tau$-shifted Fourier series}, defined
by the first equation in \eqref{eq:h_hat_formula}. As with the usual
Fourier series, we have that for $i < j$,
%
\begin{equation}
\label{eq:Wi,jFourier} W_{[i, j)}(x, y) = (h_i * h_{i + 1}
* \cdots* h_{j - 1}) (x - y),
\end{equation}
where $h_i * h_{i + 1} * \cdots* h_{j - 1}$ has the $\tau$-shifted
Fourier series
%
\begin{equation}
(h_i * \cdots* h_{j - 1})\sphat (k) = (2
\pi)^{j - i - 1} \prod^{j -
1}_{l = i}
\hat{h}_l(k) = \frac{1}{2\pi} e^{-{(t_j - t_i) k^2}/{(2n)}}.
\end{equation}
Furthermore, as $W_{[i, m)} \Psi$ is an operator from $\ell^2(n)$ to
$L^2(\unitC)$, it is represented by an $n$-dimensional row vector. Its
$l$th component is
%
\begin{equation}\qquad
(W_{[i, m)} \Psi)_l(x) = \int^\pi_{-\pi}
W_{[i, m)}(x, y) \psi _l(y) \,dy = (h_i * \cdots*
h_{m - 1}) * \psi_l(x),
\end{equation}
whose $\tau$-shifted Fourier series is
%
\begin{eqnarray}
&&\bigl(\widehat{(W_{[i, m)} \Psi)_l}\bigr) (k)\nonumber\\
&&\qquad=
\bigl((h_i * \cdots* h_{m - 1}) * \psi _l\bigr)
\sphat (k) = 2\pi(h_i * \cdots* h_{m - 1})\sphat (k) \hat {
\psi }_l(k)
\\
&&\qquad= g_{l - 1}(k) e^{-{(T - t_i)k^2}/{(2n)}}.\nonumber
\end{eqnarray}
Similarly, $\Phi W_{[1, j)}$ is an operator from $L^2(\unitC)$ to
$\ell
^2(n)$, and is then represented by an $n$ dimensional column vector.
Its $l$th component,
%
\begin{equation}
(\Phi W_{[1, j)})_l(x) = \int^{\pi}_{-\pi}
\phi_l(y) W_{[1, j)}(y, x) \,dy,
\end{equation}
satisfies
%
\begin{eqnarray}
(\Phi W_{[1, j)})_l(-x) = \tilde{\phi}_l *
(h_1 * \cdots* h_{j -
1}) (x)
\nonumber
\\[-8pt]
\\[-8pt]
\eqntext{\mbox{where } \tilde{
\phi}_l(x) = \phi_l(-x),}
\end{eqnarray}
and the $\tau$-shifted Fourier series is
%
\begin{eqnarray}
&&\bigl(\widehat{(\Phi W_{[1, j)})_l}\bigr) (-k)\nonumber \\
&&\qquad= \bigl(
\tilde{\phi}_l * (h_1 * \cdots * h_{j - 1})\bigr)
\sphat (k) = 2\pi\hat{\phi}_l(-k) (h_1 * \cdots*
h_{j -
1})\sphat (k)
\\
&&\qquad = f_{l - 1}(k) e^{-{t_j k^2}/{(2n)}}.\nonumber
\end{eqnarray}
Also for the $(i, j)$ entry of the matrix $M$ defined in \eqref
{eq:defn_M}, we have
%
\begin{eqnarray}
\label{eq:M_ij_in_Fourier_series_sum} %
M_{ij} &=& \int^{\pi}_{-\pi}
\int^{\pi}_{-\pi} \phi_i(x)
W_{[1,
m)}(x, y) \psi_j(y) \,dx \,dy
\nonumber\\
&=& (2\pi)^2 \sum_{k \in\intZ+\tau} \hat{
\phi}_j(-k) (h_1 * \cdots * h_{m - 1})\sphat (k)
\hat{\psi}_j(k)
\\
&= &2\pi\sum_{k \in\intZ+\tau} f_{i - 1}(k)
g_{j - 1}(k) e^{-
{Tk^2}/{(2n)}}.\nonumber %
\end{eqnarray}
To simplify the expression of $\tilde{K}_{t_i, t_j}(x, y)$, we fix the
formula \eqref{eq:phi_and_psi_in_polynomials} for $\phi_j(x)$ and
$\psi
_j(x)$ as
%
\begin{equation}
\label{eq:explicit_expression_of_f_g} f_j(k) = g_j(k) = p^{(T; \tau)}_{n, j}
\biggl( \frac{k }{n} \biggr),
\end{equation}
where $p^{(T; \tau)}_{n, j}$ is the discrete Gaussian orthogonal
polynomial defined in \eqref{eq:defn_of_discrete_Gaussian_OP}. Then~\eqref{eq:M_ij_in_Fourier_series_sum} yields
%
\begin{equation}
M_{ij} = %
\cases{ 2\pi n h^{(T; \tau)}_{n, j},
& \quad $\mbox{if $i = j$}$, \vspace*{2pt}
\cr
0, & \quad $\mbox{otherwise}$, } %
\end{equation}
where $h^{(T; \tau)}_{n, j}$ is defined in \eqref{eq:defn_of_h_nk}. Thus,
%
\begin{eqnarray}
\label{eq:corr_kernel} %
\tilde{K}_{t_i, t_j}(x, y)& =& \sum
^{n - 1}_{l = 0} \biggl( \sum
_{k
\in\intZ+\tau} g_l(k) e^{-{(T - t_i)k^2}/{(2n)}} e^{ikx}
\biggr) \nonumber\\
&&{}\times\frac{1}{2\pi n h^{(T; \tau)}_{n, l}} \biggl( \sum_{k \in\intZ
+\tau}
f_l(k) e^{-{t_jk^2}/{(2n)}} e^{-iky} \biggr)
\\
&=& \frac{n}{2\pi} \sum^{n - 1}_{k = 0}
\frac{1}{h^{(T; \tau)}_{n,
k}} S_{k, T - t_i}(x) S_{k, t_j}(-y),\nonumber %
\end{eqnarray}
where
%
\begin{equation}
\label{eq:Sla} S_{k, a}(x) = \frac{1}{n} \sum
_{s \in L_{n, \tau}} p^{(T;\tau)}_{n,k}(s) e^{-{an s^2}/{2}}
e^{ixns}.
\end{equation}
At last, by \eqref{eq:Wi,jwithcirc}, \eqref{eq:h_hat_formula},
\eqref{eq:Wi,jFourier}, we have that
%
\begin{equation}
\label{eq:Wcirc_ij_exact_formula} \WWW_{[i, j)}(x, y) =
\frac{1}{2\pi} \sum
_{s \in L_{n, \tau}} e^{-
{(t_j - t_i)ns^2}/{2} - in(y - x)s}.
\end{equation}

After arriving at a computable formula of $R^{(n)}_{0 \to T}(a^{(1)}_1,
\ldots,
a^{(1)}_{k_1}; \ldots;\break a^{(m)}_1, \ldots,  a^{(m)}_{k_m}; t_1, \ldots,
t_m; \tau)$ defined in \eqref{eq:defn_Rmelon}, we go back to examine
its probabilistic meaning. The special choice that $\tau= \epsilon(n)$
gives us the correlation function of the $\mathrm{NIBM}_{0 \to T}$, namely
%
\begin{eqnarray}
&&R^{(n)}_{0 \to T}\bigl(a^{(1)}_1, \ldots,
a^{(1)}_{k_1}; \ldots; a^{(m)}_1, \ldots,
a^{(m)}_{k_m}; t_1, \ldots, t_m;
\epsilon(n)\bigr)\nonumber\\
&&\qquad = \lim_{\Delta x \to0} \frac
{1}{(\Delta x)^{k_1 + \cdots+ k_m}}
\\
&&\quad\qquad{}\times\Prob \left(
\begin{array}{c}
\mbox{$n$ particles in $\mathrm{NIBM}_{0 \to T}$, there
is a particle in}\\
\mbox{$\bigl[a^{(i)}_j, a^{(i)}_j
+ \Delta x\bigr)$ for $j = 1, \ldots, k_i$ at time
$t_i$}
\end{array} \right).
\nonumber
\end{eqnarray}
Letting $\tau$ vary, the Fourier coefficients of $R^{(n)}_{0 \to
T}
(a^{(1)}_1, \ldots, a^{(1)}_{k_1}; \ldots; a^{(m)}_1, \ldots,
a^{(m)}_{k_m};\break t_1, \ldots, t_m; \tau )$ encode the correlation
functions of particles in $\mathrm{NIBM}_{0 \to T}$ with fixed total
winding number, so that
%
\begin{eqnarray}
\label{eq:sum_of_corr_func_over_winding_number} &&
\frac{R_n(T; \tau)}{R_n(T; \epsilon(n))} R^{(n)}_{0 \to
T}
\bigl(a^{(1)}_1, \ldots, a^{(1)}_{k_1};
\ldots; a^{(m)}_1, \ldots, a^{(m)}_{k_m};
t_1, \ldots, t_m; \tau\bigr)
\nonumber\\
&&\qquad=\sum_{\omega\in\intZ} \bigl(R^{(n)}_{0 \to T}
\bigr)_{\omega}\bigl(a^{(1)}_1, \ldots,
a^{(1)}_{k_1}; \ldots; a^{(m)}_1, \ldots,
a^{(m)}_{k_m} t_1, \ldots, t_m\bigr)\\
&&\qquad\quad{}\times e^{2\pi o(\tau+ \epsilon(n))i},
\nonumber
\end{eqnarray}
where
%
\begin{eqnarray}
&&\bigl(R^{(n)}_{0 \to T}\bigr)_{\omega}
\bigl(a^{(1)}_1, \ldots, a^{(1)}_{k_1};
\ldots; a^{(m)}_1, \ldots, a^{(m)}_{k_m};
t_1, \ldots, t_m\bigr)\nonumber\\
&&\qquad =
\lim_{\Delta x \to0} \frac{1}{(\Delta x)^{k_1 + \cdots+ k_m}}\\
&&\qquad\quad{}\times \Prob \left(
\begin{array}{c}
\mbox{$n$ particles in $\mathrm{NIBM}_{0 \to T}$ with
total}\\
\mbox{winding number $\omega$, there is a particle
in}\\
\mbox{ $\bigl[a^{(i)}_j, a^{(i)}_j + \Delta x
\bigr)$ for $j = 1, \ldots, k_i$ at time $t_i$}
\end{array}
 \right).
\nonumber
\end{eqnarray}

\section{Asymptotic results for discrete Gaussian orthogonal
polynomials} \label{sec:asymptotics}

In this section, we state the asymptotic results for the discrete
Gaussian orthogonal polynomials \eqref{eq:defn_of_discrete_Gaussian_OP}
which will be used in Sections~\ref{sec:winding_number} and \ref
{sec:correlation}. The results are derived from the interpolation
problem and the corresponding Riemann--Hilbert problem associated with
the discrete orthogonal polynomials, and the proofs are outlined in
Section~\ref{sec:Riemann-Hilbert} unless otherwise stated.

\subsection{The equilibrium measure and the $g$-function}\label
{results_eq_measure}
A key ingredient in the Riemann--Hilbert analysis of orthogonal
polynomials is the equilibrium measure associated with the weight
function. The equilibrium measure associated with the weight
$e^{-nTx^2/2}$ for the discrete Gaussian orthogonal polynomials defined
on the lattice $L_{n, \tau}$ is the unique probability measure which
minimizes the functional,
%
\begin{equation}
\label{eq1} H(\nu)=\int\int\log\frac{1}{|x-y|}\,d\nu(x)\,d\nu(y)+\int
\frac{Tx^2}{2} \,d\nu(x),
\end{equation}
over the set of probability measures $\nu$ on $\R$ satisfying
%
\begin{equation}
\label{eq:constraint} \,d\nu(x) \leq \,dx,
\end{equation}
where $dx$ denotes the differential with respect to Lebesgue measure.
It is well known [\citet{Kuijlaars00}] that there is a unique solution to
\eqref{eq1} satisfying \eqref{eq:constraint}, and we call it the
\emph
{equilibrium measure} for the discrete Gaussian orthogonal polynomials.
The upper constraint \eqref{eq:constraint} implies that the equilibrium
measure is absolutely continuous with respect to Lebesgue measure and,
therefore, has an associated density. Let us denote this density by
$\rho_T(x)$.

We define the $g$-function associated with the discrete Gaussian
orthogonal polynomials as the log transform of the equilibrium measure:
%
\begin{equation}
\label{g1} g(z) := \int\log(z-x) \rho_T(x)\,dx,
\end{equation}
where we take the principal branch for the logarithm. Then the
Euler--Lagrange variational conditions for the equilibrium problem
\eqref{eq1} are
%
\begin{equation}
\label{eq:variational_condition} g_+(x) + g_-(x) - \frac{Tx^2}{2} - l %
\cases{ = 0,
&\quad $\mbox{if $0 < \rho_T(x) < 1$}$, \vspace*{2pt}
\cr
\leq0, &\quad $
\mbox{if  $\rho_T(x) = 0$}$, \vspace*{2pt}
\cr
\geq0, &\quad  $\mbox{if $
\rho_T(x) = 1$}$, } %
\end{equation}
where $g_+$ and $g_-$ refer to the limiting values from the upper and
lower half-planes, respectively, and $l\in\R$ is a constant Lagrange
multiplier. Since the external potential $Tx^2/2$ is convex and even,
the equilibrium measure is supported on a single interval $[-\be,\be]$.
We have for all $x \in(-\infty, \beta)$,
%
\begin{equation}
\label{eq:differencegggg} g_+(x) - g_-(x) = 2\pi i \int^{\beta}_x
\rho_T(x).
\end{equation}

Without the upper constraint \eqref{eq:constraint}, it is well known
that the solution $\nu_T$ to the minimization problem \eqref{eq1} is
given by the Wigner semicircle law [\citet{Deift99}, Section~6.7]. That
is, $\nu_T$ is supported on a single interval $[-\beta, \beta]$ and
%
\begin{eqnarray}
\label{eq2} d\nu_T(x) = \rho_T(x) {\bolds
\chi}_{[-\beta, \beta]}(x) \,dx
\nonumber
\\[-8pt]
\\[-8pt]
\eqntext{\mbox {where } \displaystyle\beta= \frac{2}{\sqrt{T}},
\rho_T(x) = \frac
{T}{2\pi} \sqrt{\frac{4}{T}-x^2}.}
\end{eqnarray}
Clearly, this $\rho_T(x)$ has its maximum value at $x=0$ and $\rho
_T(0)=\sqrt{T}/\pi$. It follows that \eqref{eq2} satisfies the
variational problem \eqref{eq1} with constraint \eqref{eq:constraint}
if and only if $0<T\le\pi^2$. We therefore denote the critical value
$T_c := \pi^2$ as in \eqref{eq:first_defn_T_c}, and we have the
following proposition.

\begin{prop}
For $T \leq T_c = \pi^2$, the equilibrium measure for the discrete
Gaussian orthogonal polynomials is given by the Wigner semicircle law
\eqref{eq2}.
\end{prop}

For $T > T_c$, the probability measure given by the Wigner semicircle
law \eqref{eq2} does not satisfy the constraint \eqref{eq:constraint}.
In this case the equilibrium measure is still supported on a single
interval $[-\be, \be]$, but now there is a \emph{saturated region}
$[-\al, \al]$, where $0<\al<\be$, on which the density $\rho_T(x)$ is
identically 1. Since $\rho_T(x)$ is an even function and has total
integral $1$, \eqref{eq:differencegggg} then implies that for $x
\in(-\alpha, \alpha)$ we have
%
\begin{equation}
\label{eq:differencegcentralll} g_+(x) - g_-(x) = i\pi- 2\pi i x.
\end{equation}
To present the solution of the minimization problem \eqref{eq1} and
\eqref{eq:constraint}, we introduce a parameter $k \in(0, 1)$ and use
elliptic integrals with parameter $k$, defined as
%
\begin{equation}
F(z; k) = \int^z_0 \frac{ds}{\sqrt{(1 - s^2)(1 - k^2 s^2)}},\qquad  E(z;
k) = \int^z_0 \frac{\sqrt{1 - k^2 s^2}}{\sqrt{1 - s^2}} \,ds.\hspace*{-20pt}
\label{eq:defn_of_K} 
\end{equation}
In the definitions of $F(z; k)$ and $E(z; k)$ we assume $z \in\compC
\setminus\{(-\infty, 1) \cup(1, \infty)\}$. We also use the complete
elliptic integrals $\mathbf K$ and $\mathbf E$ defined in \eqref
{eq:complete_elliptic}. Given any $k\in(0,1)$, we express the
endpoints of the support and saturated region of the equilibrium
measure $\alpha$ and $\beta$ as
%
\begin{equation}
\label{eq:defn_of_beta_alpha} \beta= \beta(k) = \bigl(2\E- \bigl(1 - k^2\bigr)\K
\bigr)^{-1},\qquad \alpha= \alpha(k) = k\beta(k).
\end{equation}
Note that by \citet{Erdelyi-Magnus-Oberhettigner-Tricomi81}, Table~4 on page
319, and notation defined in~\eqref{eq:def_tilde2},
%
\begin{eqnarray}
\label{eq:Gauss_transform_KE} \tilde{\K}& =& \K \biggl( \frac{2\sqrt{k}}{1 + k} \biggr) = (1 + k) \K
(k),
\nonumber
\\[-8pt]
\\[-8pt]
\nonumber
\tilde{\E} &=& \E \biggl( \frac{2\sqrt{k}}{1 + k} \biggr) = \frac{2\E(k) - (1 - k^2)\K(k)}{1 + k},
\end{eqnarray}
and so we have
%
\begin{equation}
\label{eq:alternative_beta} \beta= \frac{1}{(1 + k)\tilde{\E}}.
\end{equation}
Using \eqref{eq:Gauss_transform_KE}, we parametrize $T$ by $k$ as in
\eqref{eq:parametrization_of_T_intro},
%
\begin{equation}
\label{eq:T_parametrized_by_k} T = T(k) = 4\K\beta^{-1} = 4 \tilde{\K} \tilde{\E}.
\end{equation}

By the following lemma, the parametrization is well defined.

\begin{lem} \label{lem:T_parametrized_by_k}
$\K(k) \E(k)$ is a strictly increasing function of $k \in[0,1)$ and
%
\begin{equation}
\label{eq:limitsofTk} \lim_{k \to0_+} \K(k) \E(k) = T_c =
\pi^2, \qquad \lim_{k \to1} \K(k) \E(k) = +\infty.
\end{equation}
\end{lem}

Now we can state the result of the equilibrium measure for $T > T_c$.

\begin{prop}
For $T > T_c = \pi^2$, $T = T(k)$ is parametrized by $k \in(0, 1)$ as
in \eqref{eq:T_parametrized_by_k}, and the equilibrium measure for the
discrete Gaussian orthogonal polynomials is supported on a single
interval $[-\beta, \beta]$ with a saturated region $[-\al, \al]$ where
$\beta= \beta(k)$ and $\al=\al(k)$ are defined in \eqref
{eq:defn_of_beta_alpha}. The density $\rho_T(x)$ for the equilibrium
measure is given by the formula
%
\begin{equation}
\label{rho_T_formula} \rho_T(x) = %
\cases{ 1, &\quad $\mbox{if $x
\in[-\alpha, \alpha]$,}$ \vspace*{2pt}
\cr
\displaystyle\frac{2}{\pi\alpha} \biggl[ \E\int
^{\beta}_x \frac
{ds}{\sqrt{(\alpha^{-2} s^2 - 1)(1 - \beta^{-2} s^2)}}&\vspace*{2pt}\cr
\hspace*{58pt}{} - \displaystyle\K\int
^{\beta
}_x \frac{\sqrt{1 - \beta^{-2} s^2}}{\sqrt{\alpha^{-2} s^2 - 1}} \,ds \biggr], &\quad $
\mbox{if $x \in(\alpha, \beta)$,}$ \vspace*{2pt}
\cr
\rho_T(-x) ,& \quad $
\mbox{if $x \in(-\beta, -\alpha)$,}$ \vspace *{2pt}
\cr
0 ,&\quad  $\mbox{otherwise}$. }\hspace*{-40pt}
\end{equation}
\end{prop}

Note that for $x \in(\alpha, \beta)$, using formulas
\citet{Gradshteyn-Ryzhik07},
3.152-10,
page 280 and 3.169-17, page 309 and \citet{Byrd-Friedman71},
413.01, page~228, $\rho_T(x)$ can be expressed in a more compact form
%
\begin{eqnarray}
\label{eq:alternative_rho_T} %
\rho_T(x)& =& \frac{2}{\pi} \biggl[ (
\E- \K)F \biggl( \sqrt{\frac{1 -
{x^2}/{\beta^2}}{1 - k^2}}; k' \biggr) + \K E
\biggl( \sqrt {\frac{1 -
{x^2}/{\beta^2}}{1 - k^2}}; k' \biggr) \biggr]
\nonumber\\
&=& \Lambda_0 \biggl( \sqrt{\frac{1 - {x^2}/{\beta^2}}{1 - k^2}}; k \biggr)
\\
&= &\frac{2}{\pi\beta x} \sqrt{\bigl(\beta^2 - x^2\bigr)
\bigl(x^2 - \alpha^2\bigr)} \Pi _1 \biggl( -
\frac{\alpha^2}{x^2}, k \biggr),\nonumber %
\end{eqnarray}
where $k' = \sqrt{1 - k^2}$, $\Lambda_0(x; k)$ is the Heuman's Lambda
function [see \citet{Byrd-Friedman71}, 150.03, page 36, and note that our
$x$ corresponds to $\sin\beta$ in \citet{Byrd-Friedman71},
150.03, page
36], and the $\Pi_1$ denotes the complete elliptic
integral of the third kind [in the notational
conventions of \citet{Erdelyi-Magnus-Oberhettigner-Tricomi81},
Section~13.8 (3), page 317],
%
\begin{equation}
\Pi_1(\nu, k) = \int^1_0
\frac{dx}{(1 + \nu x^2) \sqrt{(1 - x^2)(1 -
k^2 x^2)}}.
\end{equation}
The formulas \eqref{eq:alternative_rho_T} have appeared several times
in the physics literature in the context of Yang--Mills theory [\citet
{Douglas-Kazakov93}, \citet{Gross-Matytsin95}].

In our asymptotic analysis of $\mathrm{NIBM}_{0 \to T}$, the function
$g(z)$ defined in
\eqref{g1} plays an important role. In particular, we must use the
derivative of this function to locate critical points. The following
proposition gives an explicit formula of $g'(z)$.

\begin{prop}
For $T \leq T_c$,
%
\begin{equation}
\label{g13} g'(z)=\frac{T}{2} \biggl(z-
\sqrt{z^2-\frac{4}{T}} \biggr),
\end{equation}
and for $T > T_c$,
%
\begin{eqnarray}
\label{eq:formulagsuperrr} %
 g'(z)& =& 2 \biggl[
\frac{\K}{\beta}z - \frac{\K}{\alpha} \int^z_0
\frac{\sqrt{1 - \beta^{-2} s^2}}{\sqrt{1 - \alpha^{-2} s^2}} \,ds \nonumber\\
&&\hspace*{11pt}{}+ \frac
{\E}{\alpha} \int^z_0
\frac{ds}{\sqrt{(1 - \alpha^{-2} s^2)(1 -
\beta
^{-2} s^2)}} \mp\frac{\pi i}{2} \biggr]
\nonumber
\\[-8pt]
\\[-8pt]
\nonumber
&= &2 \biggl[ \frac{\K}{\beta}z - \K E \biggl( \frac{z}{\alpha}; k \biggr) +
\E F \biggl( \frac{z}{\alpha}; k \biggr) \mp\frac{\pi i}{2} \biggr],
\\
\eqntext{\mbox{for $\pm\Im z >0$}.}
\end{eqnarray}
\end{prop}

Note that $g(z)$ is single valued on $(\beta, +\infty)$. This is clear
in \eqref{g13}, and we may write \eqref{eq:formulagsuperrr} in the form
%
\begin{equation}
\label{eq:formulagsuperrr_alt} g'(z) =\frac{2\K z}{\be} -2\E\be\int
_{\be}^z \frac{ds}{\sqrt
{(s^2-\al
^2)(s^2-\be^2)}} -
\frac{2\K}{\be}\int_\be^z
\frac{\sqrt
{s^2-\be
^2}}{\sqrt{s^2-\al^2}} \,ds ,\hspace*{-20pt}
\end{equation}
where the square roots are positive for $s>\be$ and have cuts on
$(-\be
,-\al) \cup(\al,\be)$.

With the notation defined in this section, we rewrite $t^c$ defined in
\eqref{eq:T_c_in_introduction} for the supercritical case of $\mathrm
{NIBM}_{0 \to T}$ as
[by \eqref{eq:differencegcentralll}, $g''(z)$ is well defined in a
neighborhood of $0$]
%
\begin{eqnarray}
\label{eq:defn_T_c:elliptic} %
t^c &:=& g''(0)
= \frac{T}{2} - \frac{2}{\alpha}(\K- \E) = \frac
{2}{\alpha}\bigl(\E-
(1 - k)\K\bigr)
\nonumber
\\[-8pt]
\\[-8pt]
\nonumber
&=& \frac{(1 + k)^2}{k} \E \biggl( \frac{2\sqrt{k}}{1 + k} \biggr) \biggl( \E \biggl(
\frac{2\sqrt{k}}{1 + k} \biggr) - \biggl( \frac
{1 -
k}{1 + k} \biggr)^2 \K
\biggl( \frac{2\sqrt{k}}{1 + k} \biggr) \biggr).
\end{eqnarray}

The formulas \eqref{g13} and \eqref{eq:formulagsuperrr} can be
integrated to obtain expressions for $g(z)$, where the constant of
integration is determined by the condition $g(z)\sim\log(z)$ as $z\to
\infty$. Then the Lagrange multiplier $l$ in \eqref
{eq:variational_condition} can be determined from the equality in
\eqref
{eq:variational_condition}. Although they are not indispensable in this
paper, for completeness we present the formulas for $g(z)$ and $l$
below. In the subcritical case $0<T<T_c = \pi^2$, explicit calculations
give that
%
\begin{eqnarray}
\label{g12} g(z)& =& \frac{T}{4} z \biggl( z - \sqrt{z^2 -
\frac{4}{T}} \biggr) - \log \biggl( z - \sqrt{z^2 -
\frac{4}{T}} \biggr)\nonumber\\
&&{} - \frac{1}{2} + \log2 - \log T\quad \mbox{and}\\
e^l&=&\frac{1}{Te}.\nonumber
\end{eqnarray}
In the supercritical case $T>T_c$, we present the formula for $g(z)$
and the Lagrange multiplier in the following proposition.

\begin{prop}
For $T > T_c = \pi^2$ the function $g(z)$ is given by
%
\begin{eqnarray}
\label{eq:formula_gfunction_supercrit} g(z)&=&zg'(z)-\frac{\K z^2}{\be}+
\frac{\K}{\be}\sqrt{
\bigl(z^2-\be^2\bigr) \bigl(z^2-
\al^2\bigr)}\nonumber
\\[-8pt]
\\[-8pt]
\nonumber&&{}+\log \bigl(\sqrt{z^2-\be^2}+
\sqrt{z^2-\al^2} \bigr)
+\frac{\K\be
}{2}
\bigl(1+k^2\bigr)-1-\log 2 ,
\end{eqnarray}
where $g'(z)$ is as in \eqref{eq:formulagsuperrr_alt} and the
principal branches are taken for the square roots and logarithms.
The Lagrange multiplier $l$ in the Euler--Lagrange variational
conditions \eqref{eq:variational_condition} is given by
%
\begin{equation}
\label{eq:formula_Lagrange_mult_supercrit} l=\log\bigl(\be^2-\al^2\bigr)+\K\be
\bigl(1+k^2\bigr)-2(1+\log2).
\end{equation}
\end{prop}

\begin{pf}
Using integration by parts, we have
%
\begin{equation}
g(z)=zg'(z)-\int zg''(z) \,dz +
\mbox{const.}
\end{equation}
The second term in this formula can be integrated directly using \eqref
{eq:formulagsuperrr_alt}, and this determines $g(z)$ up to the
constant term, which is obtained by the condition $g(z) \sim\log(z)$
as $z\to\infty$. This proves \eqref{eq:formula_gfunction_supercrit}. To
obtain \eqref{eq:formula_Lagrange_mult_supercrit}, we use \eqref
{eq:variational_condition} at $x=\be$, which implies
%
\begin{equation}
l=2g(\be)-\frac{T\be^2}{2} ,
\end{equation}
which we evaluate using \eqref{eq:formula_gfunction_supercrit}.
\end{pf}

\subsection{Asymptotics of the discrete Gaussian orthogonal polynomials}

We now summarize the asymptotics of the discrete Gaussian
orthogonal polynomials \eqref{eq:defn_of_discrete_Gaussian_OP} and
their discrete Cauchy transforms used in this paper. For a real
function $f(x)$, define its discrete Cauchy transform $Cf$ on the
weighted lattice $L_{n,\tau}$ as
%
\begin{equation}
\label{eq:def_Cauchy_trans} Cf(z):= \frac{1}{n}\sum_{x\in L_{n,\tau}}
\frac{f(x) e^{-
({nT}/{2})x^2}}{z-x}.
\end{equation}
In the subcritical case $T<T_c$, the discrete Gaussian orthogonal
polynomials are exponentially close, as $n\to\infty$, to the rescaled
Hermite polynomials, for which there are exact formulas. To present the
asymptotics in the supercritical case, we first fix some notation.
Define the function
%
\begin{equation}
\label{m4} \gamma(z) := \biggl(\frac{(z+\be)(z-\al)}{(z-\be)(z+\al)} \biggr)^{1/4},
\end{equation}
with a cut on $[-\be,-\al]\cup[\al,\be]$, taking the branch such that
$\gamma(z) \sim1$ as $z\to\infty$.
Recall the elliptic nome $q$ defined in \eqref
{eq:def_elliptic_nome_intro} for $T>T_c$. We will use the Jacobi theta
functions with elliptic nome $q$,
%
\begin{eqnarray}
\label{eq:def_Jacobi_theta} \th_3(z) &:=& \th_3(z; q)= 1+2\sum
_{j=1}^\infty q^{j^2} \cos(2 j z) ,
\nonumber
\\[-8pt]
\\[-8pt]
\nonumber
\th_4(z) &:= &\th_4(z; q)= 1+2\sum
_{j=1}^\infty(-1)^j q^{j^2}
\cos (2 j z).
\end{eqnarray}
We will also use the notation $\tilde{k}$, $\tilde{\mathbf K}$ and
$\tilde
{\mathbf E}$ defined in \eqref{eq:def_tilde1} and \eqref{eq:def_tilde2}, as
well as the function
%
\begin{equation}
\label{m5} u(z):= \frac{\pi(\al+\be)}{4\tilde{\mathbf K}} \int_\be^z
\frac
{dx}{\sqrt{(x^2-\al^2)(x^2-\be^2)}}.
\end{equation}

Fix some $0\le\delta<1$ and $\ep>0$. Define the domain $D(\delta
,\ep,n)$ as
%
\begin{equation}
\label{eq:defnDden} D(\delta,\ep,n) = \bigl\{z \mid\vert z \pm\al\vert>\ep, \vert z
\pm \be\vert>\ep, \vert\Im z \vert> \ep n^{-\delta}\bigr\}.
\end{equation}
We then have the following proposition which describes the asymptotics
of the discrete Gaussian orthogonal polynomials on the domain $D(\delta
,\ep,n)$.

\begin{prop}\label{asymptotics_of_OPs}
For any $T > T_c$, as $n \to\infty$, the discrete Gaussian orthogonal
polynomials \eqref{eq:defn_of_discrete_Gaussian_OP} satisfy
%
\begin{eqnarray}\label{asub1}
p^{(T; \tau)}_{n, n}(z)& =& e^{ng(z)} M_{11}(z)
\bigl(1+\Er_{11}(n,z)\bigr),
\nonumber
\\[-8pt]
\\[-8pt]
\nonumber
\label{asub1a}
\frac{p^{(T; \tau)}_{n, n-1}(z)}{h^{(T; \tau)}_{n, n-1}}& = &e^{n(g(z)-l)}
M_{21}(z) \bigl(1+\Er_{21}(n,z)\bigr),
\\
\bigl(Cp^{(T; \tau)}_{n, n} \bigr) (z)& =& e^{-n(g(z)-l)}
M_{12}(z) \bigl(1+\Er _{12}(n,z)\bigr),
\nonumber
\\[-8pt]
\\[-8pt]
\nonumber
\frac{ (Cp^{(T; \tau)}_{n, n} )(z)}{h^{(T;
\tau)}_{n, n-1}}
&=& e^{-ng(z)} M_{22}(z) \bigl(1+\Er_{22}(n,z)\bigr),
\end{eqnarray}
where
%
\begin{eqnarray}
M_{11}(z)& =& \frac{1}{2} \biggl(\gamma(z)+\frac{1}{\gamma(z)}
\biggr)\frac
{\th_3(0)\th_3(u(z)-\pi/4-\pi(\tau-\epsilon(n)))}{\th_3(\pi
(\tau
-\epsilon(n)))\th_3(u(z)-\pi/4)}, \label{eq:M_11_formula_intro}
\\
M_{21}(z) &= &\frac{1}{4\pi} \biggl(\gamma(z)-\frac{1}{\gamma
(z)}
\biggr)\frac
{\th_3(0)\th_3(u(z)+\pi/4-\pi(\tau-\epsilon(n)))}{\th_3(\pi
(\tau
-\epsilon(n)))\th_3(u(z)+\pi/4)}, \label{eq:M_21_formula_intro}
\\
M_{12}(z)& =& \pi \biggl(\gamma(z)-\frac{1}{\gamma(z)} \biggr)
\frac
{\th
_3(0)\th_3(u(z)+\pi/4+\pi(\tau-\epsilon(n)))}{\th_3(\pi(\tau
-\epsilon(n)))\th
_3(u(z)+\pi/4)}, \label{eq:M_12_formula_intro}
\\
M_{22}(z) &= &\frac{1}{2} \biggl(\gamma(z)+\frac{1}{\gamma(z)}
\biggr)\frac
{\th_3(0)\th_3(u(z)-\pi/4+\pi(\tau-\epsilon(n)))}{\th_3(\pi
(\tau
-\epsilon(n)))\th_3(u(z)-\pi/4)}. \label{eq:M_22_formula_intro}
\end{eqnarray}
These asymptotics are uniform in $\tau$ and for $z\in D(\delta,\ep
,n)$ in
the following sense. There exists a constant $C(\ep)>0$ such that for
each $0<\delta<1$, the errors in \eqref{asub1} and \eqref{asub1a} satisfy
%
\begin{equation}
\sup_{z\in D(\delta,\ep,n)} \bigl\vert\Er_*(n,z) \bigr\vert< C(\ep) n^{-(1-\delta
)}\qquad
\mbox{where } * = 11, 21, 12, 22.
\end{equation}
\end{prop}

A similar result with a weaker error holds in the critical case $T=T_c
+ \bigO(n^{-2/3})$. We have the following proposition.

\begin{prop} \label{prop:asymptotics_of_OPs_critical}
Fix $\ep>0$ and $0\le\delta<1/3$. For $T=T_c(1-2^{-2/3}\sigma n^{-2/3})$,
the discrete Gaussian orthogonal polynomials \eqref
{eq:defn_of_discrete_Gaussian_OP} satisfy the asymptotics \eqref{asub1}
in the domain $\{ z \mid\vert z \pm\beta\vert> \varepsilon, \vert
\Im z \vert> \ep n^{-\delta}\}$, where the function $g(z)$ is defined in~\eqref{g12}, the functions $M_{11}(z)$ and $M_{21}(z)$ are given by
%
\begin{eqnarray}
M_{11}(z) = \frac{1}{2} \biggl( \gamma(z) + \frac{1}{\gamma(z)}
\biggr), \qquad M_{21}(z) = \frac{1}{4\pi} \biggl( \gamma(z) -
\frac{1}{\gamma(z)} \biggr)
\nonumber
\\[-8pt]
\\[-8pt]
\eqntext{ \mbox{where } \displaystyle\gamma(z) = \biggl( \frac{z + \beta}{z -
\beta}
\biggr)^{{1}/{4}},}
\end{eqnarray}
such that $\beta$ is defined as in \eqref{eq2} and $\gamma$ is defined
with a cut $[-\beta, \beta]$ and the branch $\gamma(z) \sim1$ as $z
\to\infty$. The errors $\Er_{11}(n,z), \Er_{21}(n,z)$ are of the order
$n^{-(1/3 -\delta)}$.
\end{prop}

In the critical case, the asymptotic formulas for the discrete Gaussian
orthogonal polynomials close to the origin are described in terms of
the matrix function $\bolds\Psi(\zeta, s)$ defined in \eqref{cr4} and
\eqref{cr5}. We do not describe these asymptotics in full generality,
but do give the following formula for the Christoffel--Darboux kernel
in a small neighborhood of the origin and a rough estimate of the
orthogonal polynomials.

\begin{prop} \label{prop:asymptotics_of_OPs_origin}
Fix $\ep>0$ and $0 < \delta<1/3$, and let $T=T_c(1-\break 2^{-2/3}\sigma
n^{-2/3})$. For all $z,w\in\{z \in\C\vert |z|<\ep n^{-\delta}\}$ the
following asymptotic formula holds:
%
\begin{eqnarray}
\label{critical_CD_kernel}&&
e^{-({nT}/{4})(z^2+w^2)} \frac{p^{(T; \tau)}_{n, n}(z)p^{(T; \tau
)}_{n, n-1}(w)-p^{(T; \tau)}_{n, n-1}(z)p^{(T; \tau)}_{n, n}(w)}{h^{(T;
\tau)}_{n, n-1}(z-w)} \nonumber\\
&&\qquad=
\frac{1}{2\pi i (z-w)} %
\pmatrix{ -e^{-i\pi(nz-\tau)}\vspace*{2pt}
\cr
e^{i\pi(nz-\tau)} } %
^T \bolds\Psi\bigl(d n^{{1}/{3}}
z ; \sigma\bigr)^{-1}\bolds\Psi\bigl(d n^{
{1}/{3}} w ; \sigma
\bigr) \\
&&\qquad\quad{}\times
\pmatrix{ e^{i\pi(nw-\tau)} \vspace*{2pt}
\cr
e^{-i\pi(nw-\tau)} }
 \bigl(1+\bigO\bigl(n^{-(1/3-\delta)}\bigr) \bigr),\nonumber
\end{eqnarray}
where $d = 2^{-5/3}\pi$ is defined in \eqref{eq:defn_d_intro}. Also the
following estimate holds uniformly in $\{z \in\C\vert |z|<\ep
n^{-\delta}\}$:
%
\begin{equation}
\label{eq:rough_estimate_critical_zero} p^{(T; \tau)}_{n, n}(z) = \bigO
\bigl(e^{ng(z)}\bigr), \qquad \frac{p^{(T; \tau
)}_{n, n - 1}(z)}{h^{(T; \tau)}_{n, n - 1}} = \bigO\bigl(e^{n(g(z) - l)}
\bigr).
\end{equation}
\end{prop}

Proposition~\ref{prop:asymptotics_of_OPs_origin} follows from the
Riemann--Hilbert analysis of \citet{Liechty12}. The fomula \eqref
{critical_CD_kernel} appears in a slightly different form in
\citet{Liechty12},
equation~(6.12). 

We will also need asymptotic results for the discrete Gaussian
orthogonal polynomials on $\R$ outside of the support of the
equilibrium measure. The following proposition extends the asymptotics
of Proposition~\ref{asymptotics_of_OPs} to this region. The Cauchy
transforms in \eqref{asub1a} have poles on $L_{n,\tau}$, so we must
exclude the points in this lattice from the formulation of the
asymptotic result. Define the regions
%
\begin{eqnarray}
\label{eq:defnEeandEent} E(\ep) &= &\bigl\{(-\infty, -\be-\ep] \cup[\be+\ep, \infty
)\bigr\} \times [-i\ep , i\ep],
\nonumber
\\[-8pt]
\\[-8pt]
\nonumber
E(\ep; n, \tau) &=& E(\ep){}\Big \backslash{}\bigcup
_{x \in
L_{n, \tau}} \biggl\{ z \Big| \vert z-x \vert<
\frac{\ep}{n} \biggr\}.
\end{eqnarray}
Then we have a result parallel to Proposition~\ref{asymptotics_of_OPs}.

\begin{prop} \label{prop:asymptotics_of_OPs_outside}
Fix $\ep>0$. Then the asymptotics \eqref{asub1} are valid on $E(\ep)$,
and the asymptotics \eqref{asub1a} are valid on $E(\ep; n, \tau)$. In
both cases, the errors are of the order $n^{-1}$.
\end{prop}

The functions $M_{11}(z), M_{21}(z), M_{12}(z)$, and $M_{22}(z)$ in
Proposition~\ref{asymptotics_of_OPs} are entries of the $2 \times2$
matrix 
$ \bigl(
{
{1 \atop 0}\enskip
{0 \atop-2\pi i}}
 \bigr)^{-1}
\mathbf M(z) \bigl ({
{1 \atop 0}\enskip
{0 \atop-2\pi i}}
 \bigr)$ as in \eqref{m1}, where $\mathbf M(z)$ in
defined in Section~\ref{sec:Model_RHP}; see formula \eqref{eq314}. By the
Riemann--Hilbert problem satisfied by $\mathbf M(z)$, we have that
$\det\mathbf M(z) = 1$, and so
%
\begin{equation}
\label{eq:Mijdet1} M_{11}(z) M_{22}(z) - M_{12}(z)
M_{22}(z) = 1 ,
\end{equation}
for all $z$ where they are defined. The jump condition for the $2
\times2$ matrix Riemann--Hilbert problem for $\mathbf M(z)$ given in
Section~\ref{sec:Model_RHP} implies that for $x \in(-\alpha, \alpha)$,
%
\begin{eqnarray}
\label
{eq:jump_M_11_21}(M_{11})_+(x) &= &(M_{11})_-(x) e^{2\pi i(\tau+ \epsilon(n))},
\nonumber
\\[-8pt]
\\[-8pt]
\nonumber
(M_{21})_+(x) &= &(M_{21})_-(x) e^{2\pi i(\tau+ \epsilon(n))},
\\
\label{eq:jump_M_12_22}(M_{12})_+(x) &= &(M_{12})_-(x) e^{-2\pi i(\tau+ \epsilon(n))},
\nonumber
\\[-8pt]
\\[-8pt]
\nonumber
(M_{22})_+(x) &=& (M_{22})_-(x) e^{-2\pi i(\tau+ \epsilon(n))}.
\end{eqnarray}

We now summarize the asymptotic formulas for the recurrence
coefficients and the normalizing constants. In \eqref{gamma_super}, we
use the Jacobi elliptic function $\dn(u,\tilde k)$; see, for example,
\citet{Watson-Whittaker96}.

\begin{prop}\label{asymptotics_normalizing_constants} As $n \to\infty$
the recurrence coefficients $ (\gamma^{(T; \tau)}_{n, n}
)^2$ in
\eqref{eq:three_term_recurrence} satisfy the following asymptotic formulas:
\begin{longlist}[(a)]
\item[(a)] In the subcritical case $T<T_c = \pi^2$,
%
\begin{equation}
\label{gamma_sub} \bigl(\gamma^{(T; \tau)}_{n, n}
\bigr)^2=\frac{1}{T} +\bigO \bigl(e^{-cn}\bigr) ,
\end{equation}
where $c>0$ is a constant which depends on $T$.
\item[(b)] In the critical case $T=T_c(1-2^{-2/3}\sigma n^{-2/3})$, as $n
\to
\infty$,
%
\begin{eqnarray}
\label{gamma_critical} &&\bigl(\gamma^{(T; \tau)}_{n, n}
\bigr)^2 = \frac{1}{T} \biggl(1-\frac
{2^{5/3}}{n^{1/3}}q(\sigma)
\cos \bigl(2\pi\bigl(\tau+\epsilon(n)\bigr) \bigr)
\nonumber
\\[-8pt]
\\[-8pt]
\nonumber
&&\hspace*{70pt}{}+\frac
{2^{4/3}}{n^{2/3}}q(
\sigma)^2\cos(4\pi\tau)+\bigO\bigl(n^{-1}\bigr) \biggr).
\end{eqnarray}
\item[(c)] In the supercritical case $T>T_c = \pi^2$,
%
\begin{eqnarray}
\bigl(\gamma^{(T; \tau)}_{n, n} \bigr)^2 =
\frac{\dn^2(2\tilde
{\mathbf K}(\tau+1/2+\epsilon(n)),\tilde{k})}{4\tilde{\mathbf E}^2} +\bigO \bigl(n^{-1}\bigr). \label{gamma_super}
\end{eqnarray}
\end{longlist}
\end{prop}

The formula \eqref{gamma_sub} states that in the subcritical case, the
recurrence coefficients are exponentially close as $n\to\infty$ to the
recurrence coefficients for the rescaled Hermite polynomials; see, for
example, \citet{Liechty12}, Appendix B. The asymptotic formula \eqref
{gamma_critical} was proved in \citet{Liechty12}, and formula
\eqref
{gamma_super} follows from the Riemann--Hilbert analysis presented in
Section~\ref{sec:Riemann-Hilbert}.

\section{Distribution of winding numbers} \label{sec:winding_number}

In this section, we prove Theorem~\ref{winding_number_theorem}. For the
proof of this theorem, we will use the formulas \eqref{Rno_def} and
\eqref{eq:R_n_ratio}. They state that the total winding number for $n$
particles in $\mathrm{NIBM}_{0 \to T}$ is given by the formula
%
\begin{equation}
\label{winu1} \mathbb{P}(\mbox{Total winding number equals }
\omega)=e^{2\pi
i\omega
\epsilon(n)} \int_0^1
\frac{R_n(T;\tau) e^{-2\pi i \omega\tau
}}{R_n(T;\epsilon(n))} \,d\tau,\hspace*{-30pt}
\end{equation}
which according to \eqref{eq:defnRttau} is
%
\begin{eqnarray}
\label{winu2} %
\mathbb{P}(\mbox{Total winding
number equals } \omega)&=&e^{2\pi i
\omega
\epsilon(n)} \int_0^1
\frac{\Hankel_n(T;\tau) e^{-2\pi i \omega
\tau
}}{\Hankel
_n(T;\epsilon(n))} \,d\tau
\nonumber
\\[-8pt]
\\[-8pt]
\nonumber
&=&\int_0^1 \frac{\Hankel_n(T;\tau-\epsilon(n)) e^{-2\pi i \omega
\tau
}}{\Hankel_n(T;\epsilon(n))} \,d\tau.
\end{eqnarray}
In order to evaluate this integral, we will use the following
deformation equation for $\Hankel_n(T;\tau)$ with respect to $\tau$.

\begin{prop}\label{deftau}
The Hankel determinant $\Hankel_n(T;\tau)$ satisfies the differential equation
%
\begin{equation}
\label{deftaueq} \frac{\d^2}{\d\tau^2} \log\Hankel_n(T;\tau) =
T^2 \bigl(\gamma ^{(T; \tau
)}_{n, n}
\bigr)^2 -T ,
\end{equation}
where the recurrence coefficient $\gamma^{(T; \tau)}_{n, n}$ is
defined in
\eqref{eq:three_term_recurrence}.
\end{prop}

\begin{pf}
Introducing a linear term into the exponent of the symbol for the
Hankel determinant, we define
%
\begin{equation}
\label{pf2} \Hankel_n(T;\tau; t):=\det \biggl( \frac{1}{n}
\sum_{x \in L_{n,
\tau}} x^{j + k - 2} e^{-({nT}/{2}) (x^2+{2tx}/{n} )}
\biggr)^n_{j, k = 1} ,
\end{equation}
and the monic orthogonal polynomials
%
\begin{equation}
\label{op2} \frac{1}{n}\sum_{x\in L_{n,\tau}}
p_{n,j}^{(T;\tau;t)}(x)p_{n,
l}^{(T;\tau;t)}(x)
e^{-({nT}/{2}) (x^2+{2tx}/{n} )} =h_{n,j}^{(T;\tau;t)} \delta_{jl}.
\end{equation}
It is well known then [see, e.g., \citet{Bleher-Liechty13},
Theorem~2.4.3] that this Hankel determinant satisfies
%
\begin{eqnarray}
\label{pf3} \frac{\d^2}{\d t^2} \log\Hankel_n(T;\tau; t)=
\frac{T^2
h_{n,n}^{(T,\tau;t)}}{h_{n,n-1}^{(T,\tau;t)}} = T^2 \bigl(\gamma ^{(T; \tau
; t)}_{n, n}
\bigr)^2,
\nonumber
\\[-8pt]
\\[-8pt]
\eqntext{\mbox{where } \displaystyle\gamma^{(T; \tau; t)}_{n,
j} :=
\biggl(\frac{h^{(T; \tau; t)}_{n, j}}{h^{(T; \tau; t)}_{n,
j-1}} \biggr)^{1/2}.}
\end{eqnarray}
Completing the square in \eqref{pf2}, we find that
%
\begin{eqnarray}
\label{pf4} %
 \Hankel_n(T;\tau; t)&=&\det
\biggl( \frac{e^{{Tt^2}/{(2n)}}}{n} \sum_{x \in L_{n, \tau}}
x^{j + k - 2} e^{-({nT}/{2}) (x+
{t}/{n} )^2} \biggr)^n_{j, k = 1}
\nonumber
\\[-8pt]
\\[-8pt]
\nonumber
&=&e^{{T t^2}/{2}} \Hankel_n(T;\tau+t; 0).
\end{eqnarray}
Taking the logarithm and differentiating twice with respect to $t$, we obtain
%
\begin{equation}
\label{pf5} \frac{\d^2}{\d t^2} \log\Hankel_n(T;\tau; t)=T+
\frac{\d^2}{\d t^2} \log\Hankel_n(T;\tau+t; 0) ,
\end{equation}
and combining \eqref{pf3} with \eqref{pf5} gives
%
\begin{equation}
\label{pf6} \frac{\d^2}{\d t^2} \log\Hankel_n(T;\tau+t; 0) =
T^2 \bigl(\gamma^{(T;
\tau; t)}_{n, n}
\bigr)^2-T.
\end{equation}
Now replacing $\d^2/\d t^2$ with $\d^2/\d\tau^2$ on the left-hand side
of \eqref{pf6} and plugging in $t=0$ gives \eqref{deftaueq}, and the
proposition is proved.
\end{pf}

We can now use this proposition to write an integral equation for the
ratio in equation \eqref{winu2}.
For $\epsilon(n)=0$ or $\epsilon(n)=1/2$, it is clear that $\Hankel
_n(T,\tau
)$ satisfies the symmetries
%
\begin{equation}
\Hankel_n\bigl(T;\epsilon(n)+\tau\bigr) = \Hankel_n
\bigl(T;\epsilon(n)-\tau\bigr) = \Hankel _n\bigl(T;\tau- \epsilon(n)
\bigr).
\end{equation}
Therefore, we have
%
\begin{equation}
\frac{\d}{\d\tau} \log\Hankel_n(T;\tau) \bigg|_{\tau=\epsilon
(n)}=0 ,
\end{equation}
and then Proposition~\ref{deftau} implies the integral formula
%
\begin{eqnarray}
\label{inteq} \log\frac{\Hankel_n(T; \tau-\epsilon(n))}{\Hankel_n(T; \epsilon
(n))} &=& \log \frac{\Hankel_n(T; \epsilon(n) + \tau)}{\Hankel_n(T; \epsilon
(n))}
\nonumber
\\[-8pt]
\\[-8pt]
\nonumber
& =& \int
_{\epsilon(n)}^{\epsilon(n) + \tau} \int_{\epsilon(n)}^u
\bigl(T^2 \bigl(\gamma^{(T;
v)}_{n, n}
\bigr)^2 -T \bigr) \,dv \,du.
\end{eqnarray}

\textit{Subcritical case}.
In the subcritical case $T< T_c$ we can apply the asymptotic formula
\eqref{gamma_sub} for $ (\gamma^{(T; v)}_{n, n} )^2$. Then
combining \eqref{winu2} and \eqref{inteq} gives \eqref{wnsub}.

\textit{Supercritical case}.
In the supercritical case $T> T_c$, We will use the notation~$\tilde
{k}$, $\tilde{\mathbf K}$ and $\tilde{\mathbf E}$ introduced in
\eqref
{eq:def_tilde1} and \eqref{eq:def_tilde2}, as well as the elliptic nome
$q$ introduced in~\eqref{eq:def_elliptic_nome_intro}. We apply the
asymptotic formula \eqref{gamma_super} to the integral equation \eqref
{inteq}, giving
%
\begin{eqnarray}
\label{ratiosuper} %
&& \log\frac{\Hankel_n(T,\tau-\epsilon(n))}{\Hankel_n(T,\epsilon
(n))}\hspace*{-30pt}\nonumber 
\\
&&\qquad= \int
_{\epsilon(n)}^{\epsilon(n) + \tau} \int_{\epsilon(n)}^u
\biggl(\frac
{T^2}{4\tilde{\mathbf E}^2}\dn^2 \biggl( 2\tilde{\mathbf K} \biggl( v+
\frac{1}{2} +\epsilon(n) \biggr),\tilde{k} \biggr) -T \biggr) \,dv \,du + \bigO
\bigl(n^{-1}\bigr)\hspace*{-30pt}
\nonumber
\\[-8pt]\hspace*{-30pt}
\\[-8pt]
\nonumber
&&\qquad= \int^{\tau}_0 \int^u_0
\biggl(\frac{T^2}{4\tilde{\mathbf
E}^2}\dn ^2 \biggl( 2\tilde{\mathbf K} \biggl(
v+ \frac{1}{2} + 2\epsilon(n) \biggr),\tilde{k} \biggr) -T \biggr) \,dv \,du +
\bigO\bigl(n^{-1}\bigr)\hspace*{-30pt}
\\
&&\qquad= \int^{\tau}_0 \int^u_0
\biggl(\frac{T^2}{4\tilde{\mathbf
E}^2}\dn ^2 \biggl( 2\tilde{\mathbf K} \biggl(
v+ \frac{1}{2} \biggr),\tilde{k} \biggr) -T \biggr) \,dv \,du + \bigO
\bigl(n^{-1}\bigr),\hspace*{-30pt}\nonumber
\end{eqnarray}
where we use that $\dn(u, \tilde{k})$ has period $2\tilde{\K}$ as a
function of $u$ [\citet{Erdelyi-Magnus-Oberhettigner-Tricomi81},
Table~5 on page
341]. Let us discuss how to
compute the integral
%
\begin{equation}
\label{doubleint1} \int_0^\tau\int
_0^u \dn^2 \biggl( 2\tilde{\mathbf
K} \biggl( v + \frac
{1}{2} \biggr), \tilde{k} \biggr) \,dv \,du.
\end{equation}
The inner integral can be written as
%
\begin{equation}
\frac{1}{2\tilde{\mathbf K}} \biggl[\int_0^{2\tilde{\mathbf K}
u+\tilde
{\mathbf K}}
\dn^2(t,\tilde{k}) \,dt-\int_0^{\tilde{\mathbf K}}\dn
^2(t,\tilde{k}) \,dt \biggr].
\end{equation}
The above integrals can be written in terms of the Jacobi Zeta function
$Z(u,\tilde{k})$ [\citet{Erdelyi-Magnus-Oberhettigner-Tricomi81},
Section~13.16], which can be expressed
by the Jacobi theta function as [\citet{Watson-Whittaker96},
Sections~22.731, 21.11,
21.62],
%
\begin{equation}
\label{eq:Jacobi_zeta_in_theta} Z(t,\tilde{k})= \frac{\d}{\d t} \log\Theta(t)\qquad
 \mbox{where }
\Theta(t)= \th_4 \biggl(\frac{\pi t}{2\tilde{\mathbf K}} \biggr).
\end{equation}
Using \citet{Erdelyi-Magnus-Oberhettigner-Tricomi81},
Section~13.16, Formulas (12) and
(14), we have
%
\begin{equation}
\int_0^u \dn^2(t,\tilde{k})
\,dt=Z(u,\tilde{k})+ \frac{\tilde
{\mathbf
E}}{\tilde{\mathbf K}} u,
\end{equation}
and then
%
\begin{equation}
\int_0^u \dn^2\bigl(2\tilde{
\mathbf K}(v+1/2),\tilde{k}\bigr) \,dv=\frac
{1}{2\tilde{\mathbf K}} \bigl[Z(2\tilde{\mathbf
K} u+\tilde{\mathbf K},\tilde{k})+2\tilde{\mathbf E}u \bigr] ,
\end{equation}
where we have used that $Z(\tilde{\mathbf K},\tilde{k})=0$ by \eqref
{eq:Jacobi_zeta_in_theta} and that $\th'_4(\pi/2) = 0$ [see \citet{Watson-Whittaker96},
Section~21.11]. The integral \eqref{doubleint1}
is thus
%
\begin{eqnarray}
\label{int1} &&\int_0^\tau\int
_0^u \dn^2\bigl(2\tilde{\mathbf
K}(v+1/2),\tilde{k}\bigr) \,dv \,du
\nonumber
\\[-8pt]
\\[-8pt]
\nonumber
&&\qquad=\frac{1}{2\tilde{\mathbf K}} \biggl[\frac{1}{2\tilde{\mathbf
K}}
\int_{\tilde{\mathbf K}}^{2\tilde{\mathbf K} \tau+\tilde{\mathbf K}} Z(t, \tilde k) \,dt+\int
_0^\tau2\tilde{\mathbf E}u \,du \biggr].
\end{eqnarray}
Integrating the right-hand side of \eqref{int1}, we obtain
%
\begin{eqnarray}
&&\int_0^\tau\int_0^u
\dn^2\bigl(2\tilde{\mathbf K}(v+1/2),\tilde{k}\bigr) \,dv \,du
\nonumber
\\[-8pt]
\\[-8pt]
\nonumber
&&\qquad=
\frac{1}{2\tilde{\mathbf K}} \biggl[\frac{1}{2\tilde{\mathbf
K}}\log \biggl(\frac{\Theta(2\tilde{\mathbf K} \tau+\tilde{\mathbf
K})}{\Theta
(\tilde{\mathbf K})}
\biggr)+\tilde{\mathbf E}\tau^2 \biggr].
\end{eqnarray}
Combining with \eqref{inteq} and \eqref{ratiosuper}, we obtain
%
\begin{eqnarray}
&&\log\frac{\Hankel_n(T,\tau-\epsilon(n))}{\Hankel_n(T,\epsilon
(n))}
\nonumber
\\[-8pt]
\\[-8pt]
\nonumber
&&\qquad= \frac
{T^2}{8\tilde{\mathbf K}\tilde{\mathbf E}^2} \biggl[\frac{1}{2\tilde
{\mathbf K}}\log
\biggl(\frac{\Theta(2\tilde{\mathbf K} \tau+\tilde
{\mathbf K})}{\Theta(\tilde{\mathbf K})} \biggr)+\tilde{\mathbf E}\tau ^2 \biggr]-
\frac{T}{2}\tau^2+\bigO\bigl(n^{-1}\bigr).
\end{eqnarray}
The parametrization $T = 4\tilde{\K} \tilde{\E}$ in \eqref
{eq:T_parametrized_by_k} then implies
[\citet{Watson-Whittaker96}, Section~21.11],
%
\begin{eqnarray}
 \log\frac{\Hankel_n(T,\tau-\epsilon(n))}{\Hankel_n(T,\epsilon(n))} &=&\log \biggl(
\frac{\Theta(2\tilde{\mathbf K} \tau+\tilde{\mathbf
K})}{\Theta
(\tilde{\mathbf K})} \biggr)+\bigO\bigl(n^{-1}\bigr)
\nonumber
\\[-8pt]
\\[-8pt]
\nonumber
&=&\log \biggl(\frac{\th_3(\pi\tau)}{\th_3(0)} \biggr)+\bigO \bigl(n^{-1}\bigr).
\end{eqnarray}
Then the Fourier series \eqref{eq:def_Jacobi_theta} for the function
$\th_3$ and the identity [\citet{Watson-Whittaker96}, Section~21.8],
%
\begin{equation}
\th_3(0)^2 =\frac{2\tilde{\mathbf K}}{\pi},
\end{equation}
imply \eqref{wnsuper}.

\textit{Critical case}.
We now consider the critical case $T=T_c(1-2^{-2/3}\sigma n^{-2/3})$. In
this part of the proof, we use the notation $q(s)$ for the
Hastings--McLeod solution to the Painlev\'{e} equation \eqref{cr2} and
\eqref{cr3}. Inserting the asymptotic formula \eqref{gamma_critical}
into this integral equation \eqref{inteq} yields
%
\begin{eqnarray}
&&\log\frac{\Hankel_n(T; \tau-\epsilon(n))}{\Hankel_n(T; \epsilon
(n))}
\nonumber
\\
&&\qquad=
T2^{4/3}\int_{\epsilon(n)}^{\epsilon(n) + \tau} \int
_{\epsilon
(n)}^u \biggl(\frac{2^{1/3}}{n^{1/3}}q(\sigma)\cos
\bigl(2\pi\bigl(v+\epsilon (n)\bigr) \bigr)\\
&&\hspace*{125pt}{}+\frac
{1}{n^{2/3}}q(
\sigma)^2\cos(4\pi v)+\bigO\bigl(n^{-1}\bigr) \biggr) \,dv \,du
,\nonumber
\end{eqnarray}
which is integrated to obtain
%
\begin{eqnarray}
&&\log\frac{\Hankel_n(T; \tau-\epsilon(n))}{\Hankel_n(T; \epsilon
(n))}
\nonumber
\\
&&\qquad=T2^{4/3} \biggl(-\frac{2^{1/3}q(\sigma)}{4\pi^2n^{1/3}} \bigl(1-
\cos (2\pi\tau) \bigr)\\
&&\hspace*{65pt}{}+\frac{q(\sigma)^2}{16\pi^2n^{2/3}} \bigl(1-\cos (4\pi\tau ) \bigr)
\biggr)+\bigO\bigl(n^{-1}\bigr).\nonumber
\end{eqnarray}
Using the scaling \eqref{eq:T_scaling} for $T$, we find
%
\begin{eqnarray}
&&\log\frac{\Hankel_n(T; \tau-\epsilon(n))}{\Hankel_n(T; \epsilon
(n))}
\nonumber
\\[-8pt]
\\[-8pt]
\nonumber
&&\qquad=-\frac
{q(\sigma)}{2^{1/3}n^{1/3}} \bigl(1-\cos(2\pi\tau) \bigr)+
\frac
{2^{1/3}q(\sigma)^2}{8n^{2/3}} \bigl(1-\cos(4\pi\tau) \bigr)+\bigO \bigl(n^{-1}
\bigr) ,
\end{eqnarray}
which we exponentiate to obtain
%
\begin{eqnarray}
&&\frac{\Hankel_n(T; \tau-\epsilon(n))}{\Hankel_n(T; \epsilon
(n))}
\nonumber
\\[-8pt]
\\[-8pt]
\nonumber
&&\qquad=1-\frac
{q(\sigma)}{2^{1/3}n^{1/3}} \bigl(1-\cos(2\pi\tau) \bigr)+
\frac
{q(\sigma
)^2}{2^{2/3}n^{2/3}} \bigl(1-\cos(2\pi\tau) \bigr)+\bigO\bigl(n^{-1}
\bigr) ,
\end{eqnarray}
and the formulas \eqref{wncrit} follow immediately from \eqref{winu2}.

Theorem~\ref{winding_number_theorem} is thus proved.

\section{Correlation function of particles} \label{sec:correlation}

In this section, we do asymptotic analysis to the $\tau$-deformed
correlation kernel $K_{t_i, t_j}(x, y)$ in \eqref{eq:corr_kernel}, and
prove Theorems~\ref{thmm:corr_kernel_intro} and~\ref
{thmm:fixed_winding_number_correlations} for the limiting behavior of
$\mathrm{NIBM}_{0 \to T}$ in the critical and supercritical cases. In
the critical case,
we simply let $\tau= \epsilon(n)$ and the asymptotics of $K_{t_i,
t_j}(x, y)$ gives Theorem~\ref{thmm:corr_kernel_intro}(b); see Remark~\ref{rmk:correlation_kernel}.
In the supercritical case, we need the following technical result.

\begin{thmm} \label{thmm:tau_deformed_correlation_kernels}
Assume $T>T_c$. There exists $d>0$ defined in \eqref{eq:Taylor_exp_upp_Pearcey}
such that when we scale $t_i$ and $t_j$ close to $t^c$, and $x$ and $y$
close to $-\pi$ as in \eqref{eq:coefficients_Pearcey}, the $\tau
$-deformed correlation kernel $K_{t_i, t_j}(x,y)$ has the limit
independent of the parameter~$\tau$
%
\begin{equation}
\label{eq:tau_deformed_cusp} \lim_{n\to\infty} K_{t_i, t_j}(x,y) \biggl
\llvert \frac{dy}{d\eta} \biggr\rrvert = K^{\Pearcey}_{-\tau_j,-\tau_i}(\eta,
\xi).
\end{equation}
\end{thmm}

Theorem~\ref{thmm:tau_deformed_correlation_kernels} yields Theorem~\ref
{thmm:corr_kernel_intro}(a) as $\tau=
\epsilon(n)$, while in Section~\ref{sec:proof_of_Theorem_fixed_winding_number}
it is shown that Theorem~\ref{thmm:fixed_winding_number_correlations} also follows from Theorem~\ref{thmm:tau_deformed_correlation_kernels}.

In Section~\ref{sec:algebraic_contour_integrals}, we lay out the
contour integral formulas to do asymptotic analysis, and the
supercritical and critical cases are undertaken in Sections~\ref{sec:kernel_cusp} and \ref{sec:kernel_tacnode}, respectively.
Throughout this section, we simplify the notation for the orthogonal
polynomials \eqref{eq:defn_of_discrete_Gaussian_OP} a bit, writing
$p_k(x)$ for $p_{n,k}^{(T;\tau)}(x)$ when there is no possibility of confusion.

\subsection{Contour integral formula of the \texorpdfstring{$\tau$}{tau}-deformed correlation
kernel} \label{sec:algebraic_contour_integrals}

First, we express the function $S_{k, a}(x)$ defined in \eqref
{eq:Sla} in contour integral formulas that are convenient for
asymptotic analysis. Under some circumstances, it is convenient to
express $S_{k, a}(x)$ by an integral over an infinite contour. Consider
the function
%
\begin{equation}
\label{eq:tilde_S_defined} P_{k, a}(z; x) = \pi p_k(z)
e^{-{anz^2}/{2}} \frac{e^{i(x - \pi)nz
+ i\tau\pi}}{\sin(\pi nz - \tau\pi)}.
\end{equation}
It is straightforward to check that $P_{k, a}(z; x)$ has poles only at
lattice points of $L_{n, \tau}$, and
%
\begin{equation}
\Res_{z = s \in L_{n, \tau}} P_{k, a}(z; x) = \frac{1}{n}
p_k(s) e^{-
{ans^2}/{2}} e^{ixns}.
\end{equation}
Since $a$ is assumed to be positive, $P_{k, a}(z; x)$ vanishes
exponentially fast as $z \to\infty$ in the direction $0$ or direction
$\pi$. Thus, if $\Sigma^+$ is a contour in the upper half-plane and
from $e^0 \cdot\infty$ to $e^{\pi i} \cdot\infty$, and $\Sigma^-$ is
a contour in the lower half-plane and from $e^{\pi i} \cdot\infty$ to
$e^0 \cdot\infty$, we have
%
\begin{eqnarray}
\label{eq:Skapolys} S_{k, a}(x) = \frac{1}{2\pi i} \oint_{\Sigma}
P_{k, a}(z; x) \,dz = \oint _{\Sigma} p_k(z)
e^{-{anz^2}/{2}} \frac{e^{ixnz}}{e^{2\pi i nz -
2\tau\pi i} - 1} \,dz
\nonumber
\\[-8pt]
\\[-8pt]
\eqntext{\mbox{where } \Sigma= \Sigma^+ \cup
\Sigma^-.}
\end{eqnarray}

Under some other circumstances, it is convenient to express $S_{k,
a}(x)$ as the sum of a contour integral and a remainder that is
negligible in the asymptotic analysis. For any $M > 0$ such that $\pm
M$ are not lattice points in $L_{n, \tau}$, we write
%
\begin{eqnarray}
S_{k, a}(x) = \frac{1}{n} \mathop{\sum
_{s \in L_{n, \tau}}}_{ \vert s
\vert\leq M} p_k(s) e^{-{ans^2}/{2}}
e^{ixns} + s^{(M)}_{k,
a}(x)
\nonumber
\\[-8pt]
\\[-8pt]
\eqntext{\mbox{where }
\displaystyle s^{(M)}_{k, a}(x) = \frac{1}{n} \mathop{\sum
_{s \in L_{n, \tau}}}_{ \vert s \vert> M} p_k(s) e^{-
{ans^2}/{2}}
e^{ixns}.}
\end{eqnarray}
Recall the discrete Cauchy transform $Cp_k(z)$ defined in \eqref
{eq:def_Cauchy_trans}. Let $\Gamma$ be a closed contour such that the
part of $L_{n, \tau}$, $\{ s \in L_{n, \tau}\mid\vert s \vert\leq M
\}$
is enclosed in $\Gamma$ while the rest of $L_{n, \tau}$ is outside of
$\Gamma$. By the calculation of residues,
%
\begin{equation}
\label{eq:SkapolyandCauchy} S_{k, a}(x) = \frac{1}{2\pi i} \oint_{\Gamma}
Cp_k(z) e^{{(T -
a)nz^2}/{2}} e^{ixnz} \,dz + s^{(M)}_{k, a}(x).
\end{equation}

Therefore, by \eqref{eq:Skapolys} and \eqref
{eq:SkapolyandCauchy}, we can write \eqref{eq:corr_kernel} as
%
\begin{eqnarray}
\label{KKmajorKminor} \tilde{K}_{t_i, t_j}(x, y) &=& K^{\major}_{t_i, t_j}(x,
y) + K^{\minor
}_{t_i, t_j}(x, y)
\nonumber
\\[-8pt]
\\[-8pt]
\nonumber
&=& K^{\major}_{t_i, t_j}(x,
y; M) + K^{\minor}_{t_i,
t_j}(x, y; M),
\end{eqnarray}
where $K^{\major}_{t_i, t_j}(x, y; M)$ and $K^{\minor}_{t_i, t_j}(x, y;
M)$ depend on the positive constant $M$ which we suppress if there is
no possibility of confusion. They are defined as
%
\begin{eqnarray}
\label{eq:Kmajor} K^{\major}_{t_i, t_j}(x, y) &=& \frac{n}{4\pi^2 i}
\oint_{\Gamma} \,dz \oint_{\Sigma} \,dw \Biggl( \sum
^{n - 1}_{k = 0} \frac{1}{h^{(T; \tau
)}_{n, k}} Cp_k(z)
p_k(w) \Biggr) e^{{t_i nz^2}/{2} - {t_j
nw^2}/{2}}
\nonumber
\\[-8pt]
\\[-8pt]
\nonumber
&&{}\times\frac{-e^{ixnz - iynw}}{1 - e^{2\pi i nw - 2\tau\pi i}},\\
\label{eq:Kminor} %
K^{\minor}_{t_i, t_j}(x, y) &=&
\frac{n}{2\pi} \oint_{\Sigma} \,dw \Biggl( \sum
^{n - 1}_{k = 0} \frac{1}{h^{(T; \tau)}_{n, k}} s^{(M)}_{k,
T - t_i}(x)
p_k(w) \Biggr) e^{-{t_j nw^2}/{2}} \nonumber\\
&&{}\times\frac{-e^{-iynw}}{1
- e^{2\pi i nw - 2\tau\pi i}}
\\
\nonumber
&=& \frac{1}{2\pi} \oint_{\Sigma} \,dw \mathop{\sum
_{s \in L_{n, \tau}
}}_{ \vert s \vert> M} \Biggl( \sum
^{n - 1}_{k = 0} \frac{1}{h^{(T;
\tau)}_{n, k}} p_k(s)
p_k(w) \Biggr) e^{-{(T - t_i)ns^2}/{2} -
{t_j nw^2}/{2}} \\
&&{}\times \frac{-e^{ixns - iynw}}{1 - e^{2\pi i nw - 2\tau\pi i}}.\nonumber %
\end{eqnarray}
In \eqref{eq:Kmajor} and \eqref{eq:Kminor}, we assume that the
contour $\Gamma$ is the same as in \eqref{eq:SkapolyandCauchy},
$\Sigma$ is the same as in \eqref{eq:Skapolys}, and $\Gamma$ and
$\Sigma$ are disjoint.

Recall the well-known Christoffel--Darboux formula [\citet{Szego75},
Chapter~3.2]
%
\begin{eqnarray}
\label{eq:C-D} &&\sum^{n - 1}_{k = 0}
\frac{1}{h^{(T; \tau)}_{n, k}} p_k(z) p_k(w)
\nonumber
\\[-8pt]
\\[-8pt]
\nonumber
&&\qquad = \frac{1}{h^{(T; \tau)}_{n, n - 1}}
\frac{p_n(z) p_{n - 1}(w) - p_{n -
1}(z) p_n(w)}{z - w}.
\end{eqnarray}
We derive its straightforward variation
%
\begin{eqnarray}
\label{eq:C-D_variation} %
&& \sum^{n - 1}_{k = 0}
\frac{1}{h^{(T; \tau)}_{n, k}} Cp_k(z) p_k(w)
\nonumber\\
&&\qquad= \sum^{n - 1}_{k = 0} \frac{1}{nh^{(T; \tau)}_{n, k}} \sum
_{s \in
L_{n, \tau}} \frac{p_k(s) e^{-{Tns^2}/{2}}}{z - s} p_k(w)
\nonumber\\
&&\qquad= \sum_{s \in L_{n, \tau}} \frac{1}{n} \frac{e^{-{Tns^2}/{2}}}{z
- s}
\sum^{n - 1}_{k = 0} \frac{1}{h^{(T; \tau)}_{n, k}}
p_k(s) p_k(w)
\nonumber\\
&&\qquad= \sum_{s \in L_{n, \tau}} \frac{1}{h^{(T; \tau)}_{n, n - 1}}
\frac
{p_n(s) p_{n - 1}(w) - p_{n - 1}(s) p_n(w)}{n(z - s)(s - w)} e^{-
{Tns^2}/{2}}
\nonumber\\
&&\qquad= \frac{1}{h^{(T; \tau)}_{n, n - 1}} \frac{1}{z - w} \sum_{s \in
L_{n, \tau}}
\frac{1}{n} \frac{p_n(s) e^{-{Tns^2}/{2}}}{z - s} p_{n -
1}(w) \nonumber
\\
&&\hspace*{88pt}\qquad\quad{}- \frac{1}{n}
\frac{p_n(s) e^{-{Tns^2}/{2}}}{w - s} p_{n -
1}(w)
\nonumber\\
&&\hspace*{88pt}\qquad\quad{} - \frac{1}{n} \frac{p_{n - 1}(s) e^{-{Tns^2}/{2}}}{z
- s} p_n(w)\nonumber\\
&&\hspace*{88pt}\qquad\quad{} +
\frac{1}{n} \frac{p_{n - 1}(s) e^{-{Tns^2}/{2}}}{w -
s} p_n(w)
\\
&&\qquad= \frac{1}{h^{(T; \tau)}_{n, n - 1}} \biggl( \frac{Cp_n(z) p_{n -
1}(w) - Cp_{n - 1}(z) p_n(w)}{z - w} \nonumber\\
&&\hspace*{36pt}\qquad\quad{}- \frac{Cp_n(w) p_{n - 1}(w) -
Cp_{n - 1}(w) p_n(w)}{z - w} \biggr).\nonumber
\end{eqnarray}
Using \eqref{eq:C-D_variation} and \eqref{eq:C-D} and noting that
$\oint
_{\Gamma} \frac{dz}{z - w}  ( Cp_n(w) p_{n - 1}(w) -\break  Cp_{n - 1}(w)
p_n(w)  ) = 0$ for $w \in\Sigma$, we simplify \eqref{eq:Kmajor}
and \eqref{eq:Kminor} as
\begin{eqnarray}
\label{eq:Kmajorcontourint} K^{\major}_{t_i, t_j}(x, y) &=&
 \frac{n}{4\pi^2 i}
\oint_{\Gamma} \,dz \oint_{\Sigma} \,dw \frac{1}{h^{(T; \tau)}_{n, n - 1}}
\frac{Cp_n(z) p_{n
- 1}(w) - Cp_{n - 1}(z) p_n(w)}{z - w}\hspace*{-30pt}
\nonumber
\\[-8pt]\hspace*{-30pt}
\\[-8pt]
\nonumber
&&{}\times e^{{t_i nz^2}/{2} - {t_j nw^2}/{2}} \frac{-e^{ixnz -
iynw}}{1 - e^{2\pi i nw - 2\tau\pi i}},\hspace*{-30pt}
\nonumber
\end{eqnarray}
%
\begin{eqnarray}
\label{eq:Kminorcontourint} K^{\minor}_{t_i, t_j}(x, y) &= &\frac{1}{2\pi}
\oint_{\Sigma} \,dw \mathop{\sum_{s \in L_{n, \tau}}}_{ \vert s \vert> M}
\frac{1}{h^{(T; \tau
)}_{n, n - 1}} \frac{p_n(s) p_{n - 1}(w) - p_{n - 1}(s) p_n(w)}{s - w}
\nonumber
\\[-8pt]
\\[-8pt]
\nonumber
&&{}\times e^{-{(T - t_i)ns^2}/{2} - {t_j nw^2}/{2}} \frac
{-e^{ixns - iynw}}{1 - e^{2\pi i nw - 2\tau\pi i}}.
\end{eqnarray}
These formulas are convenient in the derivation of the Pearcey kernel.
For the tacnode kernel, however, it is more convenient to write \eqref
{eq:corr_kernel} in the form
%
\begin{eqnarray}
\label{eq:double_countour_for_tacnode} 
&&\frac{n}{2\pi} \oint_\Sg \,dz \oint_{\Sg} \,dw\,
e^{-({n}/{2})[t_j w^2+(T-t_i)z^2]} e^{in(xz-yw)} \Biggl(\sum_{k=0}^{n-1}
\frac{p_k(z)p_k(w)}{h_{n,k}^{(T;\tau)}} \Biggr)\nonumber
\\
&&\quad{} \times\frac{e^{2\pi i(nz-\tau)}}{(e^{2\pi i(nz-\tau
)}-1)(e^{2\pi i(nw-\tau)}-1)}\nonumber
\\
&&\qquad=\frac{n}{2\pi} \oint_\Sg \,dz \oint_{\Sg} \,dw\,
e^{-({n}/{2})[t_j
w^2+(T-t_i)z^2]} e^{in(xz-yw)}
\\
&&\qquad\quad{} \times \biggl(\frac
{p_n(z)p_{n-1}(w)-p_{n-1}(z)p_n(w)}{h_{n,n-1}^{(T;\tau)}(z-w)} \biggr)\nonumber\\
&&\qquad\quad{}\times\frac{e^{2\pi i(nz-\tau)}}{(e^{2\pi i(nz-\tau)}-1)(e^{2\pi
i(nw-\tau
)}-1)},\nonumber
\end{eqnarray}
by \eqref{eq:Skapolys} and \eqref{eq:C-D}, noting that the term
$e^{ixnz}$ can be replaced by $e^{ix nz + 2\pi i (nz - \tau)}$ in
\eqref
{eq:Skapolys}.

\subsection{Limiting Pearcey process} \label{sec:kernel_cusp}

In this subsection, we assume that $t_i, t_j, x, y$ are defined by
\eqref{eq:coefficients_Pearcey}, and the parameter $d$ in \eqref
{eq:coefficients_Pearcey} is to be determined later in \eqref{eq:d_in_Pearcey}.

To evaluate $K^{\major}_{t_i, t_j}(x, y)$ in \eqref
{eq:Kmajorcontourint}, we define some notation. We denote for any $z
\in\compC\setminus(-\infty, \beta)$
%
\begin{eqnarray}
\label{eq:defn_I_tildeI} I(z) &=& g(z) - \frac{t_c z^2}{2} + i\pi z,
\nonumber
\\[-8pt]
\\[-8pt]
\nonumber
 \tilde{I}(z) &=&
\cases{ \displaystyle g(z) - \frac{t_c z^2}{2} + i\pi z = I(z), &\quad $\mbox{if $\Im z
> 0$}$, \vspace*{2pt}
\cr
\displaystyle g(z) - \frac{t_c z^2}{2} - i\pi z, &\quad  $\mbox{if $\Im z <
0$}$. } %
\end{eqnarray}
Although $I(z)$ and $\tilde{I}(z)$ are generally not well defined on
the real line, we define
%
\begin{equation}
I(x) = \lim_{y \to0^+} I(x + iy),\qquad \tilde{I}(x) = \lim
_{y \to
0^+} \tilde{I}(x + iy)\qquad \mbox{for $x \in\realR$}.
\end{equation}
Note that by the relation \eqref{eq:differencegcentralll} of
$g_+(x)$ and $g_-(x)$ for $x \in(-\alpha, \alpha)$, the $g$-function
defined on $\compC_+$ can be analytically continued to $\compC_-$
through the interval $(-\alpha, \alpha)$. This analytic continuation is
well defined on $\compC\setminus((-\infty, -\alpha) \cup(\alpha,
\infty))$, and we denote it as $\tilde{g}(z)$. By \eqref
{eq:differencegcentralll}, we have
%
\begin{equation}
\label{eq:defn_tilde_g} \tilde{g}(z) = %
\cases{ g(z), & \quad $\mbox{if $\Im z >
0$}$, \vspace*{2pt}
\cr
g(z) + i\pi- 2\pi i z, & \quad $\mbox{if $\Im z < 0$}$. }
\end{equation}
Thus, we can express $\tilde{I}(z)$ as
%
\begin{equation}
\label{eq:alternative_formula_tilde_g} \tilde{I}(z) = %
\cases{\displaystyle  \tilde{g}(z) -
\frac{t_c z^2}{2} + i\pi z, &\quad  $\mbox{if $\Im z > 0$}$, \vspace*{2pt}
\cr
\displaystyle \tilde{g}(z) - \frac{t_c z^2}{2} + i\pi z - i\pi,& \quad $\mbox{if $\Im z < 0$}$. }
\end{equation}
We also define the function $F(z, w)$ for $z, w \in\compC\setminus
\realR$ as
%
\begin{eqnarray}
\label{eq:occurance_of_F} F(z, w)& = &\frac{e^{n(g(z) - g(w))}}{h^{(T; \tau)}_{n, n - 1}} \bigl( Cp_n(z)
p_{n - 1}(w) - Cp_{n - 1}(z) p_n(w) \bigr)
\nonumber
\\[-8pt]
\\[-8pt]
\nonumber
&&{}\times
\frac{-1}{1 -
e^{2\pi i nw - 2\tau\pi i}}.
\end{eqnarray}
Then we write \eqref{eq:Kmajorcontourint} as
%
\begin{eqnarray}
\label{eq:K^major_Pearcey} K^{\major}_{t_i, t_j}(x, y) &=& \frac{n}{4\pi^2 i}
\oint_{\Gamma} \,dz \oint_{\Sigma} \,dw e^{-nI(z) + n\tilde{I}(w)}
\nonumber
\\[-8pt]
\\[-8pt]
\nonumber
&&{}\times \frac{e^{n^{{1}/{2}}
({d^2}/{2})(\tau_i z^2 - \tau_j w^2) - i n^{{1}/{4}} \,d(\xi z -
\eta w)}}{z - w} F(z, w).
\end{eqnarray}

In Appendix~\ref{sec:critical_pt}, we construct the contour $\tilde
{\Sigma}$ that intersects the real axis at $0$ and lies above the real
axis elsewhere, such that $\Re I(z)$ attains its global maximum on
$\tilde{\Sigma}$ uniquely at $0$, and construct the contour $\tilde
{\Gamma}$ that lies above or on the real axis, passes through $0$,
overlaps with the real axis in the vicinity of $0$, starts at $M$ and
ends at $-M$, where $M > \beta$ such that $\Re I(z)$ attains its global
minimum on $\tilde{\Gamma}$ uniquely at $0$. We define $\Sigma^{\upp}$
by a deformation of $\tilde{\Sigma}$ such that $\Sigma^{\upp}$ is
identical to $\tilde{\Sigma}$ outside of the region $\{ z \in\compC
\mid\vert z \vert< 4 d^{-1} n^{-1/4} \}$, and in this region the
corner of $\tilde{\Sigma}$ is leveled to be a horizontal base that is
above $0$ by $2 d^{-1} n^{-1/4}$. We also define $\Gamma^{\upp}$ by a
deformation of $\tilde{\Gamma}$ as follows. First, we shift $\tilde
{\Gamma}$ upward by $d^{-1} n^{-1/4}$, and then connect the two end
points of the shifted $\tilde{\Gamma}$, namely $\pm M + i d^{-1}
n^{-1/4}$, to $\pm M$, respectively, by vertical bars of length $d^{-1}
n^{-1/4}$. $\Gamma^{\upp}$ is the result of the deformation. At last,
we construct $\Sigma^{\low}$ and $\Gamma^{\low}$ by a reflection of
$\Sigma^{\upp}$ and $\Gamma^{\upp}$, respectively, about the real axis,
and define
%
\begin{equation}
\Sigma= \Sigma^{\upp} \cup\Sigma^{\low}, \qquad\Gamma=
\Gamma^{\upp
} \cup\Gamma^{\low},
\end{equation}
with the orientation prescribed for $\Sigma$ and $\Gamma$. See Figures~\ref{fig:Sigma_Gamma_local_Pearcey} and \ref{fig:schematic_Pearcey}. We
assume, without loss of generality, that $\pm M$ defined in Appendix~\ref{sec:critical_pt} are not lattice points of $L_{n, \tau}$, otherwise
we deform the contour around $\pm M$ by $\bigO(n^{-1})$.
\begin{figure}

\includegraphics{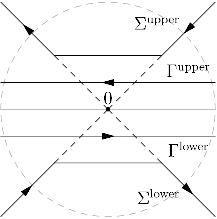}

\caption{$\Sigma^{\upp}$ and $\Gamma^{\upp}$ are deformed from
$\tilde
{\Sigma}$ and $\tilde{\Gamma}$, respectively. The circled region is the
neighborhood $\{z \in\compC\mid\vert z \vert< 4d^{-1} n^{-1/4} \}$.}
\label{fig:Sigma_Gamma_local_Pearcey}
\end{figure}

\begin{figure}[b]

\includegraphics{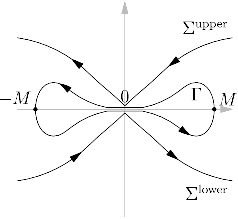}

\caption{Schematic figures of $\Sigma$ and $\Gamma$. They are close at
$0$ but do not intersect.}
\label{fig:schematic_Pearcey}
\end{figure}




Then we denote
%
\begin{eqnarray}
\Gamma_{\local} = \Gamma\cap N_{n^{-{2}/{9}}}(0),\qquad \Sigma
_{\local} = \Sigma\cap N_{n^{-{2}/{9}}}(0)
\nonumber
\\[-8pt]
\\[-8pt]
\eqntext{\mbox{where }
N_{n^{-{2}/{9}}}(0) = \bigl\{ z \in\compC\mid\vert z \vert< n^{-{2}/{9}}
\bigr\},}
\end{eqnarray}
and divide $\Gamma_{\local}$ and $\Sigma_{\local}$ into upper and lower
parts, respectively, as
%
\begin{eqnarray}
\Gamma^{\upp}_{\local} &=& \Gamma_{\local} \cap\compC_+,\qquad
\Gamma ^{\low}_{\local} = \Gamma_{\local} \cap\compC_-,
\nonumber
\\[-8pt]
\\[-8pt]
\nonumber
\Sigma^{\upp
}_{\local} &=& \Sigma_{\local} \cap\compC_+,\qquad
\Sigma^{\low}_{\local
} = \Sigma_{\local} \cap\compC_-.
\end{eqnarray}
By \eqref{asub1}, \eqref{asub1a}, \eqref{eq:Mijdet1}, \eqref
{eq:jump_M_11_21} and \eqref{eq:jump_M_12_22}, we have that for $z \in
\Gamma_{\local}$ and $w \in\Sigma_{\local}$,
%
\begin{equation}
\label{eq:Fzwatcut} F(z, w) = %
\cases{ \bigl(1 + \bigO
\bigl(n^{-{3}/{4}}\bigr)\bigr) \bigl(1 + \bigO\bigl(\vert z \vert\bigr) + \bigO\bigl( \vert
w \vert\bigr)\bigr), \vspace*{2pt}\cr
\qquad  \mbox{if $z \in\Gamma^{\upp}_{\local}$ and $w
\in \Sigma^{\upp}_{\local}$}, \vspace*{2pt}
\cr
\bigl(-e^{2\tau\pi i} + \bigO\bigl(n^{-{3}/{4}}\bigr)\bigr) \bigl(1 +
\bigO\bigl(\vert z \vert\bigr) + \bigO\bigl( \vert w \vert\bigr)\bigr),\vspace*{2pt}\cr
\qquad  \mbox{if $z \in
\Gamma^{\low}_{\local}$ and $w \in\Sigma^{\low}_{\local}$},
\vspace*{2pt}
\cr
\bigl((-1)^n + \bigO\bigl(n^{-{3}/{4}}\bigr)
\bigr) \bigl(1 + \bigO\bigl(\vert z \vert\bigr) +
\bigO\bigl( \vert w \vert\bigr)\bigr),\vspace*{2pt}\cr
\qquad \mbox{if $z
\in\Gamma^{\upp}_{\local}$ and $w \in \Sigma^{\low}_{\local}$},
\vspace*{2pt}
\cr
\bigl((-1)^n e^{4\tau\pi i} + \bigO
\bigl(n^{-{3}/{4}}\bigr)\bigr) \bigl(1 + \bigO\bigl(\vert z \vert\bigr) +
\bigO\bigl( \vert
w \vert\bigr)\bigr), \vspace*{2pt}\cr
\qquad \mbox{if $z \in\Gamma^{\low
}_{\local}$ and $w
\in\Sigma^{\upp}_{\local}$}. } %
\end{equation}

Note that for $z$ in the upper half-plane around $0$, $0$ is a triple
zero of $I'(z)$ by Lemma~\ref{lem:triple_zero_I}, and the Taylor
expansions of $I(z)$ and $\tilde{I}(z)$ are
%
\begin{equation}
\label{eq:Taylor_exp_upp_Pearcey} I(z) = \tilde{I}(z) = I(0) + \tfrac{1}{24}
\tilde{g}^{(4)}(0) z^4 + \bigO\bigl(z^5\bigr),
\end{equation}
where $\tilde{g}(z)$ is defined as the analytic continuation of $g(z)$
across $(-\alpha, \alpha)$ as defined in \eqref{eq:defn_tilde_g}; see
Lemma~\ref{lem:triple_zero_I}. By \eqref{eq:formulagsuperrr},
%
\begin{eqnarray}
\label{eq:4th_derivative_of_g} \tilde{g}^{(4)}(0) &=& \frac{1}{\alpha^3}\bigl(\bigl(1
+ k^2\bigr) \E- \bigl(1 - k^2\bigr) \K \bigr)
\nonumber
\\[-8pt]
\\[-8pt]
\nonumber
& =&
\frac{k^2}{\alpha^3} \int^1_0
\frac{(1 - s^2) + (1 - k^2
s^2)}{\sqrt{(1
- s^2)(1 - k^2 s^2)}} \,ds > 0.
\end{eqnarray}
For $z$ in the lower half-plane around $0$, the Taylor expansion of
$I(z)$ and $\tilde{I}(z)$ are, by \eqref
{eq:alternative_formula_tilde_g} and \eqref{eq:defn_tilde_g},
%
\begin{eqnarray}
\label{eq:Taylor_exp_low_Pearcey} \tilde{I}(z)& =& I(0) - \pi i + \tfrac{1}{24}
\tilde{g}^{(4)}(0) z^4 + \bigO\bigl(z^5\bigr),
\nonumber
\\[-8pt]
\\[-8pt]
\nonumber
I(z) &=& I(0) - \pi i + 2\pi i z + \tfrac{1}{24} \tilde {g}^{(4)}(0)
z^4 + \bigO\bigl(z^5\bigr).
\end{eqnarray}
We make the change of variables
%
\begin{equation}
z = \bigl( \tfrac{1}{6} g^{(4)}(0) \bigr)^{- {1}/{4}}
n^{-
{1}/{4}} u,\qquad  w = \bigl( \tfrac{1}{6} g^{(4)}(0)
\bigr)^{-
{1}/{4}} n^{-{1}/{4}} v,
\end{equation}
and let
%
\begin{equation}
\label{eq:d_in_Pearcey} d = \bigl( \tfrac{1}{6} g^{(4)}(0)
\bigr)^{{1}/{4}}.
\end{equation}
Then by the Taylor expansions \eqref{eq:Taylor_exp_upp_Pearcey} and
\eqref{eq:Taylor_exp_low_Pearcey}, for $z \in\Gamma^{\upp}_{\local}$
and $w \in\Sigma^{\upp}_{\local}$,
%
\begin{eqnarray}
\label{eq:asy_crit_upper_Pearcey} \tilde{I}(w) &=& I(0) + \frac{1}{4n} v^4 +
\bigO \biggl( \frac
{v^5}{n^{{5}/{4}}} \biggr),
\nonumber
\\[-8pt]
\\[-8pt]
\nonumber
 I(z) &= &I(0) + \frac{1}{4n}
u^4 + \bigO \biggl( \frac{u^5}{n^{{5}/{4}}} \biggr),
\end{eqnarray}
and for $z \in\Gamma^{\low}_{\local}$ and $w \in\Sigma^{\low
}_{\local
}$, noting that $\Im z = -d^{-1}n^{-1/4}$ for $z \in\Gamma^{\low
}_{\local}$,
%
\begin{eqnarray}
\label{eq:asy_crit_lower_Pearcey} \tilde{I}(w)& =& I(0) + \frac{1}{4n} v^4 +
\bigO \biggl( \frac
{v^5}{n^{{5}/{4}}} \biggr) - \pi i,
\nonumber
\\[-8pt]
\\[-8pt]
\nonumber
 I(z) &=& I(0) +
\frac{2\pi
}{d n^{{1}/{4}}} + \frac{1}{4n} u^4 + \bigO \biggl(
\frac
{u^5}{n^{{5}/{4}}} \biggr) + (2\pi\Re z - \pi)i.
\end{eqnarray}

By the asymptotics \eqref{eq:asy_crit_upper_Pearcey}, \eqref
{eq:asy_crit_lower_Pearcey} and \eqref{eq:Fzwatcut}, together with
\eqref{eq:int_formula_Pearcey}, we have that
%
\begin{eqnarray}
\label{eq:essential_Pearcey} %
&& \oint_{\Gamma^{\upp}_{\local}} \,dz \oint_{\Sigma_{\local}} \,dw
e^{-nI(z) + n\tilde{I}(w)} \frac{e^{n^{{1}/{2}}
({d^2}/{2})(\tau
_i z^2 - \tau_j w^2) - i n^{{1}/{4}} \,d(\xi z - \eta w)}}{z - w} F(z, w)
\nonumber\\
&&\qquad= dn^{{1}/{4}} \oint_{\Gamma^{\upp}_{\local}} \,dz \oint _{\Sigma
_{\local}} \,dw
\frac{e^{{v^4}/{4} - ({\tau_j}/{2}) u^2 +
i\eta
v}}{e^{{u^4}/{4} - ({\tau_i}/{2}) u^2 + i\xi u}}\nonumber\\
&&\qquad\quad{}\times \frac{1 +
\bigO
 ( {u}/{n^{{1}/{4}}}  ) + \bigO (
{v}/{n^{{1}/{4}}}  )}{u - v}
\\
&&\qquad= \frac{1}{dn^{{1}/{4}}} \biggl( \oint_{\Gamma_P} \,du \oint _{\Sigma_P} \,dv
\frac{e^{{v^4}/{4} - ({\tau_j}/{2}) u^2 +
i\eta
v}}{e^{{u^4}/{4} - ({\tau_i}/{2}) u^2 + i\xi u}} \frac{1}{u - v} +
 \bigO\bigl(n^{-{1}/{4}}\bigr)
\biggr)
\nonumber\\
&&\qquad= \frac{4\pi^2 i}{dn^{{1}/{4}}} \bigl( \tilde{K}^{\Pearcey
}_{-\tau_j, -\tau_i}(\eta,
\xi) + \bigO\bigl(n^{-{1}/{4}}\bigr) \bigr).\nonumber %
\end{eqnarray}
On the other hand, from the comparison of formulas \eqref
{eq:asy_crit_upper_Pearcey} and \eqref{eq:asy_crit_lower_Pearcey}, the
formula of $\Re I(z)$ on $\Gamma^{\low}_{\local}$ has a term $2\pi
/(dn^{1/4})$ that does not appear in the formula of $\Re I(z)$ on
$\Gamma^{\upp}_{\local}$, we have
%
\begin{eqnarray}
\label{eq:local_nonessential_Pearcey}\quad&& \oint_{\Gamma^{\low}_{\local}} \,dz \oint_{\Sigma_{\local}} \,dw
e^{-nI(z)
+ n\tilde{I}(w)} \frac{e^{n^{{1}/{2}} ({d^2}/{2})(\tau_i z^2 -
\tau_j w^2) - i n^{{1}/{4}} \,d(\xi z - \eta w)}}{z - w} F(z, w)
\nonumber
\\[-4pt]
\\[-12pt]
\nonumber
&&\qquad =
\frac{4\pi^2 i}{dn^{{1}/{4}}} \bigO \bigl(
e^{-{2\pi n^{{4}/{3}}}/{d}} \bigr).
\end{eqnarray}

For $z \in\Gamma^{\upp} \setminus\Gamma^{\upp}_{\local}$ and $w
\in
\Sigma^{\upp} \setminus\Sigma^{\upp}_{\local}$, by the property that
$\Re I(z)$ attains its global minimum on $\tilde{\Gamma}$ at $0$ and
$\Re\tilde{I}(z) = \Re I(z)$ attains its global maximum on $\tilde
{\Sigma}$ at $0$, and the local behavior of $I(z) = \tilde{I}(z)$ at
$0$ in the upper half-plane, we have that
%
\begin{eqnarray}
\Re I(z) &> &\Re I(z_0) + \ep n^{-{8}/{9}} \qquad\mbox{for $z \in z
\in\Gamma^{\upp} \setminus\Gamma^{\upp}_{\local}$},
\label
{eq:estIwnotlocal1Pearcey}
\\
\Re\tilde{I}(w)& <& \Re I(z_0) - \ep n^{-{8}/{9}}\qquad \mbox{for
$w \in z \in\Sigma^{\upp} \setminus\Sigma^{\upp}_{\local}$}.
\label
{eq:estIwnotlocalPearcey}
\end{eqnarray}
For $z \in\Gamma^{\low}$ and $w \in\Sigma^{\low}$, on the other hand,
by the formula \eqref{eq:defn_I_tildeI} of $I(z)$ and $\tilde{I}(w)$
and the property that $\Re g(z) = \Re g(\bar{z})$ that follows from the
definition~\eqref{g1} of $g(z)$, we obtain that
%
\begin{equation}
\label{eq:CineqPearceycase} \Re I(z) > \Re I(\bar{z}), \qquad
\Re\tilde{I}(w) = \Re\tilde{I}(\bar
{w}) \qquad\mbox{for $z, w \in\compC_-$},
\end{equation}
and it applies for all $w \in\Sigma^{\low}$ and $z \in\Gamma^{\upp}$
except for $z = \pm M$. Also we have the estimate for $F(z, w)$ that
for all $z \in D(\delta, \ep, n) \cup E(\ep; n, \tau)$ and $w \in
D(\delta, \ep, n) \cup E(\ep)$, where $\delta\in[0, 1)$, $\ep> 0$ is
a small positive number, and $D(\delta, \ep, n), E(\ep), E(\ep; n,
\tau
)$ are defined in \eqref{eq:defnDden} and \eqref
{eq:defnEeandEent}, by Propositions \ref{asymptotics_of_OPs}
and \ref{prop:asymptotics_of_OPs_outside},
%
\begin{eqnarray}
\label{eq:roughestimateFzw} F(z, w) = \bigO(1)
\nonumber
\\[-8pt]
\\[-8pt]
\eqntext{\mbox{if $z \in D(\delta, \ep, n) \cup E(\ep;
n, \tau)$ and $w \in D(\delta, \ep, n) \cup E(\ep)$}.}
\end{eqnarray}
Then using \eqref{eq:estIwnotlocal1Pearcey}, \eqref
{eq:estIwnotlocalPearcey}, \eqref{eq:CineqPearceycase} and
\eqref{eq:roughestimateFzw}, we have that for some $\varepsilon> 0$
%
\begin{eqnarray}
\label{eq:KmajornonessentialPearcey}&& \oint_{\Gamma} \,dz \oint_{\Sigma} \,dw
e^{-nI(z) + n\tilde{I}(w)} \frac
{e^{n^{{1}/{2}} ({d^2}/{2})(\tau_i z^2 - \tau_j w^2) - i
n^{
{1}/{4}} \,d(\xi z - \eta w)}}{z - w} F(z, w)
\nonumber
\\
&&\qquad =
\oint_{\Gamma_{\local}} \,dz \oint_{\Sigma_{\local}} \,dw e^{-nI(z) +
n\tilde{I}(w)}
\\
&&\qquad\quad{}\times \frac{e^{n^{{1}/{2}} ({d^2}/{2})(\tau_i z^2 -
\tau
_j w^2) - i n^{{1}/{4}} \,d(\xi z - \eta w)}}{z - w} F(z, w) + \frac
{1}{dn^{{1}/{4}}} o\bigl(e^{-\ep n^{{1}/{9}}}\bigr).
\nonumber
\end{eqnarray}

Next, we estimate $K^{\minor}_{t_i, t_j}(x, y)$. Using the fact that
$\Re I(z)$ attains its global minimum on $\tilde{\Gamma}$ at $0$ and
$\pm M$ are the ends of $\tilde{\Gamma}$, there is a $c_1 > 0$ such that
%
\begin{equation}
\label{eq:additional_label} \Re\tilde{I}(0) = \Re I(0) = \Re I(M) - c_1 = \Re
I(-M) - c_1.
\end{equation}
By the approximation \eqref{eq:Taylor_exp_low_Pearcey} for $w \in
\Sigma
_{\local}$ of $\tilde{I}(w)$, the estimate \eqref
{eq:estIwnotlocalPearcey} and \eqref{eq:CineqPearceycase}
for $w \in\Sigma\setminus\Sigma_{\local}$, and \eqref
{eq:additional_label}, using the asymptotic formula \eqref{asub1} of
$p_n(s)$, we have that for all $w \in\Sigma$ and $t_i, t_j, x, y$
expressed by \eqref{eq:coefficients_Pearcey},
%
\begin{eqnarray}
&&p_n(w) e^{-{t_j nw^2}/{2}} \frac{-e^{-iynw}}{1 - e^{2\pi i nw -
2\tau\pi i}}
\nonumber
\\[-8pt]
\\[-8pt]
\nonumber
&&\qquad = e^{{(T - 2t_c)nM^2}/{4}}
\bigO\bigl(e^{n (g(M) -
{TM^2}/{4} - c_1 + \ep)}\bigr),
\end{eqnarray}
where $\ep$ is an arbitrarily small positive number. Similarly, for $s
\in\realR\setminus[-M, M]$, using the asymptotics formula \eqref
{asub1} of $p_{n - 1}(s)$ and Proposition~\ref
{prop:asymptotics_of_OPs_outside}, we have
%
\begin{eqnarray}
&&\frac{1}{h^{(T; \tau)}_{n, n - 1}} p_{n - 1}(s) e^{-{(T -
t_i)ns^2}/{2}} e^{ixns}
\nonumber
\\[-8pt]
\\[-8pt]
\nonumber
&&\qquad=
e^{-{(T - 2t_c)ns^2}/{4}} \bigO\bigl(e^{n(\Re
g_+(s) - {Ts^2}/{4} - l + \ep')}\bigr),
\end{eqnarray}
where $\ep'$ is an arbitrarily small positive number. By the
inequalities \eqref{eq:variational_condition},
%
\begin{equation}
\Re g_+(s) - \frac{Ts^2}{4} \leq\frac{l}{2} \qquad\mbox{for $s \in
\realR\setminus[-M, M]$}.
\end{equation}
Hence, for all $w \in\Sigma$ and $s \in\realR\setminus[-M, M]$,
%
\begin{eqnarray}
\label{eq:estpn1spnwPearcey} &&\frac{1}{h^{(T; \tau)}_{n, n - 1}} p_{n - 1}(s) p_n(w)
e^{-{(T -
t_i)ns^2}/{2} - {t_j nw^2}/{2}} \frac{-e^{ixns - iynw}}{1 -
e^{2\pi
i nw - 2\tau\pi i}}
\nonumber
\\[-8pt]
\\[-8pt]
\nonumber
&&\qquad= \bigO\bigl(e^{n(- c_1 + \ep+ \ep')}\bigr).
\end{eqnarray}
If the factor $p_{n - 1}(s) p_n(w)$ in \eqref
{eq:estpn1spnwPearcey} is changed into $p_n(s) p_{n - 1}(w)$,
the estimate~\eqref{eq:estpn1spnwPearcey} still holds. So for
all $w \in\Sigma$ and $s \in\realR\setminus[-M, M]$,
%
\begin{eqnarray}
&&\frac{1}{h^{(T; \tau)}_{n, n - 1}} \frac{p_n(s) p_{n - 1}(w) - p_{n -
1}(s) p_n(w)}{s - w} e^{-{(T - t_i)ns^2}/{2} -
{t_j nw^2}/{2}} \nonumber\\
&&\quad{}\times \frac{-e^{ixns - iynw}}{1 - e^{2\pi i nw - 2\tau\pi i}}\\
&&\qquad =
\bigO\bigl(e^{n(-
c_1 + \ep+ \ep')}\bigr).\nonumber
\end{eqnarray}
Note that the integrand in \eqref{eq:Kminor} vanishes rapidly as $w
\in\infty$ along $\Sigma$. Thus, we have
%
\begin{equation}
\label{eq:K^minor_Pearcey} K^{\minor}_{t_i, t_j}(x, y) = \bigO
\bigl(e^{n(- c_1 + \ep+ \ep')}\bigr).
\end{equation}
The asymptotics \eqref{eq:essential_Pearcey}, \eqref
{eq:local_nonessential_Pearcey}, \eqref
{eq:KmajornonessentialPearcey} and \eqref{eq:K^minor_Pearcey},
together with \eqref{eq:K^major_Pearcey} and \eqref
{KKmajorKminor}, yield
%
\begin{equation}
\label{eq:asy_tilde_K_Pearcey} \tilde{K}_{t_i, t_j}(x, y) =
\frac{n^{{3}/{4}}}{d} \bigl(
\tilde {K}^{\Pearcey}_{-\tau_j, -\tau_i}(\eta, \xi) + \bigO\bigl(n^{-{1}/{4}}
\bigr) \bigr).
\end{equation}

At last, $\WWW_{[i, j)}(x, y)$ is defined in \eqref
{eq:Wi,jwithcirc} with explicit formula given in \eqref
{eq:Wcirc_ij_exact_formula}. It is $0$ when $t_j \leq t_i$ and when
$t_j > t_i$, a standard approximation technique gives that
%
\begin{equation}
\label{eq:asy_Wcirc_Pearcey} \WWW_{[i, j)}(x, y) =
\frac{n^{{3}/{4}}}{d}
\frac{1}{\sqrt
{2\pi
(\tau_j - \tau_i)}} e^{-{(\eta- \xi)^2}/{(2(\tau_j - \tau
_i))}} \bigl(1 + \bigO\bigl(n^{-{1}/{4}}\bigr)
\bigr).
\end{equation}

Comparing \eqref{eq:asy_tilde_K_Pearcey} and \eqref
{eq:asy_Wcirc_Pearcey} with \eqref{eq:Pearcey_kernel} and \eqref
{eq:nonessential_Pearcey}, we obtain \eqref{eq:tau_deformed_cusp}.

\subsection{Limiting tacnode process} \label{sec:kernel_tacnode}

With notation defined in \eqref{tac6}, we write \eqref
{eq:double_countour_for_tacnode} as
%
\begin{equation}
\label{tac8} \tilde{K}_{t_i, t_j}(x,y) = \frac{n}{2\pi}
\oint_{\Sg} \,dz \oint _{\Sg} \,dw J(z, w),
\end{equation}
where
%
\begin{eqnarray}
J(z, w)& =& \biggl( e^{-({nT}/{4})(z^2 + w^2)} \frac{p_n(z)p_{n - 1}(w)
- p_{n - 1}(z)p_n(w)}{h^{(T; \tau)}_{n, n - 1}(z - w)} \biggr)\nonumber\\
&&{}\times e^{n^{{2}/{3}}( {d^2}/{2}) (\tau_i z^2 - \tau_j w^2)}
e^{-in^{{1}/{3}} \,d (\xi z - \eta w)}
\\
&&{}\times\frac{e^{\pi i(nz-\tau)} e^{\pi i(nw-\tau)}}{(e^{2\pi
i(nz-\tau
)}-1)(e^{2\pi i(nw-\tau)}-1)}.
\nonumber
\end{eqnarray}

In this section, we define the shape of $\Sigma$ as follows. First, the
part of $\Sigma$ in the first quadrant consists of a horizontal ray
from $\infty\cdot e^{0}$ to $1 + i$, a line segment from $\sqrt{3} +
i$ to $(\sqrt{3} + i) d^{-1} n^{-1/3}$, and a line segment from
$(\sqrt
{3} + i) d^{-1} n^{-1/3}$ to $i d^{-1} n^{-1/3}$. The part of $\Sigma$
in the second quadrant is a reflection of that in the first quadrant
about the imaginary axis, and the part of $\Sigma$ in the lower
half-plane is a reflection of that in the upper half-plane about the
real axis. $\Sigma\cap\compC_+$ is oriented from right to left, and
$\Sigma\cap\compC_-$ is from left to right. See Figure~\ref{fig:Sigma_T_finite}. We denote $\Sigma_{\local}, \Sigma^{\upp
}_{\local
}, \Sigma^{\low}_{\local}$ as
%
\begin{eqnarray}
\Sigma_{\local} &= &\Sigma\cap\bigl\{ z \in\compC\mid\vert z \vert<
n^{-{1}/{4}} \bigr\}, \qquad\Sigma^{\upp}_{\local} =
\Sigma_{\local} \cap\compC_+,
\nonumber
\\[-8pt]
\\[-8pt]
\nonumber
 \Sigma^{\low}_{\local} &=&
\Sigma_{\local} \cap \compC_-.
\end{eqnarray}
To make the discussion about the apparent singularity $(z - w)^{-1}$
easier, we integrate $z$ on $\Sigma$ and $w$ on $\Sigma+ \frac{i}{2}
d^{-1} n^{-2/3}$ that is obtained by shifting $\Sigma$ above by $\frac
{i}{2} d^{-1} n^{-1/3}$; see Figure~\ref{fig:Sigma_T_local_shift}.
%

\begin{figure}

\includegraphics{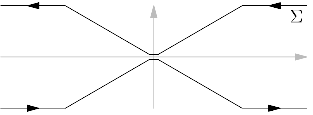}

\caption{Schematic figure of $\Sigma$.}
\label{fig:Sigma_T_finite}
\end{figure}

\begin{figure}

\includegraphics{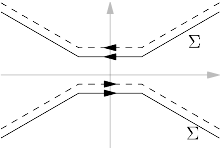}

\caption{Locally around $0$ with scaling $n^{-1/3}$. The solid curve
is $\Sigma$ and the dashed curve is $\Sigma+ \frac{i}{2} d^{-1} n^{-2/3}$.}
\label{fig:Sigma_T_local_shift}
\end{figure}

Applying the asymptotic formula \eqref{critical_CD_kernel} to the
integrand of \eqref{tac8} and taking the change of variables
%
\begin{equation}
\label{tac11} z=\frac{u}{dn^{1/3}} ,\qquad  w=\frac{v}{dn^{1/3}}
\end{equation}
we have
%
\begin{eqnarray}
\label{tac10} &&\frac{n}{2\pi} \oint_{\Sg_{\local}} \,dz
\oint_{\Sg_{\local} +
({i}/{2}) d^{-1} n^{-2/3}} \,dw J(z, w)\nonumber \\
&&\qquad= \frac{n^{{2}/{3}}}{4\pi^2 i d} \oint_{\Sg_T^*} \,du
\oint_{\Sg_T^* + {i}/{2}} \,dv \frac
{e^{
{1}/{2}(\tau_i u^2 - \tau_j v^2) - i(\xi u - \eta v)}}{u - v}
\nonumber
\\[-8pt]
\\[-8pt]
\nonumber
&&\qquad\quad{}\times %
\pmatrix{\displaystyle \frac{1}{1 - e^{2\pi i(nz - \tau)}} \vspace*{2pt}
\cr
\displaystyle\frac{1}{1 - e^{-2\pi i(nz - \tau)}} } %
^T \bolds\Psi(u;
\sigma)^{-1}\bolds\Psi(v; \sigma) \\
&&\qquad\quad{}\times\pmatrix{
\displaystyle\frac{1}{1 - e^{-2\pi i(nw - \tau)}} \vspace*{2pt}
\cr
\displaystyle\frac{-1}{1 - e^{2\pi i(nw - \tau)}} } %
 \bigl(1 +
\bigO\bigl(n^{-{1}/{4}}\bigr) \bigr),
\nonumber
\end{eqnarray}
where $\Sigma_T^*$ is the large but finite contour
%
\begin{equation}
\Sigma_T^* = \Sigma_T \cap N_{dn^{1/12}}(0)\qquad
\mbox{where $N_{dn^{1/12}}(0) = \bigl\{ z \mid\vert z \vert<
dn^{1/12} \bigr\}$},
\end{equation}
and $\Sigma_T$ is shown in Figure~\ref{fig:Sigma_T}.
Note that for $z \in\Sigma^{\upp}_{\local}$, or equivalently, $u
\in
\Sigma_T^* \cap\compC_+$,
%
\begin{eqnarray}
\label{eq:upper_factor} \frac{1}{1 - e^{2\pi i(nz - \tau)}} &=&
 1 + \bigO\bigl(e^{-{2n^{
{2}/{3}}}/{d}}\bigr),
\nonumber
\\[-8pt]
\\[-8pt]
\nonumber
\frac{1}{1 - e^{-2\pi i(nz - \tau)}}& =&
\bigO \bigl(e^{-{2n^{{2}/{3}}}/{d}} \bigr),
\end{eqnarray}
and for $z \in\Sigma^{\low}_{\local}$, or equivalently, $u \in
\Sigma
_T^* \cap\compC_-$,
%
\begin{eqnarray}
\label{eq:lower_factor} \frac{1}{1 - e^{2\pi i(nz - \tau)}} &=&
\bigO\bigl(e^{-{2n^{
{2}/{3}}}/{d}}\bigr),
\nonumber
\\[-8pt]
\\[-8pt]
\nonumber
\frac{1}{1 - e^{-2\pi i(nz - \tau)}} &=& 1 + \bigO
 \bigl(e^{-{2n^{{2}/{3}}}/{d}} \bigr).
\end{eqnarray}
For $w \in\Sigma^{\upp}_{\local} + \frac{i}{2} d^{-1} n^{-2/3}$ or
$\Sigma^{\upp}_{\local} + \frac{i}{2} d^{-1} n^{-2/3}$, we have
analogous result for $(1 - e^{\pm2\pi i(nw - \tau)})^{-1}$, and omit
the explicit formulas. Substituting \eqref{eq:upper_factor} and \eqref
{eq:lower_factor} and their counterparts for $w$ into \eqref{tac10} and
using the fact that $\det\bolds\Psi(\zeta;\sigma)\equiv1$, we
find that
%
\begin{eqnarray}
\label{eq:local_est_K_tac} %
&& \frac{n}{2\pi} \oint_{\Sg_{\local}} \,dz
\oint_{\Sg_{\local} +
({i}/{2}) d^{-1} n^{-2/3}} \,dw J(z, w)
\nonumber\\
&&\qquad= \frac{n^{{2}/{3}}}{4\pi^2 i d} \oint_{\Sg_T^*} \,du \oint _{\Sg
_T^* + {i}/{2}} \,dv
e^{(({\tau_i}/{2}) u^2 - ({\tau_j}/{2}) v^2)
- i(\xi u - \eta v)}
\nonumber
\\[-8pt]
\\[-8pt]
\nonumber
&&\qquad\quad{}\times{}\frac{f(u; \sigma)g(v;\sigma) - g(u; \sigma)f(v;
\sigma)}{(u - v)} \bigl(1 + \bigO \bigl(n^{-{1}/{4}}
\bigr) \bigr)
\\
&&\qquad= \frac{n^{{2}/{3}}}{d} K^{\tac}_{\tau_i, \tau_j}(\xi, \eta) \bigl(1 +
\bigO \bigl(n^{-{1}/{4}} \bigr) \bigr),\nonumber %
\end{eqnarray}
where $f(u; \sigma)$ and $g(u; \sigma)$ are defined in \eqref{tac18}.

By the estimates of $p_n(z)$ and $p_{n - 1}(z)/h^{(T; \tau)}_{n, n -
1}$ in Proposition~\ref{prop:asymptotics_of_OPs_critical} and \eqref
{eq:rough_estimate_critical_zero} in Proposition~\ref
{prop:asymptotics_of_OPs_origin}, we have that for all $z \in\Sigma$
and $w \in\Sigma+ \frac{i}{2} d^{-1} n^{-2/3}$
%
\begin{equation}
\frac{p_n(z)p_{n - 1}(w) - p_{n - 1}(z)p_n(w)}{h^{(T; \tau)}_{n, n -
1}(z - w)} = e^{ng(z) + ng(w) - nl} \bigO\bigl(n^{{2}/{3}}\bigr),
\end{equation}
where $g(z)$ is defined in \eqref{g12} and the $n^{2/3}$ factor comes
from $(z - w)^{-1}$. Hence, for all $z \in\Sigma$ and $w \in\Sigma+
\frac{i}{2} d^{-1} n^{-2/3}$,
%
\begin{eqnarray}
\bigl\vert J(z, w) \bigr\vert&= &e^{n(\tilde{g}(z) -( {T}/{4}) z^2 + \pi iz)
+ n(\tilde{g}(w) - ({T}/{4}) w^2 + \pi iw)} e^{n^{{2}/{3}} d^2
(\tau_i z^2 - \tau_j w^2)}
\nonumber
\\[-8pt]
\\[-8pt]
\nonumber
&&{}\times e^{-in^{{1}/{3}} d (\xi z - \eta w)} \bigO
\bigl(n^{{2}/{3}}\bigr),
\end{eqnarray}
where $\tilde{g}$ is defined by $g$ in \eqref{eq:defn_tilde_g}. By
direct calculation, we have that for $z \in\Sigma\setminus\Sigma
_{\local}$, $\Re(\tilde{g}(z) - T^2 z^2/4 + \pi iz)$ decreases as $z$
moves away from $0$. Hence, by standard argument of steepest-descent
method and the result of \eqref{eq:local_est_K_tac}, we have that
%
\begin{eqnarray}
\label{tac17a} %
\tilde{K}_{t_i, t_j}(x,y)& =& \frac{n}{2\pi}
\oint_{\Sg} \,dz \oint _{\Sg+ ({i}/{2}) d^{-1} n^{-2/3}} \,dw J(z, w)\nonumber
\\
&=& \frac{n}{2\pi} \oint_{\Sg_{\local}} \,dz \oint_{\Sg_{\local} +
({i}/{2}) d^{-1} n^{-2/3}} \,dw J(z, w) +
\bigO\bigl(e^{-cn^{{1}/{4}}}\bigr)
\\
&=& \frac{n^{{2}/{3}}}{d} \tilde{K}^{\tac}_{\tau_i, \tau_j}(\xi ,\eta;
\sigma) \bigl(1 + \bigO \bigl(n^{-{1}/{4}} \bigr) \bigr),\nonumber %
\end{eqnarray}
where $c$ is a positive constant.

At last, $\WWW_{[i, j)}(x, y)$ is defined in \eqref
{eq:Wi,jwithcirc} with explicit formula given in \eqref
{eq:Wcirc_ij_exact_formula}. It is $0$ when $t_j \leq t_i$ and when
$t_j > t_i$, a standard approximation technique gives that
%
\begin{equation}\quad
\label{tac19} \WWW_{[i, j)}(x, y) = \frac{n^{{2}/{3}}}{d}
\frac{1}{\sqrt
{2\pi
(\tau_j - \tau_i)}} e^{-{(\eta- \xi)^2}/{(2(\tau_j - \tau
_i))}} \bigl(1 + \bigO\bigl(n^{-{1}/{3}}\bigr)
\bigr).
\end{equation}
Comparing \eqref{tac17a} and \eqref{tac19} with \eqref
{eq:tacnode_kernel} and \eqref{eq:nonessential_Pearcey}, we prove
\eqref{tac17}.

\subsection{Proof of Theorem \texorpdfstring{\protect\ref{thmm:fixed_winding_number_correlations}}{1.4}} \label
{sec:proof_of_Theorem_fixed_winding_number}

For notational simplicity, we only consider the limiting
$2$-correlation functions, such that $t_1, t_2 \in(0, T)$ are two
times and $x, y$ are two locations on $\T$. We assume that $t_1, t_2,
x, y$ are expressed by \eqref{eq:coefficients_Pearcey} with $i = 1$, $j
= 2$, and then
%
\begin{eqnarray}
&&\bigl(R^{(n)}_{0 \to T}\bigr)_{\omega}(x; y;
t_1, t_2)
\nonumber
\\[-4pt]
\\[-12pt]
\nonumber
&&\qquad= \lim_{\Delta x \to0}
\frac
{1}{(\Delta x)^2} \Prob \left( %
\begin{array}{c}
\mbox{there is a particle
in $[x, x + \Delta x)$ at time $t_1$,} \\
\mbox{there is
a particle in $[y, y + \Delta x)$ at time $t_2$,} \\
\mbox{and the total winding number is $\omega$} %
\end{array}
 \right).
\end{eqnarray}
From \eqref{eq:sum_of_corr_func_over_winding_number}, we have
%
\begin{eqnarray}
\label{eq:special_case_sum_of_corr_func_winding_number}&&
 \lim_{ n \to\infty} \sum_{\omega\in\intZ}
\bigl(R^{(n)}_{0 \to
T}\bigr)_{\omega
}(x; y;
t_1, t_2) e^{2\pi\omega(\tau+ \epsilon(n))i} \biggl( \frac
{d}{n^{
{3}/{4}}}
\biggr)^2
\nonumber
\\[-8pt]
\\[-8pt]
\nonumber
&&\qquad= \lim_{n \to\infty} \frac{R_n(T; \tau)}{R_n(T;
\epsilon(n))}
R^{(n)}_{0 \to T}(x; y; t_1, t_2; \tau)
\biggl( \frac
{d}{n^{
{3}/{4}}} \biggr)^2,
\end{eqnarray}
where $R^{(n)}_{0 \to T}(x; y; t_1, t_2; \tau)$ is a special case of
the $\tau
$-deformed joint correlation function defined in \eqref
{eq:defn_tau_deformed_corr_func}.

By the determinantal formula \eqref{eq:defn_Rmelon} and the asymptotic
result \eqref{eq:tau_deformed_cusp}, we have for all $\tau\in[0, 1]$,
%
\begin{equation}
\lim_{ n \to\infty} R^{(n)}_{0 \to T}(x; y;
t_1, t_2; \tau) = %
\left|\matrix{
K^{\Pearcey}_{-\tau_1, -\tau_1}(\xi, \xi) & K^{\Pearcey}_{-\tau_2,
-\tau_1}(
\eta, \xi)\vspace*{2pt}
\cr
K^{\Pearcey}_{-\tau_1, -\tau_2}(\xi, \eta) &
K^{\Pearcey}_{-\tau_2,
-\tau_2}(\eta, \eta) } \right|%
,
\end{equation}
and on the other hand by \eqref{winu1} and \eqref{wnsuper}, we have
%
\begin{eqnarray}
&&\lim_{n \to\infty} \frac{R_n(T; \tau)}{R_n(T; \epsilon(n))}\nonumber\\
&&\qquad= \lim
_{n \to\infty} \sum_{\omega\in\intZ} \mathbb{P}(
\mbox{Total winding number equals } \omega) e^{2\pi\omega(\tau+ \epsilon(n))i}
\\
&&\qquad= \sum_{\omega\in\intZ} q^{\omega^2} \sqrt{
\frac{\pi}{2 \tilde
{\K
}}} e^{2\pi\omega(\tau+ \epsilon(n))i}.\nonumber %
\end{eqnarray}
Hence, a comparison of Fourier coefficients on both sides of \eqref
{eq:special_case_sum_of_corr_func_winding_number} shows that
%
\begin{eqnarray}
&&\lim_{ n \to\infty} \bigl(R^{(n)}_{0 \to T}
\bigr)_{\omega}(x; y; t_1, t_2) \biggl(
\frac
{d}{n^{{3}/{4}}} \biggr)^2
\nonumber
\\[-8pt]
\\[-8pt]
\nonumber
&&\qquad= q^{\omega^2} \sqrt{
\frac{\pi}{2
\tilde
{\K}}} %
\left|\matrix{ K^{\Pearcey}_{-\tau_1, -\tau_1}(\xi,
\xi) & K^{\Pearcey}_{-\tau_2,
-\tau_1}(\eta, \xi) \vspace*{2pt}
\cr
K^{\Pearcey}_{-\tau_1, -\tau_2}(\xi, \eta) & K^{\Pearcey}_{-\tau_2,
-\tau_2}(
\eta, \eta) }\right| %
,
\end{eqnarray}
which is the desired result. Thus, we prove Theorem~\ref
{thmm:fixed_winding_number_correlations} in the $n = 2$ case.

\section{Interpolation problem and Riemann--Hilbert problem associated
to discrete Gaussian orthogonal polynomials} \label{sec:Riemann-Hilbert}

\subsection{Equilibrium measure and the $g$-function}
In this subsection, we prove the results presented in Section~\ref{results_eq_measure} for the supercritical case $T>T_c$.
The existence and uniqueness of the equilibrium measure associated to
the potential $Tx^2/2$ that satisfies the minimization problem \eqref
{eq1} and \eqref{eq:constraint} is proved in \citet{Kuijlaars00}, along
with several analytic properties. Thus, if we find a probability
measure $\nu_T$ with continuous density function $\rho_T(x)$ such that
the associated $g$-function satisfies the variational condition \eqref
{eq:variational_condition}, then it is the unique equilibrium measure.
For $T \leq T_c = \pi^2$, it is straightforward to verify that the
well-known semicircle law \eqref{eq2} and the $g$-function \eqref{g12}
satisfy the variational condition \eqref{eq:variational_condition}, so
the equilibrium measure is given by \eqref{eq2}. Thus, this subsection
is dedicated to the construction of the equilibrium measure and the
derivative of the $g$-function for $T > T_c = \pi^2$. The $g$-function
is then determined by its derivative up to the Lagrange multiplier $l$.
Our strategy is to construct a probability measure $\nu_T$ with
continuous density $d\nu_T(x) = \rho_T(x) \,dx$ together with the
derivative of the associated $g$-function, such that $\nu_T$ is
supported on an interval $[-\beta, \beta]$, and has a saturated region
$[-\alpha, \alpha]$, that is, $\rho_T(x) = 0$ for $x \in\realR
\setminus(-\beta, \beta)$, $\rho_T(x) = 1$ for $x \in[-\alpha,
\alpha
]$ and $0 < \rho_T(x) < 1$ for $x \in(-\beta, \alpha) \cup(\alpha,
\beta)$, and then verify that the probability measure satisfies the
variational condition \eqref{eq:variational_condition}. Therefore, we
conclude that the construction of the equilibrium measure is valid.%

The derivative of the $g$-function is expressed as
%
\begin{equation}
g'(z) = \int^{\beta}_{-\beta}
\frac{1}{z - x} \rho_T(x) \,dx,\qquad z \in\compC\setminus[-\beta,
\beta],
\end{equation}
and so the equilibrium measure $\nu_T = \rho_T(x) \bolds\chi
_{[-\beta
, \beta]}(x) \,dx$ is given as
%
\begin{equation}
\rho_T(x) = \frac{-1}{\pi} \Im g'_+(x) =
\frac{1}{\pi} \Im g'_-(x)\qquad \mbox{for $x \in[-\beta, \beta]$},
\end{equation}
where $g'_+(x)$ and $g'_-(x)$ are the limiting values from the upper
and lower half-planes, respectively.
That the measure $\nu_T$ has total measure $1$ is equivalent to
%
\begin{equation}
\label{eq:gatinfinity} g'(z) = \frac{1}{z} + \bigO
\bigl(z^{-2}\bigr)\qquad \mbox{as $z \to\infty$}.
\end{equation}
The variational problem \eqref{eq:variational_condition} implies
%
\begin{eqnarray}
g'_+(x) + g'_-(x) &=& Tx\qquad\mbox{for $x \in(-\beta,
-\alpha) \cup(\alpha, \beta)$}, \label{eq:jumpgsides}
\\
g'_+(x) - g'_-(x)& =& -2\pi i\qquad \mbox{for $x \in(-
\alpha, \alpha )$}. \label{eq:jumpgcenter}
\end{eqnarray}

To construct $g'(z)$, we use the incomplete elliptic integrals $F(z;
k)$ and $E(z; k)$ and the complete elliptic integrals $\K= \K(k)$ and
$\E= \E(k)$ introduced in \eqref{eq:defn_of_K} and \eqref
{eq:complete_elliptic}. They have the properties that
%
\begin{eqnarray}\label{eq:jump_FE_right}
F_+(x; k) + F_-(x; k) &= &2\K,\qquad  E_+(x; k) + E_-(x; k) = 2\E
\nonumber
\\[-8pt]
\\[-8pt]
\eqntext{\mbox{for $x \in
\bigl(1, k^{-1}\bigr)$}, }
\\
\label{eq:jump_FE_left}
F_+(x; k) + F_-(x; k) &=& -2\K,\qquad E_+(x; k) + E_-(x; k) = -2\E
\nonumber
\\[-8pt]
\\[-8pt]
\eqntext{\mbox{for $x \in
\bigl(-k^{-1}, -1\bigr)$}, }
\\
\label{eq:jump_FE_outer}
F_+(x; k) - F_-(x; k) &=& 2i\K',\qquad E_+(x; k) - E_-(x; k) = 2i\bigl(\K
' - \E'\bigr)
\nonumber
\\[-8pt]
\\[-8pt]
\eqntext{ \mbox{for $x \in\realR\setminus
\bigl(-k^{-1}, k^{-1}\bigr)$}.}
\end{eqnarray}
Here, we use the notation
%
\begin{equation}
\K' = \K\bigl(k'\bigr),\qquad \E' = \E
\bigl(k'\bigr) \qquad\mbox{where } k' = \sqrt {1 -
k^2}.
\end{equation}
Identities \eqref{eq:jump_FE_right} and \eqref{eq:jump_FE_left} can be
checked by straightforward computation, and \eqref{eq:jump_FE_outer}
can be checked with the help of \citet{Gradshteyn-Ryzhik07},
3.152-9, page 280 and 3.169-17, page
209. The identity \eqref{eq:jump_FE_outer} also
comes from the imaginary periods of $F(z; k)$ and $E(z; k)$; see \citet{Erdelyi-Magnus-Oberhettigner-Tricomi81},
Section~13.7, page~314. For
fixed $\alpha$ and $\beta$, let
%
\begin{equation}
k = \frac{\alpha}{\beta}.
\end{equation}
With the help of Legendre's relation
[\citet{Erdelyi-Magnus-Oberhettigner-Tricomi81},
Section~13.8, page 320,
Formula (15)],
%
\begin{equation}
\label{eq:Legendre} \K\E' + \K' \E- \K\K'
= \frac{\pi}{2},
\end{equation}
we find that when $g'(z)$ is given by
%
\begin{equation}
\label{eq:requirement_of_g'} g'(z) = %
\cases{\displaystyle \frac{Tz}{2} + 2
\E F\biggl(\frac{z}{\alpha}; k\biggr) - 2\K E\biggl(\frac{z}{\alpha
}; k
\biggr) - \pi i, &\quad  $\mbox{for $z \in\compC_+$}$, \vspace*{2pt}
\cr
\displaystyle\frac{Tz}{2} + 2\E F\biggl(\frac{z}{\alpha}; k\biggr) - 2\K E\biggl(
\frac{z}{\alpha
}; k\biggr) + \pi i, &\quad  $\mbox{for $z \in\compC_-$}$, }\hspace*{-30pt}
\end{equation}
it satisfies \eqref{eq:jumpgsides} and \eqref{eq:jumpgcenter},
and it is also well defined on $(-\infty, -\beta) \cup(\beta,
\infty)$
by analytic continuation. To make \eqref{eq:gatinfinity} hold, we
need to choose the correct values for $\alpha$ and $\beta$. As $z \to
\infty$, the asymptotic behaviors of $F(z; k)$ and $E(z; k)$ are
%
\begin{eqnarray}
F(z; k)& = &i\K' + \frac{1}{kz} + \bigO\bigl(z^{-2}
\bigr), \label
{eq:F_expansion_infty}
\\
E(z; k)& =& kz + i\bigl(\K' - \E'\bigr) +
\frac{k^{-1} - k}{2z} + \bigO \bigl(z^{-2}\bigr). \label{eq:E_expansion_infty}
\end{eqnarray}
The constant term in \eqref{eq:F_expansion_infty} is obtained by
%
\begin{eqnarray}
\mathop{\lim_{z \to\infty}}_{ z \in\compC_+} F(z; k)& = &\int
^{i
\cdot\infty}_0 \frac{ds}{\sqrt{(1 - s^2)(1 - k^2 s^2)}}\nonumber\\
& =& i \int
^{\infty}_0 \frac{dt}{\sqrt{(1 + t^2)(1 + k^2 t^2)}}
\\
& = & i F\bigl(1;
\sqrt {1 - k^2}\bigr) = i\K',\nonumber
\end{eqnarray}
where evaluation of the elliptic integral is done by \citet{Gradshteyn-Ryzhik07},
3.152-2, page
279. The $z^{-1}$ term in \eqref
{eq:F_expansion_infty} follows the asymptotics of the integrand in the
defining formula \eqref{eq:defn_of_K} of $F(z; k)$. The $z$ term in
\eqref{eq:E_expansion_infty} is obvious, and the constant term is
given by
%
\begin{eqnarray}
\label{eq:elliptic_int_with_k}
 \mathop{\lim_{z \to\infty}}_{ z \in\compC_+} E(z; k) -
kz &=& \int^{i \cdot\infty}_0 \biggl(\sqrt{
\frac{1 - k^2 s^2}{1 - s^2}} - k \biggr) \,ds
\nonumber
\\[-8pt]
\\[-8pt]
\nonumber
&=& \frac{i}{2} \int
^{\infty}_{-\infty} \biggl(\sqrt{\frac{1 + k^2
t^2}{1 + t^2}} - k
\biggr) \,dt.
\end{eqnarray}
To evaluate the integral on the right-hand side of \eqref
{eq:elliptic_int_with_k}, we define the pair of contours (see Figure~\ref{fig:contour_ellipt_int}),
%
%
\begin{figure}

\includegraphics{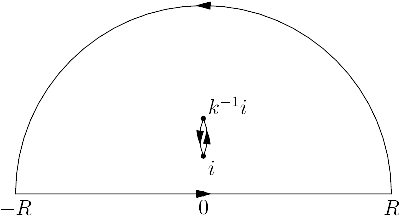}

\caption{The contours $C_R$ and $C_{1, k^{-1}}$.}
\label{fig:contour_ellipt_int}
\end{figure}
%
\begin{eqnarray}
C_R &=& [-R, R] \cup\bigl\{ R e^{i\theta} \mid\theta
\in[0, \pi] \bigr\}\qquad \mbox{counterclockwise},\hspace*{-20pt}
\nonumber\\
C_{1, k^{-1}} &=& \mbox{contour starting from $i$, along the right-hand side
of the}\hspace*{-20pt}
\nonumber
\\[-8pt]\hspace*{-20pt}
\\[-8pt]
\nonumber
&&\mbox{imaginary axis, to $k^{-1}i$, and then along the left-hand
side}\hspace*{-20pt}
\\
&&\mbox{of the imaginary axis, back to $i$}.\hspace*{-20pt}\nonumber
\end{eqnarray}
Then by the contour integral technique and \citet{Gradshteyn-Ryzhik07},
3.169-17, page
309,
%
\begin{eqnarray}
\label{eq:E-ikt_evaluated} %
&&\int^{\infty}_{-\infty} \sqrt{
\frac{1 + k^2 t^2}{1 + t^2}} - k\, dt \nonumber\\
&&\qquad= \lim_{R \to\infty} \oint_{C_R}
\sqrt{\frac{1 + k^2 t^2}{1 + t^2}} \,dt = \oint_{C_{1, k^{-1}}} \sqrt{\frac{1 + k^2 t^2}{1 + t^2}}
\,dt
\nonumber
\\[-8pt]
\\[-8pt]
\nonumber
&&\qquad= 2 \int^{k^{-1}}_1 \sqrt{\frac{1 - k^2 t^2}{1 - t^2}} \,dt \\
&&\qquad=
2\bigl(F\bigl(1; \sqrt{1 - k^2}\bigr) - E\bigl(1; \sqrt{1 -
k^2}\bigr)\bigr) = 2\bigl(\K' - \E'
\bigr),\nonumber %
\end{eqnarray}
and we get the result. The $z^{-1}$ term of \eqref
{eq:E_expansion_infty} is obtained analogously to the $z^{-1}$ term of
\eqref{eq:F_expansion_infty}.

Then as $z \to\infty$ in $\compC_+$,
%
\begin{eqnarray}
g'(z) &=& \biggl( \frac{T}{2} - \frac{2k\K}{\alpha} \biggr) z
+ 2i \biggl( \K' \E+ \K\E' - \K\K' -
\frac{\pi}{2} \biggr)
\nonumber
\\[-8pt]
\\[-8pt]
\nonumber
&&{}+ 2\alpha \biggl( \frac
{\E}{k} -
\frac{(1 - k^2)\K}{2k} \biggr) \frac{1}{z} + \bigO\bigl(z^{-2}
\bigr).
\end{eqnarray}
Note that the constant term of $g'(z)$ vanishes automatically by
Legendre's relation~\eqref{eq:Legendre}. For $k = \alpha/\beta$, the
identity \eqref{eq:gatinfinity} is satisfied when $\al$ and $\be$
are given by \eqref{eq:defn_of_beta_alpha} and \eqref{eq:T_parametrized_by_k}.

By Lemma~\ref{lem:T_parametrized_by_k}, the relation \eqref
{eq:T_parametrized_by_k} is a 1--1 correspondence between $T > T_c = \pi
^2$ and $k \in(0, 1)$. Thus, for each $T = T(k) > T_c$, there are
well-defined $\alpha, \beta$ and $\rho_T$ given by \eqref
{eq:defn_of_beta_alpha} and \eqref{eq:gatinfinity}. By the
construction of $\rho_T$, especially \eqref{eq:jumpgsides}, we have
that the measure $d\nu_T$ with density $\rho_T(x)\bolds\chi
_{[-\beta,
\beta]}(x)$ has total measure $1$, and satisfies the variational
condition on $[\alpha, \beta]$ given that the Lagrange multiplier $l$
is properly chosen and $0< \rho_T(x) < 1$ for all $x \in(\alpha,
\beta
)$. By the symmetry of $\nu_T$ about the origin, we finish the
verification that $\nu_T$ is the equilibrium measure. Additionally, we
have the following lemma, which states that the equilibrium measure is
\emph{regular} in the sense of \citet{Bleher-Liechty11}.

\begin{lem} \label{lem:variational_verification}
\textup{(a)} 
$0< \rho_T(x) < 1$ for all $x \in(\alpha, \beta)$.\vspace*{-6pt}
\begin{longlist}[(b)]
\item[(b)]
$g_+(x) + g_-(x) -\frac{Tx^2}{2} -l > 0$ for $x \in[0, \alpha)$.
\item[(c)]
$2g(x) -\frac{Tx^2}{2} -l < 0$ for $x \in(\beta, \infty)$.
\item[(d)]
There exist constants
$c_1$ and $c_2$ such that
%
\begin{equation}
\rho_T(x) = c_1\sqrt{\be-x} \bigl(1+\bigO\bigl((\be-x)
\bigr) \bigr)\qquad \mbox {as } x\to\be\mbox{ from the left}
\end{equation}
and
%
\begin{eqnarray}
1-\rho_T(x) = c_2\sqrt{x-\al} \bigl(1+\bigO\bigl((x-
\al)\bigr) \bigr)
\nonumber
\\[-8pt]
\\[-8pt]
\eqntext{\mbox{as } x\to\al\mbox{ from the right}.}
\end{eqnarray}
\end{longlist}
\end{lem}

We finish this subsection by the proof of Lemmas \ref
{lem:T_parametrized_by_k} and \ref{lem:variational_verification}.

\begin{pf*}{Proof of Lemma~\ref{lem:T_parametrized_by_k}}
The two limits in \eqref{eq:limitsofTk} are straightforward to
check from the integral formulas \eqref{eq:complete_elliptic} of $\K$
and $\E$. To see the monotonicity, we use \citet{Byrd-Friedman71},
710.00 and 710.02, page
282, and have
%
\begin{eqnarray}
\frac{d}{dk} \bigl(\K(k) \E(k)\bigr) &=&
\frac{\E^2 - (1 - k^2) \K^2}{k(1 - k^2)}
\nonumber
\\[-8pt]
\\[-8pt]
\nonumber
&=&
\frac{1}{k(1 - k^2)} \int^1_0
\frac{\sqrt{1 - k^2 s^2} - \sqrt{1 -
k^2}}{\sqrt{(1 - s^2)(1 - k^2 s^2)}} \,ds > 0,
\end{eqnarray}
which proves the monotonicity.
\end{pf*}

\begin{pf*}{Proof of Lemma~\ref{lem:variational_verification}}
For part (a), since $\rho_T(x)$
on $(0, \alpha)$ is expressed by $\Lambda_0(x; k)$ for $x \in(0, 1)$
in \eqref{eq:alternative_rho_T}, we only need to show that $\Lambda
_0(x; k) \in(0, 1)$ for $x \in(0, 1)$. By \citet{Byrd-Friedman71},
151.01, page
36, we have that $\Lambda_0(0; k) = 0$ and $\Lambda
_0(1; k) = 1$. We need only to show that $\Lambda_0(x; k)$ is strictly
increasing on $(0, 1)$. By \citet{Byrd-Friedman71},
730.04, page 284, this
is implied by the inequality $\E- (1 - k^2)x^2 \K> 0$ for all $x \in
(0, 1)$. The inequality is proved as
%
\begin{equation}
\E- \bigl(1 - k^2\bigr)x^2 \K= \int^1_0
\frac{\sqrt{1 - k^2 s^2} - x^2\sqrt{1
- k^2}}{\sqrt{(1 - s^2)(1 - k^2 s^2)}} > 0.
\end{equation}

For parts (b) and (c), we note that since $g_+(x) +
g_-(x) - Tx^2/2 - l$, which becomes $2g(x) - Tx^2/2 -l$ for $x > \beta
$, is a continuous function and $l$ is chosen so that it vanishes for
$x \in[\alpha, \beta]$, it suffices to show the inequalities
%
\begin{eqnarray}\label{eq:decreasing_g'-Tx_saturated}\qquad
\frac{1}{4}\bigl(g'_+(x) + g'_-(x) - Tx
\bigr) &=& \E F \biggl(\frac{x}{\alpha}; k \biggr) - \K E \biggl(
\frac{x}{\alpha}; k \biggr) < 0,
\nonumber
\\[-8pt]
\\[-8pt]
\eqntext{ x \in(0, \alpha),}
\\
\label{eq:decreasing_g'-Tx_outside}\frac{1}{4}\bigl(2g'(x) - Tx\bigr) &=& -\E\int
^{{x}/{\alpha}}_{k^{-1}} \frac
{ds}{\sqrt{(s^2 - 1)(k^2 s^2 - 1)}}
\nonumber
\\[-8pt]
\\[-8pt]
\nonumber
&&{} - \K\int
^{{x}/{\alpha
}}_{k^{-1}} \frac{\sqrt{k^2 s^2 - 1}}{\sqrt{s^2 - 1}} \,ds < 0, \qquad x \in (
\beta, \infty).
\end{eqnarray}
The inequality \eqref{eq:decreasing_g'-Tx_outside} obviously holds. To
prove \eqref{eq:decreasing_g'-Tx_saturated}, we use \citet{Byrd-Friedman71},
414.01, page
229,
%
\begin{eqnarray}
&&\E F \biggl(\frac{x}{\alpha}; k \biggr) - \K E \biggl(
\frac{x}{\alpha}; k \biggr)\nonumber\\
&&\qquad = \frac{1}{\beta x} \sqrt{\bigl(\beta-
x^2\bigr) \bigl(\alpha^2 - x^2\bigr)} \biggl(
\K- \Pi_1 \biggl( -\frac{x^2}{\beta^2}; k \biggr) \biggr)
\\
&&\qquad= \frac{-1}{\beta x} \sqrt{\bigl(\beta- x^2\bigr) \bigl(
\alpha^2 - x^2\bigr)} \int^1_0
\frac{({x^2}/{\beta^2}) s^2 \,ds}{(1 - ({x^2}/{\beta^2}) s^2)
\sqrt
{(1 - s^2)(1 - k^2 s^2)}},\nonumber
\end{eqnarray}
which is clearly negative for $x \in(0, \alpha)$.

Part (d) is easy to check using
formula \eqref{rho_T_formula}.
\end{pf*}

\subsection{Interpolation problem and outline of the steepest descent analysis}
The orthogonal polynomials \eqref{eq:defn_of_discrete_Gaussian_OP} are
encoded in the following interpolation problem (IP).
For a given $n=0,1,\ldots,$ find a $2\times2$ matrix-valued function
$\mathbf P_n(z)=(\mathbf P_n(z)_{ij})_{1\le i,j\le2}$ with the
following properties:
\begin{longlist}[1.]
\item[1.]
{\textit{Analyticity}}: $\mathbf P_n(z)$ is an analytic function of $z$ for
$z\in\C\setminus L_{n,\tau}$.
\item[2.]
{\textit{Residues at poles}}: At each node $x\in L_{n,\tau}$, the elements
$\mathbf P_n(z)_{11}$ and
$\mathbf P_n(z)_{21}$ of the matrix $\mathbf P_n(z)$ are analytic
functions of $z$, and the elements $\mathbf P_n(z)_{12}$ and
$\mathbf P_n(z)_{22}$ have a simple pole with the residues,
%
\begin{equation}
\label{IP1} \mathop{\mathrm{ Res}}_{z=x} \mathbf P_n(z)_{j2}=
\frac{1}{n} e^{-
{nTx^2}/{2}} \mathbf P_n(x)_{j1},\qquad
j=1,2.
\end{equation}
\item[3.]
{\textit{Asymptotics at infinity}}: There exists a function $r(x)>0$ on
$L_{n,\tau}$ such that
%
\begin{equation}
\label{IP2a} \lim_{x\to\infty} r(x)=0,
\end{equation}
and such that as $z\to\infty$, $\mathbf P_n(z)$ admits the asymptotic
expansion,
%
\begin{eqnarray}
\label{IP2} \mathbf P_n(z)\sim \biggl( I+\frac{\mathbf P_1}{z}+
\frac{\mathbf
P_2}{z^2}+\cdots \biggr) %
\pmatrix{ z^n & 0
\vspace*{2pt}
\cr
0 & z^{-n} } %
,
\nonumber
\\[-8pt]
\\[-8pt]
  \eqntext{\displaystyle z\in\C{}\Big\backslash{} \Biggl[
\bigcup_{x\in L_{n, \tau}}^\infty D \bigl(x,r(x) \bigr)
\Biggr],}
\end{eqnarray}
where $D(x,r(x))$ denotes a disk of radius $r(x)>0$ centered at $x$ and
$I$ is the identity matrix.
\end{longlist}

The unique solution to
the IP is
%
\begin{equation}
\label{IP3} \mathbf P_n(z)= %
\pmatrix{
p_{n,n}^{(T;\tau)}(z) & \bigl(Cp_{n,n}^{(T;\tau)}
\bigr) (z) \vspace *{2pt}
\cr
\bigl(h^{(T; \tau)}_{n,n-1}
\bigr)^{-1}p_{n,n-1}^{(T;\tau)}(z) & \bigl(h^{(T; \tau
)}_{n,n-1}
\bigr)^{-1} \bigl(Cp_{n,n-1}^{(T;\tau)} \bigr) (z) }
,
\end{equation}
where the weighted discrete Cauchy transform $C$ is defined in \eqref
{eq:def_Cauchy_trans}. The normalizing constants in \eqref
{eq:defn_of_h_nk} and the recurrence coefficients \eqref
{eq:three_term_recurrence} are encoded in the matrices $\mathbf P_1$
and $\mathbf P_2$ in the expansion \eqref{IP2}. Namely, we have
%
\begin{equation}
\label{IP4} h^{(T; \tau)}_{n, n}=[\mathbf P_1]_{12}
, \qquad\bigl(h^{(T; \tau
)}_{n, n-1} \bigr)^{-1}=[\mathbf
P_1]_{21}
\end{equation}
and
%
\begin{equation}
\label{IP5} \be^{(T; \tau)}_{n, n-1}=\frac{[\mathbf P_2]_{21}}{[\mathbf
P_1]_{21}}-[\mathbf
P_1]_{11}.
\end{equation}

\begin{figure}[b]

\includegraphics{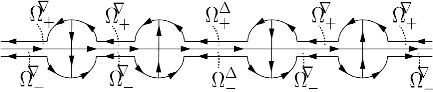}

\caption{The contour $\Sg_S$. The horizontal line is $\R$ and the
vertical segments pass through $\pm\al$ and $\pm\be$. The remaining
pieces of the contour are $\Ga_{\pm}$ which pass close to $\R$ at a
distance of $\ep n^{-\delta}$ except close to the turning points $\pm
\al$
and $\pm\be$. The regions $\Om_{\pm}^\De$ are bounded by the real line
and $\Ga_{\pm}$ with $|\Re z|<\al$, and the regions $\Om_{\pm
}^\nabla$
are bounded by the real line and $\Ga_{\pm}$ with $|\Re z|>\al$.}
\label{Contour_Sg_S}
\end{figure}

The steepest descent analysis of the IP for a general class of
orthogonal polynomials is described in \citet{Bleher-Liechty11} in the
case $\tau=0$ [see also \citet{Baik-Kriecherbauer-McLaughlin-Miller07}
for polynomials orthogonal on a finite lattice]. For the discrete
Gaussian orthogonal polynomials the analysis for a general $\tau$ was
given in \citet{Liechty12} for the case $T=T_c+o(1)$ as $n\to\infty$.
The analysis consists of a sequence of transformations
%
\begin{equation}
\mathbf P_n \rightarrow\mathbf R_n \rightarrow\mathbf
T_n \rightarrow \mathbf S_n \rightarrow\mathbf
X_n.
\end{equation}
The first transformation $\mathbf P_n \rightarrow\mathbf R_n$ reduces
the IP to a Riemann--Hilbert problem (RHP). The second transformation
$\mathbf R_n \rightarrow\mathbf T_n$ uses the $g$-function to give a
RHP which approaches the identity matrix as $z\to\infty$. The third
transformation $\mathbf T_n \rightarrow\mathbf S_n$ is local and
involves transformations only close to the support of the equilibrium
measure. The RHP for $\mathbf S_n$ can be approximated by RHPs for
which we can write explicit solutions in different regions of the
complex plane, and $\mathbf X_n$ is uniformly close to the identity matrix.

In the supercritical case $T>T_c$, one can make the reduction to a RHP
in the following way. Fix some $\ep>0$ and some $0<\delta<1$. Let
$\Ga_+$
(resp., $\Ga_-$) be a contour from $e^{i0} \cdot\infty$ to $e^{i\pi}
\cdot\infty$ (resp., $e^{-i \pi}\cdot\infty$ to $e^{i 0}\cdot
\infty$)
which lies in the upper (resp., lower) half-plane and sits at a
distance $\ep n^{-\delta}$ from the real line except close to the turning
points $\pm\al$ and $\pm\be$, where it maintains a fixed distance
$\ep
$ from these points; see Figure~\ref{Contour_Sg_S}. We let $\Om_{\pm
}^\De$ be the region bounded by the real line and $\Ga_{\pm}$ with
$|\Re z|<\al$, and $\Om_{\pm}^\nabla$ the region bounded by the real
line and $\Ga_{\pm}$ with $|\Re z|>\al$. We make the reduction of the
IP to a RHP and the transformations to the RHP as in \citet{Liechty12};
see \citet{Liechty12}, Figure~2 and equations (4.27), (4.28), (4.32).
Note that the lattice shift parameter which we call $\tau$ is called
$(-\al)$ in \citet{Liechty12}.

Let us briefly describe the explicit transformations involved in the
steepest descent analysis. Introduce the functions
%
\begin{equation}
\label{red1} \Pi(z):=\frac{\sin(n\pi z-\tau\pi)}{n\pi} , \qquad G(z):=g_+(x)-g_-(z) ,
\end{equation}
where $g_\pm(z)$ are defined first on $\R$ as the limiting values of
the $g$-function from $\C_{\pm}$, and then extended to a small
neighborhood of $\R$ by analytic continuation. Notice that the function
$G(z)$ is also given by the integral formula \eqref{eq:differencegggg}.
The transformations described above involve the matrices
%
\begin{eqnarray}
\mathbf D^u_{\pm}(z) &=&
\pmatrix{ 1 & -\displaystyle\frac{e^{-({nT}/{2})z^2}}{n\Pi(z)}e^{\pm i\pi(n z-\tau)} \vspace*{2pt}
\cr
0
& 1 } %
,\nonumber \\
\mathbf D^l_{\pm}(z)& =& %
\pmatrix{ \Pi(z)^{-1} & 0 \vspace*{2pt}
\cr
-ne^{({nT}/{2})z^2}e^{\pm i\pi(n z-\tau)}
& \Pi(z) } %
,
\nonumber
\\[-8pt]
\\[-8pt]
\nonumber
j_\pm(z)&=& %
\pmatrix{ 1 & 0 \vspace*{2pt}
\cr
e^{\mp nG(z)} & 1 } %
,\\
 \mathbf A_\pm(z) &=& %
\pmatrix{\displaystyle \mp\frac{1}{2n\pi i}e^{\mp i\pi(n z-\tau)} &
 0 \vspace *{2pt}
\cr
0 & \mp2n
\pi i e^{\pm i\pi(n z-\tau)} } \nonumber %
.
\end{eqnarray}
After the first two transformations of the IP, the matrix $\mathbf
S_n(z)$ is defined as
%
\begin{eqnarray}\qquad
\label{eq:explicit_formula_S} \mathbf S_n(z)=\cases{ %
e^{-({nl}/{2})\sigma_3}
\pmatrix{ 1 & 0 \vspace*{2pt}
\cr
0 & -2\pi i } %
\mathbf
P_n(z) \mathbf D_{\pm}^l(z) %
\pmatrix{ 1 & 0 \vspace*{2pt}
\cr
0 & -2\pi i } %
^{-1}\vspace*{2pt}\cr
\quad{}\times
e^{-n(g(z)-{l}/{2})\sigma_3}\mathbf A_\pm(z),\vspace*{2pt}\cr
\qquad \mbox{for } z\in
\Om_\pm^\De,
\vspace*{2pt}\cr
e^{-({nl}/{2})\sigma_3} %
\pmatrix{ 1 & 0 \vspace*{2pt}
\cr
0 & -2\pi i }
\mathbf P_n(z) \mathbf D_{\pm}^u(z)
\pmatrix{ 1 & 0 \vspace*{2pt}
\cr
0 & -2\pi i } %
^{-1}\vspace*{2pt}\cr
\quad{}\times
e^{-n(g(z)-{l}/{2})\sigma_3}j_\pm(z)^{\mp1},
\vspace*{2pt}\cr
\qquad \mbox{for } z\in\Om_\pm^\nabla \mbox{ and } \al \le|
\Re z| \le\be,
\vspace*{2pt}\cr
e^{-({nl}/{2})\sigma_3} %
\pmatrix{ 1 & 0 \vspace*{2pt}
\cr
0 & -2\pi i }
\mathbf P_n(z) \mathbf D_{\pm}^u(z)
\pmatrix{ 1 & 0 \vspace*{2pt}
\cr
0 & -2\pi i } %
^{-1}
\vspace*{2pt}\cr
\quad{}\times
e^{-n(g(z)-{l}/{2})\sigma_3},
\vspace*{2pt}\cr
\qquad \mbox{for } z\in\Om_\pm^\nabla \mbox{ and } |\Re z|
\ge\be,
\vspace*{2pt}\cr
e^{-({nl}/{2})\sigma_3} %
\pmatrix{ 1 & 0 \vspace*{2pt}
\cr
0 & -2\pi i }
\mathbf P_n(z) %
\pmatrix{ 1 & 0 \vspace*{2pt}
\cr
0 & -2\pi i } %
^{-1} e^{-n(g(z)-{l}/{2})\sigma_3},\vspace*{2pt}\cr
\qquad \mbox{otherwise},}\hspace*{-20pt}
\end{eqnarray}
where $\sigma_3 =  \bigl(
{{ 1 \atop 0} \enskip{ 0 \atop -1}}
 \bigr)$ is the third Pauli matrix.
This matrix function satisfies the following RHP.
\begin{itemize}
\item
$\mathbf S_n(z)$ is an analytic function of $z$ for $z\in\C\setminus
\Sg
_S$, where $\Sg_S$ consists $\R$, $\Ga_+$, and $\Ga_-$, along with the
four vertical segments $[\pm\be-i\ep, \pm\be+i\ep]$ and $[\pm\al
-i\ep, \pm\al+i\ep]$, oriented as shown in Figure~\ref{Contour_Sg_S}.
\item For $z\in\Sg_S$, the function $\mathbf S_n(z)$ satisfies the
jump conditions
%
\begin{equation}
\mathbf S_{n+}(z)=\mathbf S_{n-}(z) j_S(z) ,
\end{equation}
where
%
\begin{equation}
\label{eq:Sjump} j_S(z) = %
\cases{ %
\pmatrix{ 0
& 1 \vspace*{2pt}
\cr
-1 & 0 },
\qquad  \hspace*{59pt}\mbox{for } z\in(-\be,-\al) \cup(\al,
\be) , \vspace *{2pt}
\cr
\pmatrix{ e^{-i\Om_n} & 0
\vspace*{2pt}
\cr
\bigO\bigl(e^{-n^{1-\delta
}c(z)}\bigr) & e^{i\Om
_n} },
\qquad  \mbox{for } z\in(-\al,\al) , \vspace*{2pt}
\cr
\pmatrix{ 1 & \bigO\bigl(e^{-n^{1-\delta}c(z)}\bigr) \vspace*{2pt}
\cr
\bigO
\bigl(e^{-n^{1-\delta
}c(z)}\bigr) & 1},\vspace*{2pt}\cr
\quad \hspace*{125pt} \mbox{for } z \mbox{ on the
rest of } \Sg_S , } %
\end{equation}
where
%
\begin{equation}\label{eq314}
\Om_n:= \pi(n+1-2\tau) ,
\end{equation}
and $c(z)$ is a nonnegative continuous function on $\Sg_S$ which may
vanish only at the turning points $\pm\al$ and $\pm\be$.
\item As $z\to\infty$,
%
\begin{equation}
\mathbf S_n(z) = I+\frac{\mathbf S_1}{z}+\frac{\mathbf
S_2}{z^2}+\cdots.
\end{equation}
\end{itemize}
Notice that the errors in the off diagonal terms in \eqref{eq:Sjump}
are subexponential, but still smaller than any power of $n$. In the
usual method of steepest descent [\citet{Bleher-Liechty11}], these terms
are exponentially small, but our analysis is slightly different in that
we have taken the contours $\Ga_{\pm}$ to be very close to $\R$. The
reason is that in Proposition~\ref{asymptotics_of_OPs} the asymptotic
formulas are given for $z \in D(\delta,\ep,n)$, which is the region above
$\Ga_+$ and below $\Ga_-$.

\subsubsection{Model RHP} \label{sec:Model_RHP}
The model RHP appears when we drop in the jump matrix $j_S(z)$ the
terms that vanish as $n\to\infty$:
\begin{itemize}
\item
$\mathbf M(z)$ is analytic in $\C\setminus[-\be,\be]$.
\item$\mathbf M_{+}(z)=\mathbf M_{-}(z)j_M(z)$ for $z\in[-\be,\be
]$, where
%
\begin{equation}
\label{m1} j_M(z)= %
\cases{ %
\displaystyle\pmatrix{ 0
& 1 \vspace*{2pt}
\cr
-1 & 0}, %
&\quad $z\in(-\be,-\al) \cup(\al,\be) $,
\vspace*{2pt}
\cr
\displaystyle\pmatrix{ e^{-i\Om_n}, & \quad 0 \vspace*{2pt}
\cr
0 &
e^{i\Om_n} }, %
&\quad $z\in(-\al,\al) $.} %
\end{equation}
\item As $z \to\infty$,
%
\begin{equation}
\label{m2} \mathbf M(z)\sim I+\frac{\mathbf M_1}{z}+\frac{\mathbf
M_2}{z^2}+
\cdots.
\end{equation}
\end{itemize}
The solution to this RHP is described in terms of Jacobi theta
functions, and is presented in \citet{Bleher-Liechty10}, Section~8.

Consider the function $u(z)$ defined in \eqref{m5}.
This function is analytic for $z \in\C\setminus[-\be, \be]$. On that
interval it satisfies certain jump conditions
[see \citet{Bleher-Liechty10},
Section~8].
We will use the Jacobi theta functions $\th_j(z)$, $(j=3,4)$ with
elliptic nome $q$ given by \eqref{eq:def_elliptic_nome_intro}. The
solution is slightly different for $n$ odd and $n$ even. Using the
notation $\epsilon(n)$ introduced in \eqref{eq:def_hsgn}, we can
write the
solution in the following uniform way:
%
\begin{eqnarray}
\label{m16}&& \mathbf M(z) =\mathbf F(\infty)^{-1} %
\left(\matrix{ \displaystyle\frac{\gamma(z)+\gamma^{-1}(z)}{2}\frac{{\th}_3( u(z)
-\pi/4
- \pi(\tau-\epsilon(n)))}{{\th}_3( u(z) -\pi/4)} \vspace*{2pt}\cr
\displaystyle\frac{\gamma(z)-\gamma
^{-1}(z)}{2i}\frac{{\th}_3( u(z)
+ \pi/4 -\pi(\tau-\epsilon(n)))}{{\th}_3( u(z) +\pi/4)}}\right.
\nonumber
\\[-8pt]
\\[-8pt]
\nonumber
&&\hspace*{85pt}\left.\matrix{
 \displaystyle\frac{\gamma
(z)-\gamma
^{-1}(z)}{-2i}
\frac{{\th}_3( u(z) + \pi/4 +\pi(\tau-\epsilon
(n)))}{{\th
}_3( u(z) +\pi/4)} \vspace*{2pt}
\cr
\displaystyle\frac
{\gamma(z)+\gamma
^{-1}(z)}{2} \frac{{\th}_3( u(z) -\pi/4+\pi(\tau-\epsilon
(n)))}{{\th}_3(
u(z) -\pi/4)} } \right)
,
\end{eqnarray}
where
%
\begin{equation}
\label{m17} \mathbf F(\infty)= %
\pmatrix{ \displaystyle\frac{{\th}_3(\pi(\tau-\epsilon(n)))}{{\th}_3(0)} & 0
\vspace*{2pt}
\cr
0 & \displaystyle\frac{{\th}_3(\pi(\tau-\epsilon(n)))}{{\th}_3(0)} } %
.
\end{equation}
The entries of the matrix
%
\begin{equation}
\label{eq:conjugation_of_M} %
\pmatrix{ 1 & 0 \vspace*{2pt}
\cr
0 & -2\pi i }
^{-1} \mathbf M(z) %
\pmatrix{ 1 & 0
\vspace*{2pt}
\cr
0 & -2\pi i } %
,
\end{equation}
are listed in \eqref{eq:M_11_formula_intro}--\eqref{eq:M_22_formula_intro}.
Notice that the ratios of theta functions in \eqref{m16} and \eqref
{m17} become trivial when $\tau=\epsilon(n)$. The coefficient
$\mathbf
M_1$ in the expansion of $\mathbf M(z)$ at $z= \infty$ is
%
{\fontsize{10.7}{12.7}{\selectfont
\begin{equation}\qquad
\label{m20} \mathbf M_1 = %
\pmatrix{
\displaystyle\frac{\pi\be{\th}_3'(\pi(\tau-\epsilon(n)))}{4\tilde
{\mathbf K} {\th}_3(\pi(\tau-\epsilon(n)))} & \displaystyle- \frac{\be-\al}{2i} \frac
{{\th}_3(0) {\th}_4(\pi(\tau-\epsilon(n)))}{{\th}_4(0) {\th
}_3(\pi
(\tau
-\epsilon(n)))} \vspace*{2pt}
\cr
\displaystyle\frac{\be-\al}{2i} \frac{{\th
}_3(0) {\th}_4(\pi(\tau
-\epsilon(n)))}{{\th}_4(0) {\th}_3(\pi(\tau-\epsilon(n)))} &\displaystyle -\frac
{\pi\be{\th
}_3'(\pi(\tau-\epsilon(n)))}{4\tilde{\mathbf K} {\th}_3(\pi(\tau
-\epsilon(n)))} }\hspace*{-30pt} %
,
\end{equation}}}
\hspace*{-3pt}and the $(21)$-entry of the coefficient $\mathbf M_2$ is
%
\begin{equation}
\label{m21} [\mathbf M_2 ]_{21}= \frac{\pi\be(\be-\al){\th}_3(0){\th}_4'(\pi(\tau-\epsilon
(n)))}{8 i {\th
}_3(\pi(\tau-\epsilon(n))) {\th}_4(0) \tilde{\mathbf K}}.
\end{equation}

Notice that according the RHP for $\mathbf M(z)$, $\det\mathbf
M(z)$ is entire. Since $\det\mathbf M(\infty)=1$, it follows from
Liouville's theorem that $\det\mathbf M(z) \equiv1$.

\subsubsection{The local solution at \texorpdfstring{$\pm\al$}{$+-alpha$} and \texorpdfstring{$\pm\be$}{$+-beta$}}

Consider small disks $D(\pm\al, \ep)$ and $D(\pm\be, \ep)$ around
$\pm\al$ and $\pm\be$ with radius $\ep$. We seek a local parametrix
$\mathbf U(z)$ in these disks satisfying:
\begin{itemize}
\item$\mathbf U(z)$ is analytic in $\{D(\pm\al, \ep) \cup D(\pm\be,
\ep)\} \setminus\Sg_S$.
\item For $z\in\{D(\pm\al, \ep) \cup D(\pm\be, \ep)\} \cap\Sg_S$,
$\mathbf U(z)$ satisfies the jump condition $\mathbf U_{+}(z)=\mathbf
U_{-}(z)j_S(z)$.
\item On the boundary of the disks, $\mathbf U(z)$ satisfies
%
\begin{equation}
\label{pm1} \mathbf U(z)=\mathbf M(z) \bigl(I+\bigO\bigl(n^{-1}
\bigr) \bigr) ,\qquad z\in \d D(\pm\al, \ep) \cup\d D(\pm\be, \ep).
\end{equation}
\end{itemize}
The solution is given explicitly in terms of Airy functions, and we do
not describe it here.

\subsubsection{The final transformation of the RHP}\label{tt}

We now consider the contour $\Sigma_X$, which consists of the circles
$\partial D(\pm\be, \ep)$ and $\partial D(\pm\al, \ep)$, all oriented
counterclockwise, together with the parts of
$\Sigma_S \setminus\{[-\be, \al] \cup[\al,\be]\}$
which lie outside of the disks $D(\pm\be, \ep)$, $D(\pm\al, \ep)$.
Let
%
\begin{eqnarray}
\label{tt1}&& \mathbf X_n(z)\hspace*{-20pt}
\nonumber
\\[-8pt]\hspace*{-20pt}
\\[-8pt]
\nonumber
&&\qquad= %
\cases{ \mathbf
S_n(z) \mathbf M(z)^{-1}, &\quad $\mbox{for } z \mbox{ outside
the disks } D(\pm\be, \ep), D(\pm\al, \ep)$, \vspace*{2pt}
\cr
\mathbf
S_n(z) \mathbf U(z)^{-1}, & \quad $\mbox{for } z \mbox{ inside
the disks } D(\pm\be, \ep), D(\pm\al, \ep)$. }\hspace*{-2pt} %
\end{eqnarray}
Then $\mathbf X_n(z)$ satisfies a RHP with jumps on the contour $\Sg_X$
which are uniformly close to the identity matrix, and $\mathbf
X_n(\infty)=I$. The solution to this RHP is given explicitly in terms
of a Neumann series. Due to the fact that the contours $\Ga_{\pm}$ and
the real line are very close to one another (at a distance of the order
$n^{-\delta}$), we find that $\mathbf X_n(z)$ satisfies
%
\begin{equation}
\label{tt19} \mathbf X_n(z) \sim I + \bigO \biggl(
\frac{1}{n^{1-\delta
}(|z|+1)} \biggr) \qquad\mbox{as } n \to\infty,
\end{equation}
uniformly for $z\in\C\setminus\Sigma_X$, which is a weaker error
than the $\bigO(n^{-1})$ error in \citet{Bleher-Liechty11}.

\subsection{Proofs of Propositions \texorpdfstring{\protect\ref{asymptotics_of_OPs}}{3.6},
\texorpdfstring{\protect\ref{prop:asymptotics_of_OPs_outside}}{3.9} and \texorpdfstring{\protect\ref{asymptotics_normalizing_constants}}{3.10}}

We can invert the explicit transformations of the IP in different
regions of the complex plane using \eqref{tt1} and \eqref
{eq:explicit_formula_S}. The asymptotic formula \eqref{tt19} then gives
asymptotic formulas for $\mathbf P_n(z)$. Considering $z$ in the region
$D(\delta,\ep,n)$ proves Proposition~\ref{asymptotics_of_OPs}. Considering
$z\in E(\ep)$, and taking $\delta=0$ proves Proposition~\ref
{prop:asymptotics_of_OPs_outside}. For Proposition~\ref
{asymptotics_normalizing_constants}, we can invert the explicit
transformations with $\delta=0$, and Proposition~\ref
{asymptotics_normalizing_constants} then follows from \eqref{IP4},
\eqref{IP5}, and the expansions of $\mathbf M(z)$ at $z=\infty$ given
in \eqref{m20} and~\eqref{m21}.

\begin{appendix}\label{app}

\section{Construction of steepest-descent contours \texorpdfstring{$\tilde{\Gamma}$}{$tilde{Gamma}$} and
\texorpdfstring{$\tilde{\Sigma}$}{$tilde{Sigma}$}} \label{sec:critical_pt}

In this appendix, we show that the first and second derivatives of
$I(z)$, defined in \eqref{eq:defn_I_tildeI}, vanish at $z = 0$, and
construct two contours $\tilde{\Gamma}$ and $\tilde{\Sigma}$ lying in
the region $\overline{\compC_{+}} = \{ z \in\compC\mid\Im z \geq0
\}
$ and passing through $0$, such that $\tilde{\Sigma}$ is from $e^{0}
\cdot\infty$ to $e^{\pi i} \cdot\infty$ and $\tilde{\Gamma}$ is from
$M$ to $-M$ where $M > \beta$. We require that $\Re I(z)$ attains its
unique global maximum on $\tilde{\Sigma}$ at $0$, and attains its
unique global minimum on $\tilde{\Gamma}$ at~$0$. Since $\Re I(z)$ is
symmetric about the imaginary axis, we only need to construct $\tilde
{\Gamma} \cap D$ and $\tilde{\Sigma} \cap D$ where
%
\begin{equation}
D = \{ z \in\compC\mid\mbox{$\Re z \geq0$ and $\Im z \geq0$} \}
\end{equation}
and construct the other parts of them by reflection.

To simplify the notation, we take a change of variable
%
\begin{equation}
\label{eq:change_of_variable_u_z} u = \frac{z}{\alpha}.
\end{equation}
Then we have that, by \eqref{eq:defn_T_c:elliptic}, \eqref
{eq:defn_I_tildeI} and \eqref{eq:requirement_of_g'}
[\citet{Byrd-Friedman71},
140.01, page
33]
%
\begin{eqnarray}
\label{eq:intro_of_Z} I'(z) = -2\K \biggl( Z(u) - \biggl( 1 -
\frac{\E}{\K} \biggr) u \biggr)
\nonumber
\\[-8pt]
\\[-8pt]
\eqntext{\mbox{where }\displaystyle Z(u) = Z(u; k) = E(u; k) -
\frac{\E}{\K} F(u; k).}
\end{eqnarray}

\begin{rmk}
Here, the arguments of $Z(u; k)$, the Jacobi Zeta function, are
different from those in \citet{Byrd-Friedman71} such that our $u$ is
equal to $\sin\beta$ for the $\beta$ in $Z(\beta, k)$ in \citet{Byrd-Friedman71},
140.02,
03. The Jacobi Zeta function also appears in \eqref
{eq:Jacobi_zeta_in_theta}, where the arguments have same meaning as
those in \citet{Byrd-Friedman71}, 140.01, but the parameter is $\tilde
{k}$ instead of $k$.
\end{rmk}

Below we collect some results about $Z(u)$.

\begin{lem} \label{lem:about_Z}
\textup{(a)} 
$Z(u)$ is analytic in $D$,
%
\begin{eqnarray}
\label{eq:enu:lem:about_Z:a} Z'(0) = 1 - \frac{\E}{\K},
Z''(0) = 0\quad \mbox{and}\quad Z(u) = ku -
\frac{\pi i}{2\K} + \bigO\bigl(u^{-1}\bigr)
\nonumber
\\[-8pt]
\\[-8pt]
\eqntext{ \mbox{as $u \to \infty$
in $D$}.}\vspace*{-12pt}
\end{eqnarray}
\begin{longlist}[(b)]
\item[(b)]
For $x \in[0, 1]$, $Z(x)$ is a real-valued function such that
%
\begin{equation}
\label{eq:enu:lem:about_Z:b} Z(0) = Z(1) = 0\quad \mbox{and}\quad Z''(x) <
0 \qquad\mbox{for all $u \in(0, 1)$}.
\end{equation}
\item[(c)]
For $x \in[1, k^{-1}]$, $Z(x)$ is a pure imaginary-valued function
such that
%
\begin{eqnarray}
\label{eq:enu:lem:about_Z:c} \Im Z(1) = 0,\qquad
\Im Z \biggl(\frac{1}{k} \biggr) = -
\frac{\pi}{2\K} \quad\mbox{and}\quad \frac{d}{dx} \Im Z(x) < 0
\nonumber
\\[-8pt]
\\[-8pt]
\eqntext{\mbox{for all $x
\displaystyle\in \biggl(1, \frac{1}{k} \biggr)$}.}
\end{eqnarray}
\item[(d)]
For $x \in[k^{-1}, \infty)$, $Z(x) + \pi i/(2\K)$ is a real-valued
function such that
%
\begin{eqnarray}
\label{eq:enu:lem:about_Z:d}
Z \biggl(\frac{1}{k} \biggr) + \frac{\pi i}{2\K}& =& 0\quad
\mbox{and} \quad\frac{d}{dx} \biggl( Z(x) + \frac{\pi i}{2\K} \biggr) > 0,
\nonumber
\\[-8pt]
\\[-8pt]
\nonumber
\frac{d^2}{dx^2} \biggl( Z(x) + \frac{\pi i}{2\K} \biggr)&< &0 \qquad\mbox {for all
$\displaystyle x \in \biggl(\frac{1}{k}, \infty \biggr)$}.
\end{eqnarray}
\item[(e)]
For $y \in[0, \infty)$, $Z(iy)$ is a pure imaginary-valued function
such that
%
\begin{eqnarray}
\label{eq:enu:lem:about_Z:e} Z(0) = 0 \quad\mbox{and}\quad \frac{d}{dy}
 \Im Z(iy) > 0,\qquad
\frac
{d^2}{dy^2} \Im Z(iy) > 0
\nonumber
\\[-8pt]
\\[-8pt]
\eqntext{\mbox{for all $y \in(0, \infty)$}.}
\end{eqnarray}
\end{longlist}
\end{lem}

\begin{pf}
The linear term in the asymptotics in part (a) of
Lemma~\ref{lem:about_Z} is a direct consequence of the explicit formula
of $Z(u)$ in $D$,
%
\begin{equation}
\label{eq:int_formula_of_Z} Z(u) = \int^u_0
\frac{ (1 - {\E}/{\K} ) - k^2
s^2}{\sqrt
{(1 - s^2)(1 - k^2 s^2)}},
\end{equation}
which is given by \eqref{eq:intro_of_Z} and \eqref{eq:defn_of_K}. In
the integrand of \eqref{eq:int_formula_of_Z} the sign of the square
root is chosen as $\sqrt{(1 - s^2)(1 - k^2 s^2)} \sim1$ as $s$
approaches $0$ from the region $D$. To compute the constant term, it
suffices to compute the asymptotics of $Z(iy) - iky = E(iy; k) - (\E
/\K
) F(iy; k) -iky$ as $y \to+\infty$. By \citet{Gradshteyn-Ryzhik07},
3.152-1, page
279, $\lim_{y \to\infty} F(iy; k) = i\K'$, and
by the computation in equations \eqref{eq:elliptic_int_with_k} and
\eqref{eq:E-ikt_evaluated}, $\lim_{y \to\infty} E(iy; k) - iky =
i(\K'
- \E')$. Then an application of Legendre's relation \eqref{eq:Legendre}
yields the result.

From the formula \eqref{eq:int_formula_of_Z}, it is clear that: $Z(0)
= 0$; $Z(x)$ is real valued for $x \in[0, 1]$; $\Re Z(x)$ is constant
for $x \in[1, k^{-1}]$; $\Im Z(x)$ is constant for $x \in[k^{-1},
\infty)$; and $Z(iy)$ is pure imaginary for $y \in[0, \infty)$. It is
also straightforward to see that
%
\begin{equation}
Z(1) = E(1; k) - \frac{\E}{\K} F(1; k) = \E- \frac{\E}{\K} \K= 0,
\end{equation}
and with the help of \citet{Byrd-Friedman71}, 111.09, page 11, and the
Legendre's relation \eqref{eq:Legendre},
%
\begin{eqnarray}
Z\bigl(k^{-1}\bigr) &=& E\bigl(k^{-1}; k\bigr) -
\frac{\E}{\K} F\bigl(k^{-1}; k\bigr) = \E+ i\bigl(
\K' - \E '\bigr) - \frac{\E}{\K} \bigl(\K+ i
\K'\bigr)
\nonumber
\\[-8pt]
\\[-8pt]
\nonumber
&=& \frac{i}{\K}\bigl(\K\K' - \K
\E' - \E\K'\bigr) = -\frac{\pi i}{2\K}.
\end{eqnarray}
Thus, all the identities in \eqref{eq:enu:lem:about_Z:b}, \eqref
{eq:enu:lem:about_Z:c}, \eqref{eq:enu:lem:about_Z:d}, \eqref
{eq:enu:lem:about_Z:e} are all proved.

To consider the values of $Z'(0)$ and $Z''(0)$, and the inequalities
of $Z'(u)$ in \eqref{eq:enu:lem:about_Z:c}, \eqref
{eq:enu:lem:about_Z:d}, \eqref{eq:enu:lem:about_Z:e}, we write from
\eqref{eq:int_formula_of_Z}
%
\begin{equation}
\label{eq:formula_of_Z'} Z'(u) = \frac{(1 - {\E}/{\K}) - k^2 u^2}{\sqrt{(1 - u^2)(1 - k^2
u^2)}} \,ds.
\end{equation}
Note that for $u \geq1$,
%
\begin{equation}
\biggl(1 - \frac{\E}{\K} \biggr) - k^2 u^2 =
\frac{1}{\K} \int^1_0
\frac
{k^2(s^2 - u^2)}{\sqrt{(1 - s^2)(1 - k^2 s^2)}} \,ds < 0,
\end{equation}
we obtain the inequality parts of \eqref{eq:enu:lem:about_Z:c}, \eqref
{eq:enu:lem:about_Z:d}, \eqref{eq:enu:lem:about_Z:e} and the evaluation
of $Z'(0)$ and $Z''(0)$ in \eqref{eq:enu:lem:about_Z:a}.

To consider the inequalities of $Z''(u)$ in \eqref
{eq:enu:lem:about_Z:b}, \eqref{eq:enu:lem:about_Z:d} and \eqref
{eq:enu:lem:about_Z:e}, we can write $Z'(u)$ as
%
\begin{equation}
\label{eq:decreasing_of_Z'} Z'(u) = \frac{k}{\K} \sqrt{
\frac{1 - u^2}{k^{-2} - u^2}} \int^1_0
\frac
{{(s^2 - u^2)}/{(1 - u^2)}}{\sqrt{(1 - s^2)(1 - k^2 s^2)}} \,ds,
\end{equation}
for $u \in(0, 1)$ and $u \in(k^{-1}, \infty)$, where in either case
the square root is taken positive value. We observe that $Z'(u)$ is a
decreasing function on $(0, 1)$ since $(1 - u^2)/(k^{-2} - u^2)$ and
$(s^2 - u^2)/(1 - u^2)$ are both increasing, while $Z'(u)$ is also a
decreasing function on $(k^{-1}, \infty)$ by exactly the same reason.
Similarly, writing
%
\begin{equation}\label{equiv2}
\frac{d}{dy} \Im Z(iy) = \frac{k}{\K} \sqrt{1 + y^2}
{k^{-2} + y^2} \int^1_0
\frac{{(s^2 + y^2)}/{(1 + y^2)}}{\sqrt{(1 + s^2)(1 + k^2 s^2)}} \,ds,
\end{equation}
we observe that $\frac{d}{dy} \Im Z(iy)$ is increasing for all $y \in
(0, \infty)$.
This proves the inequality of $Z''(u)$ in \eqref{eq:enu:lem:about_Z:b},
\eqref{eq:enu:lem:about_Z:d} and \eqref{eq:enu:lem:about_Z:e}.
\end{pf}

\begin{lem} \label{lem:triple_zero_I}
The function $I'(z)$ has only one zero $z = 0$ in the region $D$ that
is a third-order zero, and $I^{(4)}(0) > 0$.
\end{lem}

\begin{pf}
From \eqref{eq:enu:lem:about_Z:a}, it is clear that $u = 0$ is a zero
of $Z(u) - (1 - \E/\F)u$ with order at least $3$, and then by \eqref
{eq:intro_of_Z} the same holds for $I'(z)$. On the other hand,
$I^{(4)}(z) = \tilde{g}^{(4)}(z)$, and the explicit computation \eqref
{eq:4th_derivative_of_g} of $\tilde{g}^{(4)}(0)$ shows that $I^{(4)}(0)
> 0$. Below we show that the function $Z(u) - (1 - \E/\F)u$ has only
one zero $u = 0$ in $D$, and complete the proof.

We note that by the results in Lemma~\ref{lem:about_Z}, $Z(u) - (1 -
\E
/\F)u$ has no zero in either $\{ z = x \mid x > 0 \}$ or $\{ z = iy
\mid y > 0 \}$, and it does not vanish as $u \to\infty$. So to prove
that $Z(u) - (1 - \E/\F)u$ has no other zero in $D$, we define a region
(see Figure \ref{region_D_R_1})
%
\begin{equation}
D_R(1) = \bigl\{ u \in D \mid\vert u \vert\leq R \bigr\} \setminus\bigl\{ u
\in \compC\mid\mbox{$\Re u < 1$ and $\Im u < R^{-1}$} \bigr\},
\end{equation}
where $R$ is a positive number, and need only to show that for however
large $R$, $Z(u) - (1 - \E/\K)u$ has no zero in the interior of $D_R(1)$.

By the results in Lemma~\ref{lem:about_Z}, we have that if $R$ is
large enough, then $Z$ is a homeomorphic mapping on $\partial D_R(1)$.
Then by a basis result for univalent functions, $Z$ maps the interior
of $D_R(1)$ into the region enclosed by $Z(\partial D_R(1))$ that does
not contain $0$, see Figure \ref{fig:Z_partial_D_1}. Then by a continuity argument, if $Z(u) - (1 - \E/\K
)u$ has a zero in the interior of $D_R(1)$, there must be a $t \in(0,
1 - \E/K)$ such that $Z(u) - tu$ has a zero on $\partial D_R(1)$, but
by the results in Lemma~\ref{lem:about_Z}, for all such~$t$, $Z(u) -
tu$ does not vanish on $\partial D_R(1)$ given that $R$ is large
enough. Thus, we show that $Z(u) - (1 - \E/\K)u$ has no zero other than
$0$ in $D$ by contradiction.
\end{pf}
%

\begin{figure}

\includegraphics{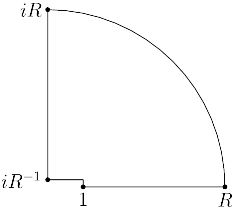}

\caption{The region $D_R(1)$.}
\label{region_D_R_1}
\end{figure}

\begin{figure}[b]

\includegraphics{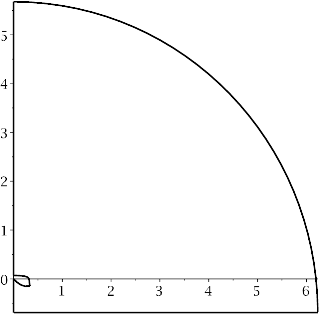}

\caption{The shape of $Z(\partial D_{R}(1))$ with $k = 0.9$ and $R = 7$.}
\label{fig:Z_partial_D_1}
\end{figure}

\begin{figure}

\includegraphics{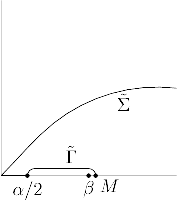}

\caption{Schematic graphs of $\tilde{\Sigma}$ and $\tilde{\Gamma}$
in $D$.}
\label{fig:Sigma_Gamma_cusp_new}
\end{figure}

Now we construct $\tilde{\Gamma}$. By \eqref{eq:enu:lem:about_Z:a} and
\eqref{eq:enu:lem:about_Z:b}, we know that $\Re(Z(u) - (1 - \E/\K)u)$
is decreasing on $[0, 1]$. By Lemma~\ref{lem:about_Z}(c) and that $(1 - \E/\K) > 0$, we also have that
$\Re
(Z(u) - (1 - \E/\K)u)$ is decreasing on $[1, k^{-1}]$. By the relation
\eqref{eq:intro_of_Z}, $\Re I(z)$ is decreasing on the interval $[0,
\beta]$. Thus, it suffices to define $\tilde{\Gamma} \cap D$ as the
interval $[0, \beta]$ if we allow $M = \beta$. Practically, for the
convenience of the asymptotic analysis in Section~\ref{sec:kernel_cusp}, we let $M$ be slightly bigger than $\beta$ and
define $\tilde{\Gamma} \cap D$ by a deformation of the interval $[0,
M]$ such that $[0, \alpha/2]$ is part of $\tilde{\Gamma} \cap D$ and
$(\alpha/2, M)$ is lifted above slightly; see Figure~\ref{fig:Sigma_Gamma_cusp_new}.

In the construction of $\tilde{\Sigma} \cap D$ and $\tilde{\Gamma}
\cap
D$, we use techniques in planar dynamical systems. Regarding $\Re I(z)$
as a function defined on the Cartesian plane whose coordinates are $\Re
z$ and $\Im z$, we define the gradient field
%
\begin{eqnarray}
\nabla\Re I(z) = \biggl( \frac{\partial}{\partial x} \Re I(z), \frac
{\partial}{\partial y} \Re
I(z) \biggr)
\nonumber
\\[-8pt]
\\[-8pt]
 \eqntext{\mbox{where } x = \Re z, y = \Im z.}
\end{eqnarray}

By Lemma~\ref{lem:about_Z}(a), (e), we have that for $y > 0$, $Z(iy) - (1 - \E/\K)iy$
is pure imaginary, and its imaginary part is positive. Then by \eqref
{eq:intro_of_Z}, we conclude that $\{ iy \mid y > 0 \}$ is an upward
flow curve of $\nabla\Re I(z)$. By Lemma~\ref{lem:about_Z}(d) and the relation \eqref{eq:intro_of_Z}, we have
that for all $x > M > \beta$, $\Im(Z(x) - (1 - \E/\K)x > 0$ and then
the gradient field $\nabla\Re I(z)$ is transversal to the interval
$[M, \infty)$ and is outward of~$D$.

Since by Lemma~\ref{lem:triple_zero_I}, $0$ is a triple zero of $I'(z)$
and $I^{(4)}(0) > 0$, there is a flow curve that ends at $0$, with
direction $\pi/4$, and we denote it as $\gamma$. Since the gradient
field $\nabla\Re I(z)$ has no singular point by Lemma~\ref
{lem:triple_zero_I}, this flow curve is from either the boundary of $D$
or $\infty$. As we showed above, the left edge of $D$ is a flow curve
and at the interval $[M, \infty)$, as part of $\partial D$, the
gradient field is outward, so the $\gamma$ cannot be from the left edge
of $D$ or $[M, \infty)$. If $\gamma$ is from $(0, M)$, then it crosses
$\tilde{\Gamma}$ at a point other than $0$, denoted by $z_0$. But by
the definition of $\tilde{\Gamma}$, $\Re I(z_0) > \Re I(0)$. On the
other hand, by the property of the flow curve $\gamma$, $\Re I(z_0) <
\Re I(0)$, and we derive a contradiction. Thus, $\gamma$ cannot be from
$\partial D$, but is from $\infty$. At last by the behavior of $\nabla
\Re I(z)$ given in Lemma~\ref{lem:about_Z}(a), we
verify that it suffices to let $\tilde{\Sigma} \cap D = \gamma$, as
shown in Figure~\ref{fig:Sigma_Gamma_cusp_new}.

\section{Proof of Proposition \texorpdfstring{\protect\ref{4541521548452154}}{1.5}}\label{sec:tacnode_equivalence}
Since $\bolds\Psi(\zeta; s)$ satisfies [\citet{Hastings-McLeod80}]
%
\begin{equation}
 \frac{\d}{\d s} \bolds\Psi(\zeta; s) = %
\pmatrix{
-i\zeta& q(s) \vspace*{2pt}
\cr
q(s) & i\zeta } %
\bolds\Psi(\zeta; s),
\end{equation}
it is easy to derive the identity that for $u, v \in\Sigma_T$,
%
\begin{eqnarray}
\label{equiv3}&& \frac{\d}{\d s} \biggl(\frac{f(u; s)g(v; s)-g(u; s)f(v;
s)}{u-v} \biggr)
\nonumber
\\[-8pt]
\\[-8pt]
\nonumber
&&\qquad=-i
\bigl(f(u; s)g(v; s)+g(u; s)f(v; s)\bigr),
\end{eqnarray}
where $f$ and $g$ are defined by $\bolds\Psi$ by \eqref{tac18}.
Hence, \eqref{eq:essential_tacnode} can be written as
%
\begin{eqnarray}
\label{equiv4} %
\tilde{K}^{\tac}_{\tau_i, \tau_j}(\xi,\eta;
\sigma) &=& \frac
{1}{4\pi
^2} \int_{\Sg_T} \,du \int
_{\Sg_T} \,dv e^{{\tau_i u^2}/{2} -
{\tau_j v^2}/{2}} e^{-i(u\xi-v\eta)}\nonumber\\
&&{}\times \int
_\sigma^\infty \,ds \bigl(f(u; s)g(v; s)+g(u; s)f(v; s)
\bigr)
\nonumber\\
&=& \int^{\infty}_{\sigma} \,ds \biggl[\biggl(
\frac{1}{2\pi} \oint _{\Sigma_T} \,du e^{{\tau_i u^2}/{2} - i\xi u} f(u; s)
\biggr)
\nonumber
\\[-8pt]
\\[-8pt]
\nonumber
&&\hspace*{39pt}{}\times
\biggl( \frac{1}{2\pi} \oint_{\Sigma_T} \,dv e^{-{\tau_i v^2}/{2} +
i\eta
v} g(v; s)
\biggr)
\\
&&\hspace*{39pt}{} + \biggl( \frac{1}{2\pi} \oint_{\Sigma_T} \,du e^{{\tau_i u^2}/{2}
- i\xi u} g(u;
s) \biggr)\nonumber\\
&&\hspace*{39pt}{}\times \biggl( \frac{1}{2\pi} \oint_{\Sigma_T} \,dv e^{-{\tau_i v^2}/{2} +
i\eta v}
f(v; s) \biggr)\biggr].\nonumber %
\end{eqnarray}

In order to relate formula \eqref{equiv4} for the tacnode kernel to the
other formula \eqref{equiv20a} defined by Airy function and related
operators, we consider the expressions for the entries of $\bolds\Psi
(\zeta; s)$ in terms of Airy functions. Introduce the functions in $x$
with parameters $\zeta$ and $s$,
%
\begin{eqnarray}
\label{equiv9} E_+(x) &=& E_+(x; \zeta, s) := e^{i(({4}/{3})
\zeta^3+(s+2x)\zeta)},
\nonumber
\\[-8pt]
\\[-8pt]
\nonumber
 E_-(x) &=& E_-(x;
\zeta, s) := e^{-i(({4}/{3})\zeta^3+(s+2x)\zeta)} = E_+(x; -\zeta, s).
\end{eqnarray}
Then the matrix entries of $\bolds\Psi(\zeta; s)$ are given by the formulas
%
\begin{eqnarray}
\Psi_{11}(\zeta; s) &= &\langle E_-, R_s+
\delta_0\rangle_0 , \qquad \Psi _{21}(\zeta; s)= -
\langle E_-, Q_s\rangle_0 , \label{equiv10}
\\
\Psi_{12}(\zeta; s) &=& -\langle E_+, Q_s
\rangle_0 , \qquad \Psi _{22}(\zeta; s) = \langle E_+,
R_s+\delta_0\rangle_0 , \label{equiv10a}
\end{eqnarray}
where the inner product $\langle\cdot, \cdot\rangle_0$, functions
$R_s$, $Q_s$, and the delta function $\delta_0$ are defined in Section~\ref{subsec:comparison_of_kernel_formulas}. The derivation of \eqref
{equiv10} is essentially given in \citet{Baik-Liechty-Schehr12}. Note
that the functions $\Phi_1(\zeta; s)$ and $\Phi_2(\zeta; s)$ in
\citet{Baik-Liechty-Schehr12},
Proposition~2.1, are the same as the functions
$\Phi^1(\zeta; s)$ and $\Phi^2(\zeta; s)$ in \citet{Claeys-Kuijlaars06},
and the entries $\Psi_{11}(\zeta; s)$ and $\Psi_{21}(\zeta; s)$ are the
same as the functions $\Phi_1(\zeta; s)$ and $\Phi_2(\zeta; s)$ in
\citet
{Claeys-Kuijlaars06}.
Using the relation given in equation (1.19) of \citet{Claeys-Kuijlaars06},
equation \eqref{equiv10} follows from Proposition 2.1 of \citet{Baik-Liechty-Schehr12}.
By the relation
\eqref{eq:symmetry_of_H-ML_solution}, \eqref{equiv10} implies~\eqref
{equiv10a}.

Consider now the integrals
%
\begin{equation}
\label{equiv11} I^{\pm}_{a, b; s}(x) := \frac{1}{2\pi}
\oint_{\Sg_T^{\pm}} e^{a\zeta
^2+ib\zeta} E_{\pm}(x; \zeta, s)\,d\zeta,
\end{equation}
where $\Sg_T^+$ (resp., $\Sg_T^-$) is the connected piece of $\Sg_T$
which lies above (resp., below) the real axis.
A simple change of variables gives that
%
\begin{eqnarray}
\label{equiv13a} %
I^+_{a, b; s}(x) &= &\frac{1}{2\pi}
\oint_{\Sigma_T^+} e^{i
({4}/{3}) \zeta^3 + a\zeta^2 + i(s+2x+b)\zeta}\,d\zeta
\nonumber
\\[-8pt]
\\[-8pt]
\nonumber
&=& -2^{-{2}/{3}} e^{- { a^3}/{24} - {a(s + 2x +
b)}/{4}} \Ai \biggl( \frac{s + 2x + b}{2^{2/3}} +
\frac{a^2}{2^{8/3}} \biggr), %
\end{eqnarray}
where we have used the integral representation of the Airy function
%
\begin{equation}
\Ai(x) = \frac{-1}{2\pi} \oint_{\Sigma_T^+} e^{({i}/{3}) \zeta
^3 +
ix \zeta} \,d\zeta.
\end{equation}
Similarly,
%
\begin{equation}
\label{equiv14} I^-_{a, b; s}(x) = 2^{-{2}/{3}}
e^{-{ a^3}/{24} - {a(s +
2x - b)}/{4}} \Ai
\biggl( \frac{s + 2x - b}{2^{2/3}} + \frac
{a^2}{2^{8/3}} \biggr).
\end{equation}

We can now write the expression \eqref{equiv4} in terms of Airy
functions and operators only, since the functions $f$ and $g$ there are
expressed by entries of $\bolds\Psi$. Notice that in the expressions
\eqref{equiv10} and \eqref{equiv10a} for the entries of the matrix
$\bolds\Psi$, the dependence on $\zeta$ lies solely in the left-hand
side of the inner products. Thus, by changing\vspace*{1pt} the order of integration,
we can write \eqref{equiv4} in terms of the integrals $I^{\pm}_{a, b;
s}(x)$. Indeed we have
%
\begin{eqnarray}
\label{equiv15} %
 \tilde{K}_{\tau_i, \tau_j}^{\tac}(
\xi, \eta; \sigma) &=& \int^\infty_\sigma \,ds \bigl[
\bigl( \bigl\langle I^-_{{\tau_i}/{2}, -\xi; s}, R_s + \delta_0
\bigr\rangle_0 + \bigl\langle I^+_{{\tau_i}/{2}, -\xi; s}, Q_s
\bigr\rangle_0 \bigr)\nonumber
\\
&&\hspace*{36pt}{} \times \bigl( -\bigl\langle I^+_{-{\tau_j}/{2}, \eta; s}, R_s + \delta
_0 \bigr\rangle_0 - \bigl\langle I^-_{-{\tau_j}/{2}, \eta; s},
Q_s \bigr\rangle_0 \bigr)
\nonumber
\\[-8pt]
\\[-8pt]
\nonumber
&&\hspace*{36pt}{} + \bigl( -\bigl\langle I^+_{{\tau_i}/{2}, -\xi; s}, R_s +
\delta_0 \bigr\rangle_0 - \bigl\langle
I^-_{{\tau_i}/{2}, -\xi; s}, Q_s \bigr\rangle_0 \bigr)
\\
&&\hspace*{36pt}{} \times \bigl( \bigl\langle I^-_{-{\tau_j}/{2}, \eta; s}, R_s + \delta
_0 \bigr\rangle_0 + \bigl\langle I^+_{-{\tau_j}/{2},\eta; s},
Q_s \bigr\rangle _0 \bigr) \bigr].\nonumber
\end{eqnarray}
Notice that in terms of the function $b_{\tau, z, \sigma}$ defined in
\eqref{equiv16},
%
\begin{eqnarray}
\label{equiv17} I^+_{\tau, z; s}(x) &=& -2^{-{2}/{3}} \pi
b_{2^{-4/3} \tau, 2^{-2/3}
z, s}(x),
\nonumber
\\[-8pt]
\\[-8pt]
\nonumber
 I^-_{\tau,z; s}(x)& =& 2^{-{2}/{3}} \pi
b_{2^{-4/3}
\tau, -2^{-2/3} z, s}(x).
\end{eqnarray}
Hence, formula \eqref{equiv15} becomes
%
\begin{eqnarray}
\label{equiv18} %
&& \tilde{K}_{\tau_i, \tau_j}^{\tac}(
\xi, \eta; \sigma) \nonumber\\
&&\qquad= 2^{-{4}/{3}} \int_\sigma^\infty
\,ds \bigl[\bigl(\langle b_{2^{-7/3} \tau_i, 2^{-2/3}
\xi,
s}, R_s +
\delta_0 \rangle_0 - \langle b_{2^{-7/3} \tau_i, -2^{-2/3}
\xi,
s},
Q_s \rangle_0 \bigr)
\nonumber\\
&&\hspace*{63pt}\qquad\quad{} \times \bigl(\langle b_{-2^{-7/3} \tau_j, 2^{-2/3} \eta, s}, R_s +
\delta_0 \rangle_0 \nonumber\\
&&\hspace*{115pt}{}- \langle b_{-2^{-7/3} \tau_j, -2^{-2/3} \eta,
s},
Q_s \rangle_0 \bigr)
\\
&&\hspace*{63pt}\qquad\quad{}+ \bigl(\langle b_{2^{-7/3} \tau_i, -2^{-2/3} \xi, s}, R_s + \delta_0
\rangle_0 - \langle b_{2^{-7/3} \tau_i, 2^{-2/3} \xi, s}, Q_s
\rangle_0 \bigr)
\nonumber\\
&&\hspace*{63pt}\qquad\quad{} \times \bigl(\langle b_{-2^{-7/3} \tau_j, -2^{-2/3} \eta, s}, R_s +
\delta_0 \rangle_0 \nonumber\\
&&\hspace*{226pt}{}- \langle b_{-2^{-7/3} \tau_j, 2^{-2/3} \eta,
s},
Q_s \rangle_0 \bigr) \bigr],\nonumber
\end{eqnarray}
which is, in terms of the function $\hat{p}_1(z; s, \tau)$ defined in
\eqref{equiv19},
%
\begin{eqnarray}
\label{equiv18b} %
 &&\tilde{K}_{\tau_i, \tau_j}^{\tac}(
\xi, \eta; \sigma) \nonumber\\
&&\qquad= 2^{-
{4}/{3}} \int^\infty_\sigma
\,ds \bigl( \hat{p}_1\bigl(-2^{-2/3}\xi; s, 2^{-7/3}
\tau_i\bigr) \hat{p}_1\bigl(-2^{-2/3}\eta; s,
-2^{-7/3}\tau_j\bigr)
\nonumber
\\[-8pt]
\\[-8pt]
\nonumber
&&\hspace*{68pt}\qquad\quad{} +\hat{p}_1\bigl(2^{-2/3}\xi; s, 2^{-7/3}
\tau_i\bigr) \hat {p}_1\bigl(2^{-2/3}\eta; s,
-2^{-7/3}\tau_j\bigr) \bigr)
\\
&&\qquad= 2^{-{2}/{3}} \tilde{\mathcal{L}}_{\tac}\bigl(2^{-2/3}\xi,
2^{-2/3}\eta; \sigma,2^{-7/3}\tau_i,
2^{-7/3}\tau_j\bigr),\nonumber
\end{eqnarray}
where $\tilde{\mathcal{L}}_{\tac}$ is defined in \eqref{equiv20}.
It is
simple to see that by \eqref{eq:nonessential_Pearcey}
%
\begin{equation}
\label{equiv18a} 2^{-{2}/{3}} \bigl( \phi_{2\cdot2^{-7/3}\tau_i, 2\cdot
2^{-7/3}\tau
_j}
\bigl(2^{-2/3}\xi, 2^{-2/3}\eta\bigr) \bigr) =
\phi_{\tau_i, \tau
_j}(\xi,\eta).
\end{equation}
Combining \eqref{equiv18} and \eqref{equiv18a} gives \eqref{equiv21},
and Proposition~\ref{4541521548452154} is proved.
\end{appendix}

\section*{Acknowledgments}
Part of this work was carried out during the workshop \textit{Random
matrices and applications} at the University of Michigan in June 2013.
The authors thank the organizers of that conference, Jinho Baik and Raj
Rao Nadakuditi, and also thank Jinho Baik and Arno Kuijlaars for
helpful comments. The authors also thank Gr\'{e}gory Schehr for
discussions and for pointing out the recent preprint \citet
{Dupic-PerezCastillo13}. D. Wang thanks Peter Forrester for discussion
on the relation to Yang--Mills theory, as well as Wexiao Shen and
Rongfeng Sun.

%

\printaddresses
\end{document}